\documentclass{article}
\usepackage{graphicx} 
\usepackage[utf8]{inputenc}
\usepackage{graphicx} 
\usepackage[margin = 1in]{geometry}
\usepackage{amsfonts}
\usepackage{mathtools}
\usepackage{amsmath}
\usepackage{amssymb}
\usepackage{amsthm}
\usepackage{authblk}
\usepackage{natbib}
\usepackage{babel}
\usepackage{graphicx}
\usepackage{graphics}
\usepackage[dvipsnames]{xcolor}
\usepackage{hyperref}
\usepackage{pgfplots}
\usepackage{amsthm}

\usepackage{amsmath,amsthm,amssymb}
\usepackage{mathrsfs}
\usepackage{amsmath} 
\usepackage{authblk}
\pgfplotsset{width = 5cm, compat = 1.17}
\numberwithin{equation}{section} 
\usepackage{enumitem} 
\usepackage{soul}

\renewcommand{\P}{\mathbb{P}}
\renewcommand{\Pr}{\mathrm{P}}

\newcommand{\R}{\mathbb{R}}
\newcommand{\E}{\mathbb{E}}

\newcommand{\var}{\text{Var}}

\newcommand{\bm}{\boldsymbol}

\newcommand{\Xvec}{\mbox{\bf X}}

\newcommand{\tvec}{\mbox{\bf t}}

\newcommand{\sigmat}{\mbox{\boldmath $\Sigma$}}
\newcommand{\muvec}{\mbox{\boldmath $\mu$}}

\DeclareMathOperator*{\argmin}{arg\,min}

\theoremstyle{definition}
\newtheorem{thm}{Theorem}[section]
\newtheorem{exa}{Example}[section]
\newtheorem{rem}{Remark}

\newtheorem{lemma}{Lemma}[section]
\newtheorem{cor}{Corollary}

\title{Exact distribution-free tests of spherical symmetry\\ applicable to high dimensional data}
\author{Bilol Banerjee and Anil K. Ghosh}
\affil{Theoretical Statistics and Mathematics Unit,\\ Indian Statistical Institute, Kolkata\\
	Emails: \{banerjeebilol, anilkghosh\}@gmail.com}
\date{}

\usepackage{tikz,siunitx}
\usetikzlibrary{shapes.geometric,shapes.symbols}

\newcommand{\tikzsymbol}[2][circle]{\tikz[baseline=-0.5ex]\node[inner
	sep=2pt,shape=#1,draw,#2]{};}%

\definecolor{applegreen}{rgb}{0.55,0.71,0.0}

\begin{document}
	
	\maketitle
	
	\vspace{-0.15in}
	\begin{abstract}
		We develop some graph-based tests for spherical symmetry of a multivariate distribution using a method based on data augmentation. 
		These tests are constructed using a new notion of signs and ranks that are computed along a path obtained by optimizing an objective function based on pairwise dissimilarities among the observations in the augmented data set. The resulting tests based on these signs and ranks have the exact distribution-free property, and irrespective of the dimension of the data, the null distributions of the test statistics remain the same. These tests can be conveniently used for high-dimensional data, even when the dimension is much larger than the sample size. Under appropriate regularity conditions, we prove the consistency of these tests in high dimensional asymptotic regime, where the dimension grows to infinity while the sample size may or may not grow with the dimension. We also propose a generalization of our methods to take care of the situations, where the center of symmetry is not specified by the null hypothesis. Several simulated data sets and a real data set are analyzed to demonstrate the utility of the proposed tests.
		
		\vspace{0.05in}
	\end{abstract}
	\maketitle
	
	\section{Introduction}
	\label{Sec: intro}
	
	A $d$-dimensional ($d>1$) random vector ${\bf X}$ is said to follow a spherically symmetric distribution (about the origin) if ${\bf X}$ and ${\bf H}{\bf X}$ have the same distribution (i.e., ${\bf X} \stackrel{D}{=}{\bf H}{\bf X}$)
	for all $d \times d$ orthogonal matrix ${\bf H}$ \citep[see][]{fang1990book}. It is an important class of distributions in statistics. Motivated by the spherical symmetry or elliptic symmetry (i.e., spherical symmetry after standardization) of the underlying distributions, several statistical methods have been developed. Robust measures of location and scale \citep[see, e.g.,][]{van2009minimum}, tests for multivariate location \citep[see, e.g.,][]{randles1989distribution, chaudhuri1993sign}, Stein estimation \citep[see, e.g.,][]{fourdrinier2018shrinkage}, classification \citep[see, e.g.,][]{ghosh2005maximum,li2012dd} and clustering \citep[see, e.g.,][]{jornsten2004clustering} are some examples of its widespread applications. So, testing the sphericity of the underlying distribution is an important statistical problem, and several methods have been proposed for it. 
	\cite{smith1977} developed a test for bivariate data that uses the $L_2$ distance between the empirical distribution function and the theoretical distribution function under spherical symmetry. \cite{baringhaus1991} derived an alternative formulation of this test and generalized it for multivariate data. \cite{henze2014} proposed a test based on the fact that if ${\bf X}$ follows a spherically symmetric distribution, the characteristic function $\varphi_{X}(\tvec)$ remains constant over all $\tvec$ lying on the surface of a sphere with the center at the origin. \cite{albisetti2020} proposed a Kolmogorov-Smirnov-type test using the fact that ${\bf X}$ is spherically symmetric if and only if $\E\{{\bf u}^{\top} {\bf X}\mid {\bf v}^{\top}{\bf X}\} = 0$ for all ${\bf u}$ and ${\bf v}$ with ${\bf u}^{\top}{\bf v}=0$. \cite{kolchinskii1998} proposed a test using the difference between the empirical spatial rank function and the theoretical spatial rank function under spherical symmetry, where the unknown components of these theoretical ranks were estimated from the data. These tests have consistency in the classical asymptotic regime, i.e., for any fixed alternative, the powers of these tests converge to unity as the sample size diverges to infinity. However, they become computationally prohibitive even for moderately high-dimensional data.
	
	\cite{fang1993} constructed a necessary test that computes 
	the Wilcoxon-Mann-Whitney statistic for several pairs of orthogonal projection directions and considers the minimum over all such projection pairs as the test statistic. \cite{diks1999} proposed a Monte Carlo test based on the difference between the density of the data generating distribution and its rotations averaged over all possible orthogonal matrices computed using Haar measure. \cite{liang2008} proposed some tests based on the fact that under spherical symmetry,  ${\bf X}/\|{\bf X}\|$ is uniformly distributed on ${\mathcal S}^{d-1}$, the surface of the unit sphere in $\R^d$, but they did not consider the independence between and $\|{\bf X}\|$ and  ${\bf X}/\|{\bf X}\|$, which is another important property of a spherical distribution. So, this test often fails to differentiate between spherically symmetric and angular symmetric distributions. Recently, \cite{huang2023multivariate} proposed tests for different notions of symmetry including spherical symmetry using optimal transport, but these tests often have poor performance 
	in high dimension. 
	
	Some tests are also available for high-dimensional data. Assuming ellipticity of the underlying distribution, \cite{zou2014} and \cite{feng2017high} proposed some tests of sphericity based on the multivariate sign function. \cite{ding2020some} also assumed elliptic symmetry of the underlying distribution and proposed a test based on the ratio of traces of different powers of the sample covariance matrix. The authors showed the consistency of their test against a fairly general class of alternatives, even when the dimension grows with the sample size at an appropriate rate. However, a study on their behavior for non-elliptic distributions is missing from the literature. Moreover, nothing is known about their behavior in the high-dimension, low-sample-size (HDLSS) asymptotic regime, where the sample size is assumed to be fixed while the dimension grows to infinity. Recently, \cite{banerjee2024consistent} proposed a test of spherical symmetry based on data augmentation, which is 
	applicable to high dimensional data. But that test does not have the distribution-free property. Use of the resampling method for calibration makes it computationally demanding. 
	Also, its behavior in the HDLSS asymptotic regime has not been studied.
	
	To overcome these limitations, in this article, we propose some graph-based tests constructed using a new notion of signs and ranks. Our tests are distribution-free, and the null distributions of the tests statistics remain the same in all dimensions. The main contributions of this article are summarized as follows:
	
	\begin{itemize}
		\item In Section 2, we provide a characterization of spherical symmetry of a distribution in terms of pairwise distances between the original and simulated random variables. We observe that the difference in pairwise distances mainly comes from differences in inner products. Motivated by this finding, we introduce a new notion of signs and ranks computed along a covering path obtained by minimizing a suitable function of these inner products. These signs and ranks have the same null distributions in all dimensions. We also investigate their behavior for non-spherical distributions (Theorem \ref{thm:sign-rank-dist}). Using these signs and ranks, we construct some exact distribution-free tests of spherical symmetry.
		
		
		\item In Section 3, we study the high dimensional behavior of our tests. We first investigate their behavior in the HDLSS asymptotic regime, where the dimension is assumed to grow to infinity while the sample size remains fixed. We show that in this setup, under the usual assumptions on moments and weak dependence \citep[see, e.g.,][]{jung2009pca}, any test of spherical symmetry based on inner products
		has a fundamental problem (see Theorem \ref{thm:HDLSS-difficulty}). We establish the consistency of our tests against a suitable class of alternatives that includes spiked covariance models (see Theorem \ref{thm:HDLSS-consistency} and Corollary \ref{cor:spiked-HDLSS}). We also investigate
		the performance of our tests in situations 
		where the sample size grows with the dimension and derived some relaxed conditions for consistency of our tests (see Theorem \ref{thm:limit-null-distribution} and \ref{thm:limit-alt-convergence}).
		
		\item We know that in the case of a spherical distribution, all diagonal elements of the scatter matrix are equal and all off-diagonal elements are zero. In Section 4, we observe that our proposed tests have excellent performance for alternatives having significant correlations among the variances. But when they are uncorrelated and difference is mainly in the scales of the variables, their performance may not be satisfactory. To take of care of this problem, we use a new objective function to construct the path and compute the signs and ranks along it. The resulting tests can detect changes in scales (see Theorem \ref{thm:modified-rank-tests}) but usually  yields poor performance if the signal against the null hypothesis mainly comes from the off-diagonal elements. So, we propose to combine the two approaches to develop some modified tests that perform well in both situations (see Theorem \ref{thm:HDLSS-boosted-tests}). 
		
		\item In Section 5, we analyze some simulated data sets and a real data set on `Earthquakes' to compare the empirical performance of our tests against state-of-the-art methods. We observe that our modified tests significantly outperforms its competitors for a wide variety of alternatives. 
		
		\item In Section 6, we propose a generalization of our method to test for spherical symmetry when the null hypothesis does not specify the center of of the distribution. Centering the observations based on an empirical estimate often leads to an inflated Type I error if the dimension is much larger than the sample size. As a remedy, we propose a generalization based on a key lemma (Lemma \ref{lemma:key-lemma}). The resulting tests control the Type I error at the desired level and also retains the desirable theoretical properties under a fairly general assumption. 
	\end{itemize}
	
	Section 7 contains some concluding remarks and ends with a discussion on some possible directions for future research. All proofs and some auxiliary theoretical results are given in the Appendix.
	
	\section{Sign test and runs test for spherical symmetry}
	\label{sec:definition}
	As we have mentioned before, a random vector ${\bf X}$ is said to follow a spherically symmetric distribution if its distribution remains invariant  over all orthogonal transformations (i,e,, $\bm X \stackrel{D}{=}{\bm H}{\bm X}$ for any orthogonal matrix ${\bm H}$).
	The following lemma provides an alternative characterization
	of spherical symmetry.
	\begin{lemma}  \label{lemma:ss-variant}
		A $d$-dimensional random vector ${\bf X}$ has a spherically symmetric distribution if and only if ${\bf X}$ and $\|{\bf X}\|{\bf U}$ has the same distribution (i.e., ${\bf  X} \stackrel{D}{=} \| {\bf X}\| {\bf U}$), where $ {\bf U}$ is independent of ${\bf X}$ and it is uniformly distributed over $\mathcal{S}^{d-1}$, the surface of the $d$-dimensional unit sphere.
	\end{lemma}
	We refer the reader to page 31 of \cite{fang1990book} for a proof of this lemma. In this article, we refer to the random vector $\bm X^\prime = \|\bm X\|\bm U$, (where $\bm U\sim \textnormal{Unif}(S^{d-1})$ and  independent of $\bm X$) as the spherically symmetric variant of $\bm X$. Motivated by the distributional equality of $\bm X$ and $\bm X^\prime$ under the null hypothesis of spherical symmetry, \cite{banerjee2024consistent} 
	provided a connection between a test for spherical symmetry and a energy statistic based two-sample test for distributional equality. In this context, we have the following result, which is a direct analog of Theorem 2 of \cite{maa1996}.
	
	\begin{lemma}\label{lemma:distance-characterization}
		Let $\bm X_1,\bm X_2$ be two independent realizations of $\bm X \sim P$, and $\bm X_1^\prime,\bm X_2^\prime$ be their spherically symmetric variants. Assume that $P$ is absolutely continuous with square integrable density function $p(.)$. Then for any function $h(\cdot, \cdot)$ that satisfies (a) $h(\bm x,\bm y)=0$ if and only if $\bm x = \bm y$, (b) $h(\bm x+\bm a, \bm y+\bm a) = h(\bm x,\bm y)$ for all $\bm a\in \R^d$ and (c) the class of sets $S_t=\{\bm x|h(0,\bm x)\leq t\}$ regularly shrinks towards $0$ as $t\downarrow 0$, we have
		\begin{align*}
			h(\bm X_1,\bm X_2)\stackrel{D}{=}h(\bm X_1,\bm X_2^\prime)\stackrel{D}{=}h(\bm X_1^\prime,\bm X_2^\prime)~~~~\text{ if and only if ~~~~$\mathrm P$ is spherically symmetric.}
		\end{align*}
	\end{lemma}
	We refer the reader to Chapter 7 of \cite{wheeden1977measure} for regularly shrinking sets and their consequences in the Lebesgue point theorem. In the literature, $h(\bm x,\bm y)$ is popularly chosen as the $\ell_2$ distance between $\bm x$ and $\bm y$. Hence we have
	\begin{align}
		\|\bm X_1-\bm X_2\|\stackrel{D}{=}\|\bm X_1-\bm X_2^\prime\|\stackrel{D}{=}\|\bm X_1^\prime-\bm X_2^\prime\|~~~~ \text{if and only if ~~~~$\mathrm P$ is spherically symmetric.}		
	\label{eq:equivalence}
	\end{align}


	\begin{figure}[t]
	\centering
	\setlength{\tabcolsep}{-2pt}
	\begin{tabular}{cc}
		(a) $\log$(Euclidean distance) for $d=10$ & (b) $\log$(squared inner-product) for $d=10$ \\
		\vspace{-0.175in}\\
		\includegraphics[height = 2.0in, width =0.45\linewidth]{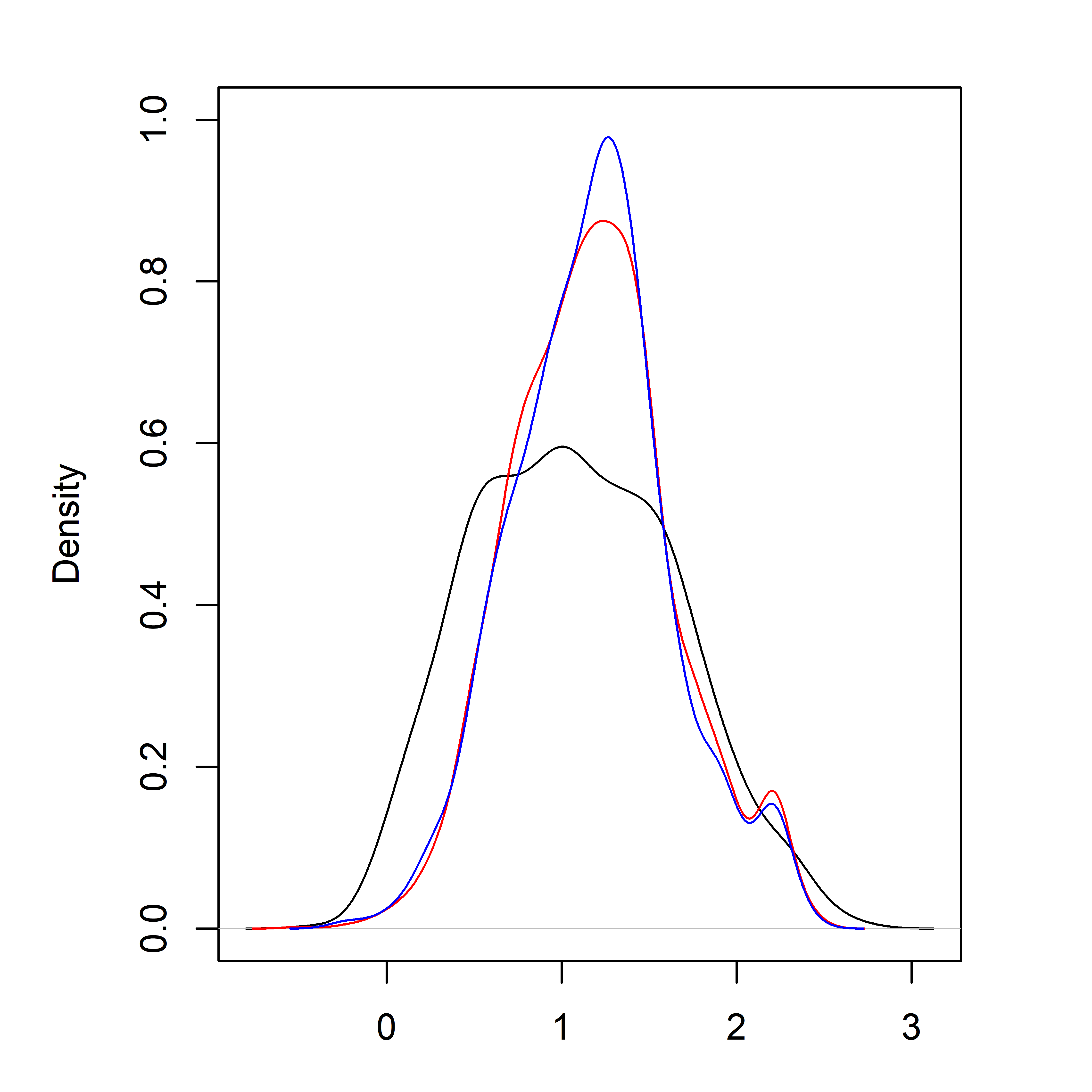} & \hspace{-0.2in}\includegraphics[height = 2.0in, width =0.45\linewidth]{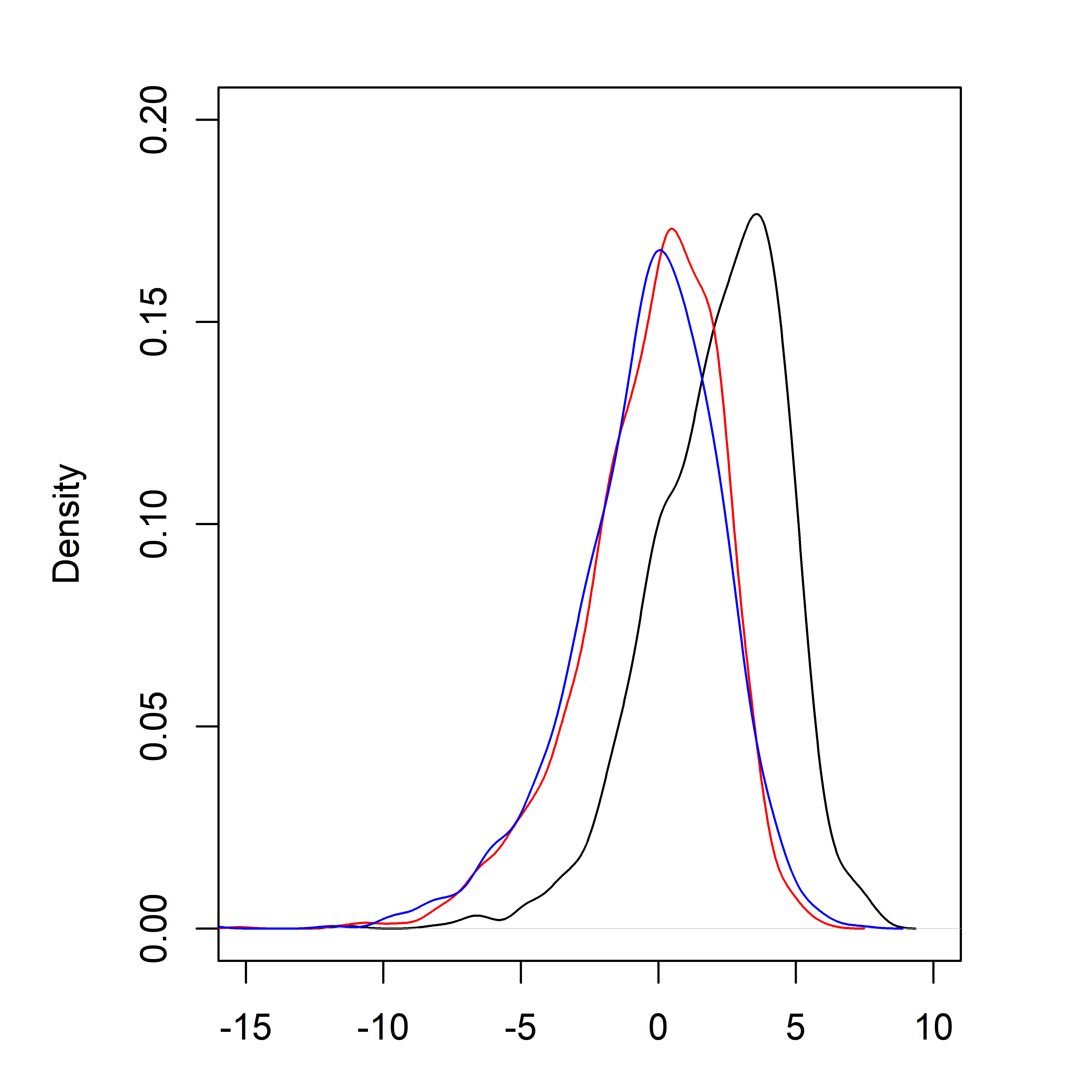}\\
		(c) $\log$(Euclidean distance) for $d=100$ & (d) $\log$(squared inner-product) for $d=100$\\
		\includegraphics[height = 2.0in, width =0.45\linewidth]{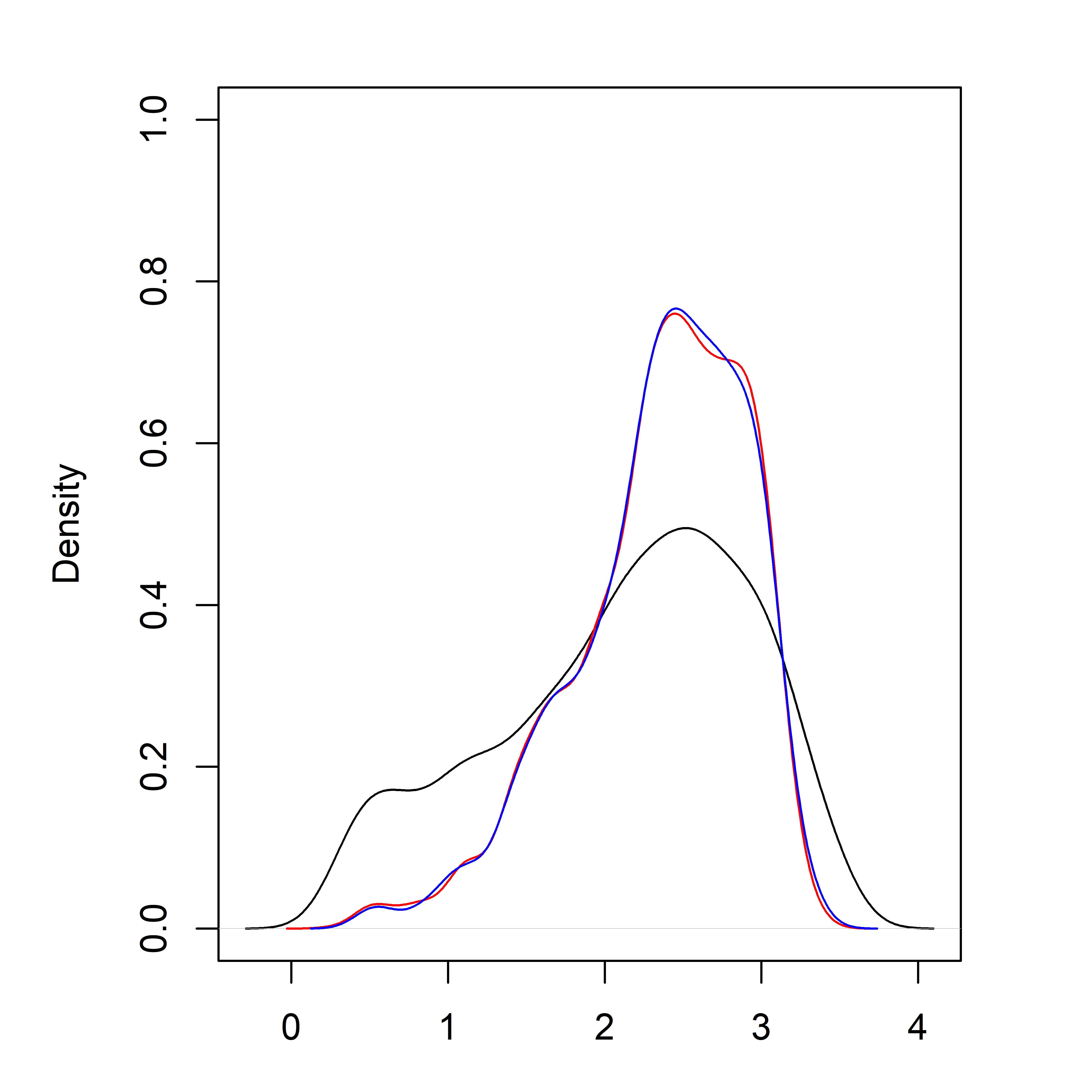} & \hspace{-0.2in}\includegraphics[height = 2.0in, width =0.45\linewidth]{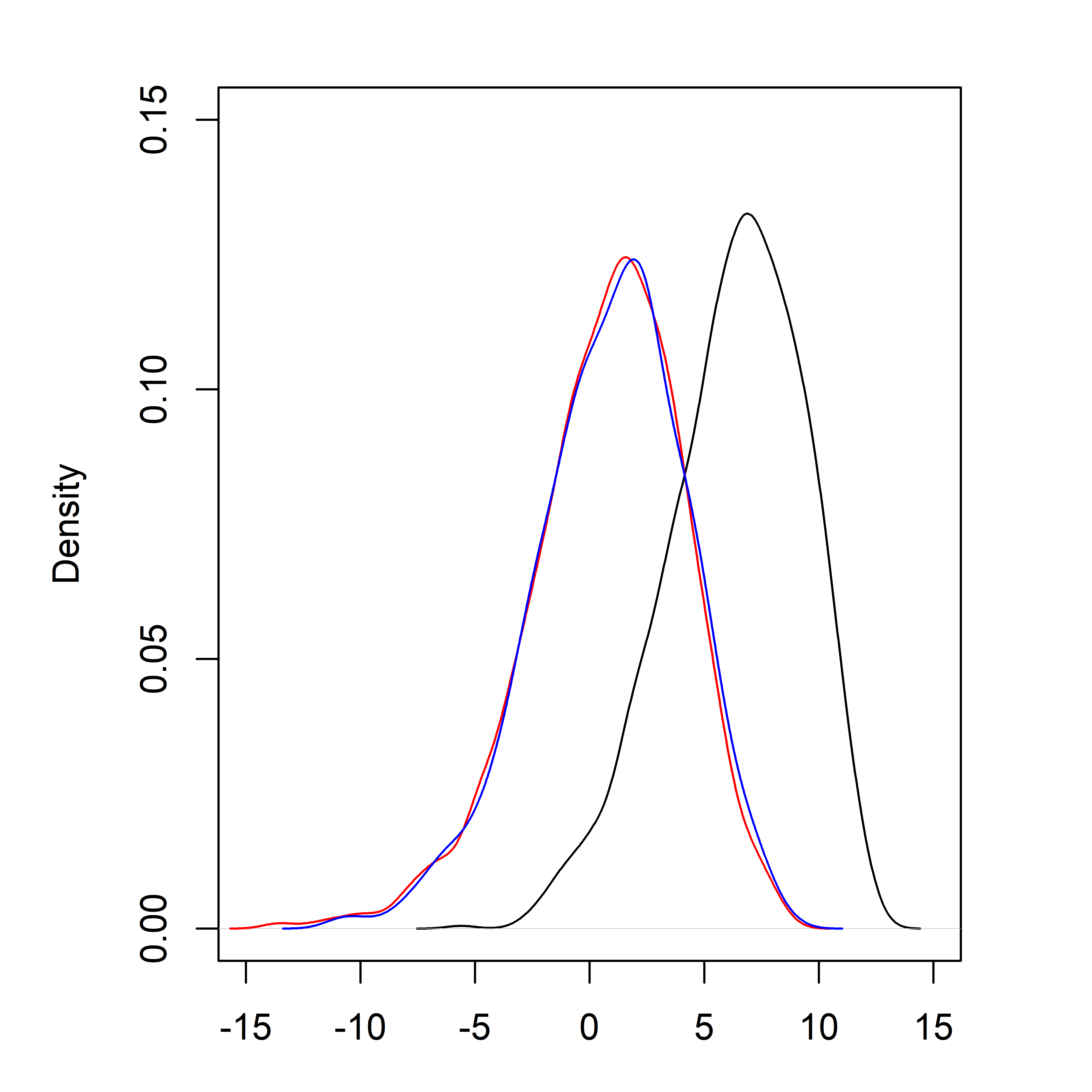}\\
	\end{tabular}
\vspace{-0.1in}
\caption[width=\linewidth]{Densities of the logarithm of $\|\bm X_1-\bm X_2\|$ (black), $\|\bm X_1-\bm X_2^\prime\|$ (blue) and $\|\bm X_1^\prime-\bm X_2^\prime\|$ (red) and those of the logarithm of $(\bm X_1^\top\bm X_2)^2$ (black), $(\bm X_1^\top\bm X_2^\prime)^2$ (blue) and $(\bm X_1^{\prime\top}\bm X_2^\prime)^2$ (red) when 50 observations are generated from $d$-dimensional normal distribution with mean ${\bf 0}$ and covariance matrix $\sigmat=((\sigma_{ij}))$ with $\sigma_{ij}=1$ for $i=j$ and $\sigma_{ij}=(d-1)/d$ for $i \neq j$.} 
	\label{fig:density-estimates}
\end{figure}
	
The left column in Figure \ref{fig:density-estimates} gives the density estimates of the logarithm of these three types of Euclidean distances when 50 observations are generated from the $d$-dimensional normal distribution (with $d=10$ and $d=100$) having mean ${\bf 0}$ and covariance matrix $\sigmat$, which has all diagonal entries $1$ and all off-diagonal entries $(d-1)/d$ (here we take the log transformation of the distances to avoid the restriction on their supports). The difference between these three distributions is evident from this figure. Note that since $\|\Xvec_1\|=\|\Xvec_1^{\prime}\|$ and $\|\Xvec_2\|=\|\Xvec_2^{\prime}\|$, the difference between $\|\bm X_1-\bm X_2\|$, $\|\bm X_1-\bm X_2^\prime\|$ and $\|\bm X_1^\prime-\bm X_2^\prime\|$ mainly comes from the inner products
$\bm X_1^\top\bm X_2$, $\bm X_1^\top\bm X_2^\prime$ and $\bm X_1^{\prime\top}\bm X_2^\prime$. However, under the assumption on $\Pr$ in Lemma \ref{lemma:distance-characterization}, we have $\E[\bm X_1^\top\bm X_2]=\E[\bm X_1^\top\bm X_2^\prime]=\E[\bm X_1^{\prime\top}\bm X_2^\prime] = 0$. But,  Lemma \ref{lem:expectations} (see Appendix B) shows that as $d$ increases, the asymptotic order of $\E(\bm X_1^{\top}\bm X_2)^2$ turns out to be higher than those of $\E(\bm X_1^{\top}\bm X_2^\prime)^2$ and 
$\E({\bm X_1^\prime}^{\top}\bm X_2^\prime)^2$. Figure \ref{fig:density-estimates} also supports that(see the right column in Figure \ref{fig:density-estimates}). Clearly, this difference becomes more prominent if we look at the densities of the logarithm of squares of corresponding inner products. The right column in Figure \ref{fig:density-estimates} shows
that $(\bm X_1^{\top}\bm X_2)^2$ is stochastically larger $(\bm X_1^{\top}\bm X_2^\prime)^2$ and 
$({\bm X_1^\prime}^{\top}\bm X_2^\prime)^2$, and the difference is more significant in the case of $d=100$. This phenomenon gives us the motivation to construct tests based on the squares of the inner products, which are described in the following subsections.

	\subsection{String signs and string ranks}
	Suppose that $\bm X_1,\bm X_2,\ldots,\bm X_n \stackrel{iid}{\sim} \mathrm P$, and $\bm X_i^\prime=\|\bm X_i\| \bm U_i$ is the spherically symmetric variant of $\bm X_i$ ($i=1,2,\ldots,n$) as defined before. Define $\bm Z_i = \bm X_i$ and $\bm Z_{n+i} = \bm X_{i}^\prime$ for $i =1,2,\ldots,n$. To test whether $\mathrm P$ is spherically symmetric, we consider an edge weighted undirected complete graph $\mathcal{K}_{2n}$ on the $2n$ vertices $\bm Z_1,\bm Z_{2},\ldots,\bm Z_{2n}$, where  $\theta(\bm Z_i,\bm Z_j) = \exp\{-(\frac{1}{d}\bm Z_i^\top\bm Z_j)^2\}$ is the weight (cost) associated with the edge joining $\bm Z_i$ and $\bm Z_j$ ($1\le i<j\le 2n$). We consider a path of length $n-1$ that traverses through either $\bm X_i$ or $\bm X_i^\prime$ for each $i=1,2,\ldots,n$, and we call it a covering path. 
	Clearly, there are $2^n n!$ many covering paths. But for every path, there exists another path in the reverse order. If we consider them as the same path, the number of distinct covering paths turns out to be $2^{n-1}n!$. When the observations  come from a continuous distributions, each of these $2^{n-1}n!$ paths have different costs (the cost of a path is defined as the sum of costs of its edges) with probability one. Among these paths, we choose the one with the minimum cost and it is called the shortest covering path 
	$\mathcal{P}$ \citep[see][ for the use of shortest covering path in the context of one-sample location problem]{biswas2015onesample}. Note that finding this path is equivalent to finding
	\begin{align}\label{eq:optimization}
		({\bm S},{\bm \Pi}) = \argmin_{\substack{\bm s\in\{0,1\}^n\\\bm\pi\in \mathcal{S}_n}} \left[\sum_{i=1}^{n-1} \theta(\bm Y_{s_{\pi_i},\pi_i},\bm Y_{s_{\pi_{i+1}},\pi_{i+1}})\right],
	\end{align}
	where ${\mathcal S}_n$ is the set of all permutations of $\{1,2,\ldots,n\}$ and for any ${\bm s}=(s_1,s_2,\ldots,s_n)$, $\bm Y_{s_i,i} = s_i \bm X_i + (1-s_i)\bm X_i^\prime$ (i.e., $\bm Y_i={\bm X}_i$ if $s_i=1$ and $\bm Y_i={\bm X}_i^\prime$ if $s_i=0$) for all $i=1,2,\ldots,n$. Here ${\bm \Pi}$ gives us the arrangement of observation numbers along $\mathcal{P}$. 
	This leads to a new notion of ranks. The position of $\bm X_i$ (or $\bm X_i^\prime$) along  ${\mathcal P}$ is called the rank of $\bm X_i$, and it is denoted by $R_i$ ($i=1,2,\ldots,n$). One can notice that 
	${\bm \Pi}^{-1}$, the inverse permutation of ${\bm \Pi}$, gives the rank vector $\bm R= (R_1,R_2,\ldots,R_n)$. Similarly, ${\bm S}$ can be viewed as a sign vector (instead of $1$ and $-1$, each of its elements takes the values $1$ and $0$), where $S_i$ gives the information whether $\bm X_i$ or $\bm X_i^\prime$ lies on the path ${\mathcal P}$. Since these signs and ranks are computed along the shortest covering path ${\mathcal P}$, which can be viewed as a string joining $n$ observations or their spherical analogs, we shall refer to them as string signs and string ranks, respectively. Note that they satisfy the following properties.
	
	\begin{thm}
		Let $\bm X_1,\bm X_2,\ldots,\bm X_n$ be independent realizations of a random vector $\bm X$  following a continuous distribution $ \mathrm P$. Also, define ${\bm S}$ and $\bm R$ as the vector of string signs and string ranks of $\{\bm X_1,\bm X_2,\ldots,\bm X_n\}$. Then, we have the following results.
		\begin{itemize}
			\item[\textnormal{(a)}] If $\mathrm P$ is spherically symmetric, ${\bm S}\sim\textnormal{Unif}(\{0,1\}^n)$,  $\bm R \sim\textnormal{Unif}({\mathcal S}_n)$ and they are independent.
			
			\item[\textnormal{(b)}] If $\mathrm P$ is not spherically symmetric, $\bm R \sim$  {Unif}$(S_n)$ and for any given $\bm R$ or ${\bm \Pi}=( \pi_1,\pi_2,\ldots,\pi_n)$, we have the following weak dependence structure of $\bm S$.
			\begin{align*}
				{S}_{\pi_1}\mid {S}_{\pi_2},\cdots,{S}_{\pi_n}& \stackrel{D}{=}{S}_{\pi_1}\mid {S}_{\pi_2},\\
				{S}_{\pi_n}\mid {S}_{\pi_1},\cdots,{S}_{\pi_{n-1}} & \stackrel{D}{=}{S}_{\pi_n}\mid {S}_{\pi_{n-1}},\\
				{S}_{\pi_i}\mid {S}_{\pi_1},\cdots,{S}_{\pi_{i-1}},
				{S}_{\pi_{i+1}},\cdots,{S}_{\pi_n} & \stackrel{D}{=}{S}_{\pi_i} \mid {S}_{\pi_{i-1}},{S}_{\pi_{i+1}} \mbox{ for } i=2,\ldots,n-1.
			\end{align*}
		\end{itemize}
		\label{thm:sign-rank-dist}
	\end{thm}
	\subsection{Tests based on sting signs and string ranks}
	From part (a) of  Theorem  \ref{thm:sign-rank-dist}, it is clear that under the null hypothesis of spherical symmetry, the distribution of $({\bm S},~{\bm R})$ 
	matches with the joint null distribution of univariate signs and runs used in one-sample testing problem
	(e.g., sign test or signed rank test). So, any test statistic computed based on these string signs and sting ranks has the distribution-free property and its null distribution matches with that of the corresponding univariate statistic based on usual signs and ranks. For instance, one can consider a test based on the sign statistic $T_S=\sum_{i=1}^{n} S_i$ or the runs statistic $T_R=1+\sum_{i=1}^{n-1} \mathrm{I}\{S_{\pi_i} \neq S_{\pi_{i+1}}\}$, where $\delta$ denotes the the indicator function. Note that the values of $T_S$ and $T_R$ remain the same if the path ${\mathcal P}$ is traversed in the reverse order. For any linear rank statistic of the form $T_{LR} =\sum_{i=1}^{n} S_i a(R_i)=\sum_{i=1}^{n}S_{\pi_i}a(i)$, we have this property if the score function $a(\cdot)$ satisfies $a(i)=a(n-i+1)$ for all $i=1,2,\ldots,n$ (note that taking $a(i)=1 ~\forall i$, we get $T_S$). In the case of signed rank test, we have $a(i)=i$, which does not satisfy this property. So, we do not recommend using this test. 
	
	\begin{figure}[b!]
		\includegraphics[height = 2.5in, width = 3.0in]{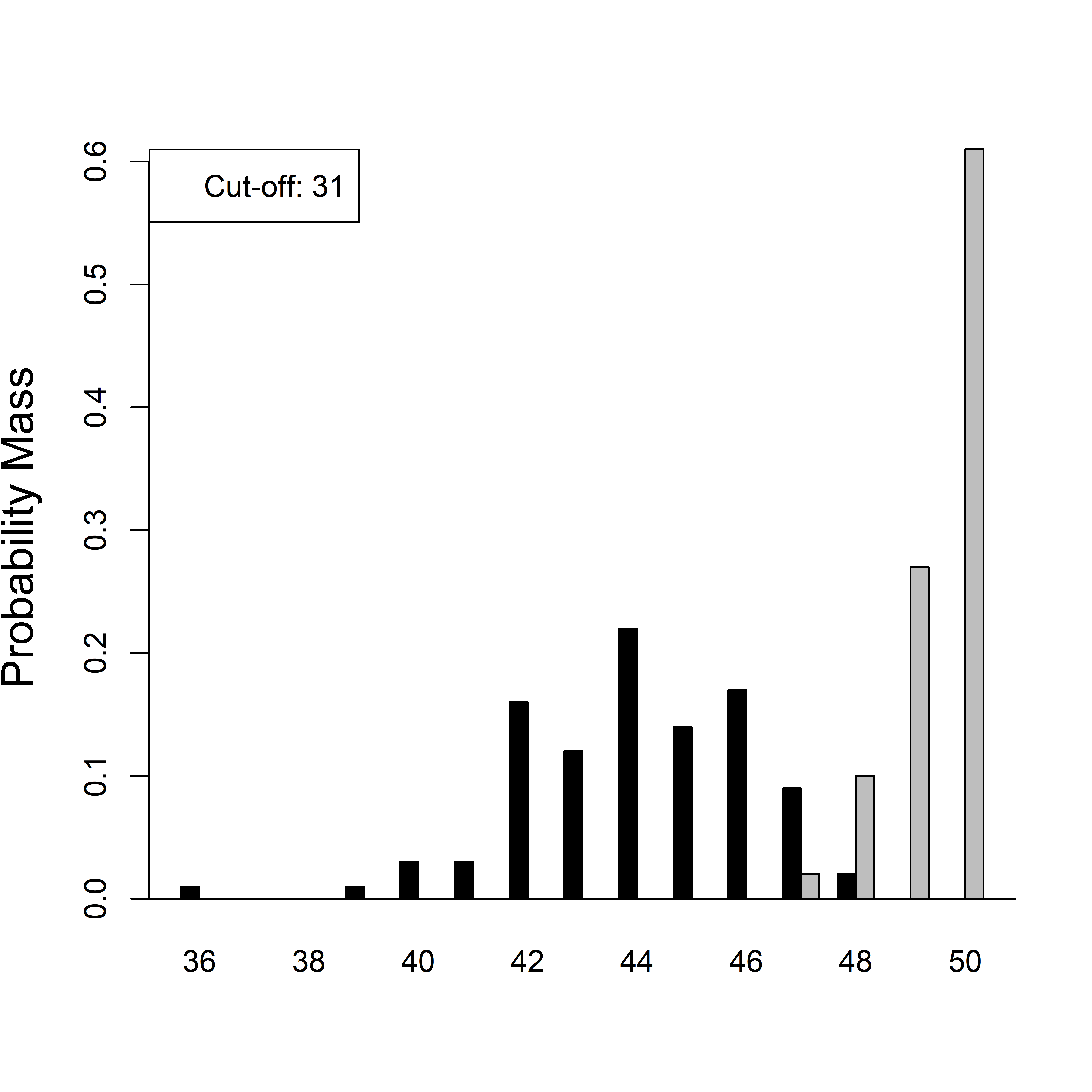}
		\includegraphics[height = 2.5in, width = 3.0in]{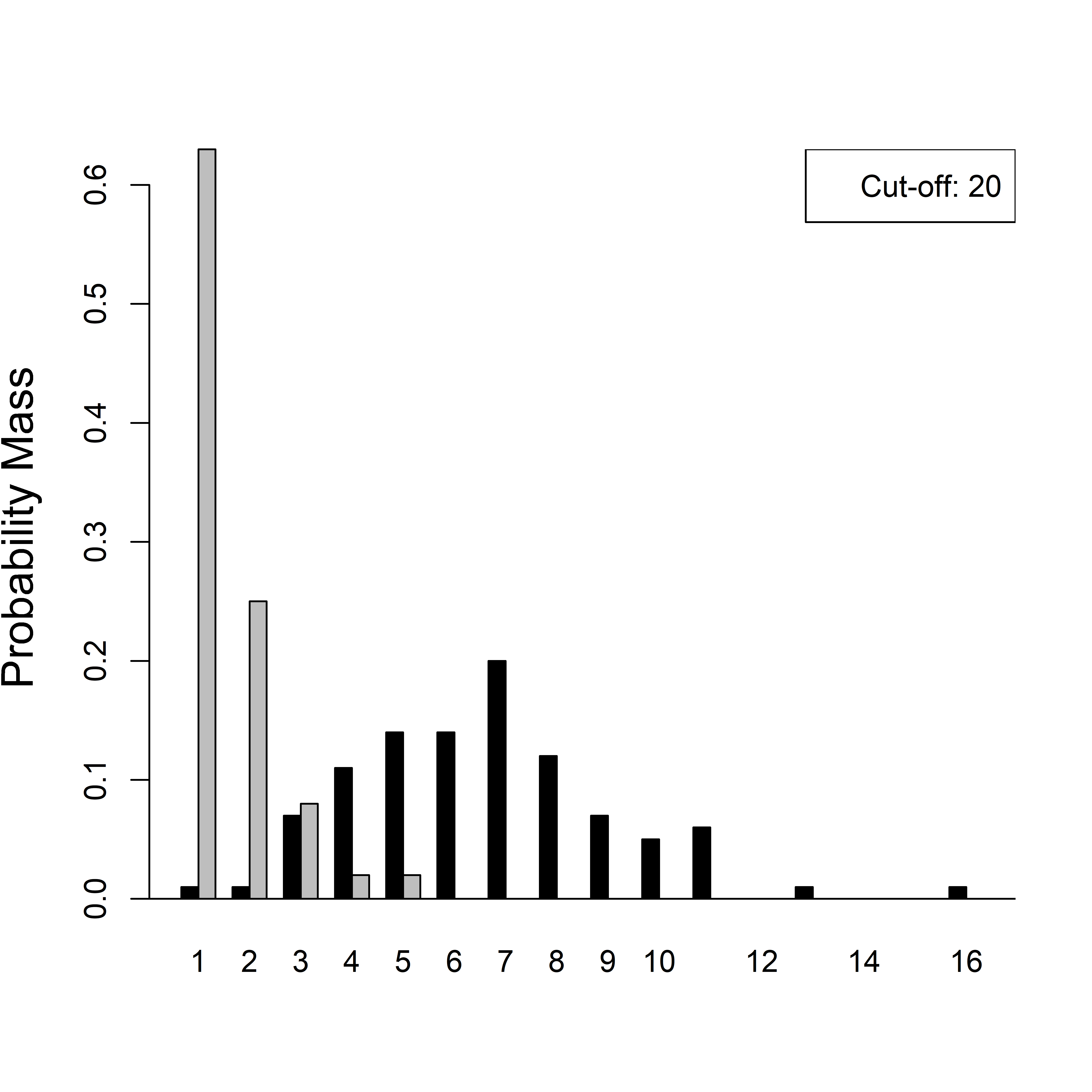}	
		\vspace{-0.25in}
		\caption{Distribution of $T_S$ and $T_R$ over 100 simulations for $d=10$ (black bar) and $d=100$ (grey bar) in the example involving normal distribution considered in Figure \ref{fig:density-estimates}.}
		\label{fig:rose-histograms-3}
	\end{figure}

	Part (b) of Theorem \ref{thm:sign-rank-dist} gives us some idea about the behavior of $\bm S$ and $\bm R$ under the alternative hypothesis (i.e., when $\mathrm P$ is not spherical). From Figure \ref{fig:density-estimates} and also from Lemma \ref{lem:expectations} (see Appendix B), it is quite transparent that when $\mathrm P$ is not spherical, the total cost of any covering path will be small if most of the edges are of the form $(\bm X_i,\bm X_j)$. 
	So, ${\mathcal P}$ is supposed to contain more
	$\bm X_i$s than ${\bm X}_i^\prime$s.
	As a result, the most of the elements of the sign vector $\bm S$ turn out to be $1$ and that leads to a higher value of $T_S$ and a lower value of $T_R$. So, we can reject the null hypothesis accordingly, where the cut-offs can be obtained from the statistical tables available for the corresponding univarite nonparametric tests. 
	To demonstrate this let us recall the example involving 50 observations from $10$-dimensional and $100$-dimensional normal distributions. In each case, we repeated the experiment $1000$ times.
	Figure \ref{fig:rose-histograms-3} shows the bar diagram of the observed values $T_S$ and $T_R$ in these $1000$ cases. We can see that in all these cases, $T_S$ took higher values than the cut-off ($31$) for the sign test at $5\%$ level. Similarly, $T_R$ turned out to be smaller than the corresponding cut-off ($20$) in all cases. Interestingly, both for sign and runs tests, the evidence was stronger for $d=100$. 
		
	\subsection{Algorithm for finding the shortest covering path}
	\label{sec:prim-algo}
	
	For finding the shortest covering path, one needs to consider $2^{n-1}n!$ distinct covering paths and choose the one with the minimum cost. So, unless the sample size is small, finding the shortest covering path $\mathcal{P}$ by complete enumeration becomes computationally infeasible. In fact, the optimization problem in \eqref{eq:optimization} is equivalent to the well known traveling salesman's problem, which is NP-hard \citep[see][]{npcompletebook}. However, following \cite{biswas2015onesample}, we can use a heuristic method based on Prim's algorithm \citep[][]{primalgo} for finding $\mathcal{P}$. Consider the undirected complete  graph ${\mathcal K}_{2n}$ with the cost matrix  $\Theta = (\theta(\bm Z_i,\bm Z_j))_{1\leq i,j\leq 2n}$ where $\bm Z_i = \bm X_i$ and $\bm Z_{n+i} = \bm X_i^\prime$ for $i=1,2,\ldots,n$ as defined before. First we select the pair $(i,j)$ ($i \neq j$ and $|i-j|\not = n$) such that $\theta(\bm Z_i,\bm Z_j)$ is minimum among the cost associated with such edges. We consider the edge $(\bm Z_i,\bm Z_j)$ as a path ${\mathcal P}$ of length $1$ with ${\mathcal E}=\{i,j\}$ as the two end points. We define the sets $A_0 = \{i,j\}$ and
	$A_1=\{k: ~k \neq \ell \mbox{ and } |k-\ell|\neq n \mbox{ for any } \ell \in A_0\}$. At the next step,
	we find $q \in A_1$ and $r \in {\mathcal E}$ such that $\theta(\bm Z_q, \bm Z_r) =\min_{~k \in A_1, \ell \in {\mathcal E}} \theta(\bm Z_k, \bm Z_{\ell})$. We
	join the edge $(\bm Z_q, \bm Z_r)$ to ${\mathcal P}$ to get a path ${\mathcal P} \leftarrow {\mathcal P} \cup (\bm Z_q, \bm Z_r)$ of length $2$. The sets 
	of visited nodes $A_0$ and the end points ${\mathcal E}$ of the path ${\mathcal P}$ are updated as $A_0 \leftarrow A_0\cup \{q\}$ and ${\mathcal E}\leftarrow \big({\mathcal E} \cup\{q\}\big) \setminus \{r\}$. The set $A_1$ is updated accordingly. We use this method repeatedly until we get $|A_0|=n$ and 
	$|A_1|=0$. The path ${\mathcal P}$ of length $n-1$ thus obtained is considered as the shortest covering path.
	Clearly, ${\mathcal P}$ contains either $\bm X_i$ or its spherically symmetric variant $\bm X_i^\prime$ for every $i=1,2,\ldots,n$.
	
	We use a toy example with $5$ bivariate observations to demonstrate this algorithm. Figure \ref{fig:scp-construction}(a) shows these $5$ observations (blue dots) 
	and their spherically symmetric variants (red dots). First, we join $\bm Z_3=\bm X_3$ and $\bm Z_2=\bm X_2$, the pair having the minimum cost (note that we do not consider the pairs ($\bm Z_i=\bm X_i, \bm Z_{5+i}=\bm X_i^\prime$) for $i=1,2,\ldots,5$) and consequently remove their spherical variants $\bm Z_8 = \bm X_3^\prime$ and $\bm Z_7 = \bm X_2^\prime$ from future considerations. So, we have $A_0=\{2,3\}$, $A_1=\{1,4,5,6,9,10\}$ and 
	${\mathcal E}=\{3,2\}$. At the next step, we join $\bm Z_4=\bm X_4$ and
	$\bm Z_3=\bm X_3$. This leads to $A_0=\{2,3,4\}$, $A_1=\{1,5,6,10\}$ and ${\mathcal E}=\{4,2\}$ (see \ref{fig:scp-construction}(b)). Next, we join $\bm Z_5=\bm X_5$ and $\bm Z_4=\bm X_4$ to get $A_0=\{2,3,4,5\}$, $A_1=\{1,6\}$ and ${\mathcal E}=\{5,2\}$ (see \ref{fig:scp-construction}(c)). Finally, $\bm Z_2=\bm X_2$ and $\bm Z_6=\bm X_1^\prime$ are joined  (see \ref{fig:scp-construction}(d)). As a result, we get $\bm X_5-\bm X_4-\bm X_3-\bm X_2-\bm X_1^\prime$
	(or $\bm X_1^\prime-\bm X_2-\bm X_3-\bm X_4-\bm X_5$) as the shortest covering path ${\mathcal P}$. So, we have string signs $S_1 = 0, S_2 = S_3 = S_4 = S_5 = 1$ and string ranks $R_1=5, R_2=4, R_3=3,R_4=2, R_5=1$ (or $R_1=1, R_2=2, R_3=3,R_4=4, R_5=5$ if traversed in the reverse order). Therefore, the sign statistic (number of blue dots on ${\mathcal P}$) $T_S$ turns out to be 4 and the runs statistic (the number of runs in the sequence  red and blue dots) $T_R$ turns out to be $2$. 
	
	\begin{figure}[t]
		\centering
		\setlength{\tabcolsep}{-2pt}
		\begin{tabular}{cccc}
			(a) Step 1& (b) Step 2 \\
			\includegraphics[width=3in,height=2.1in]{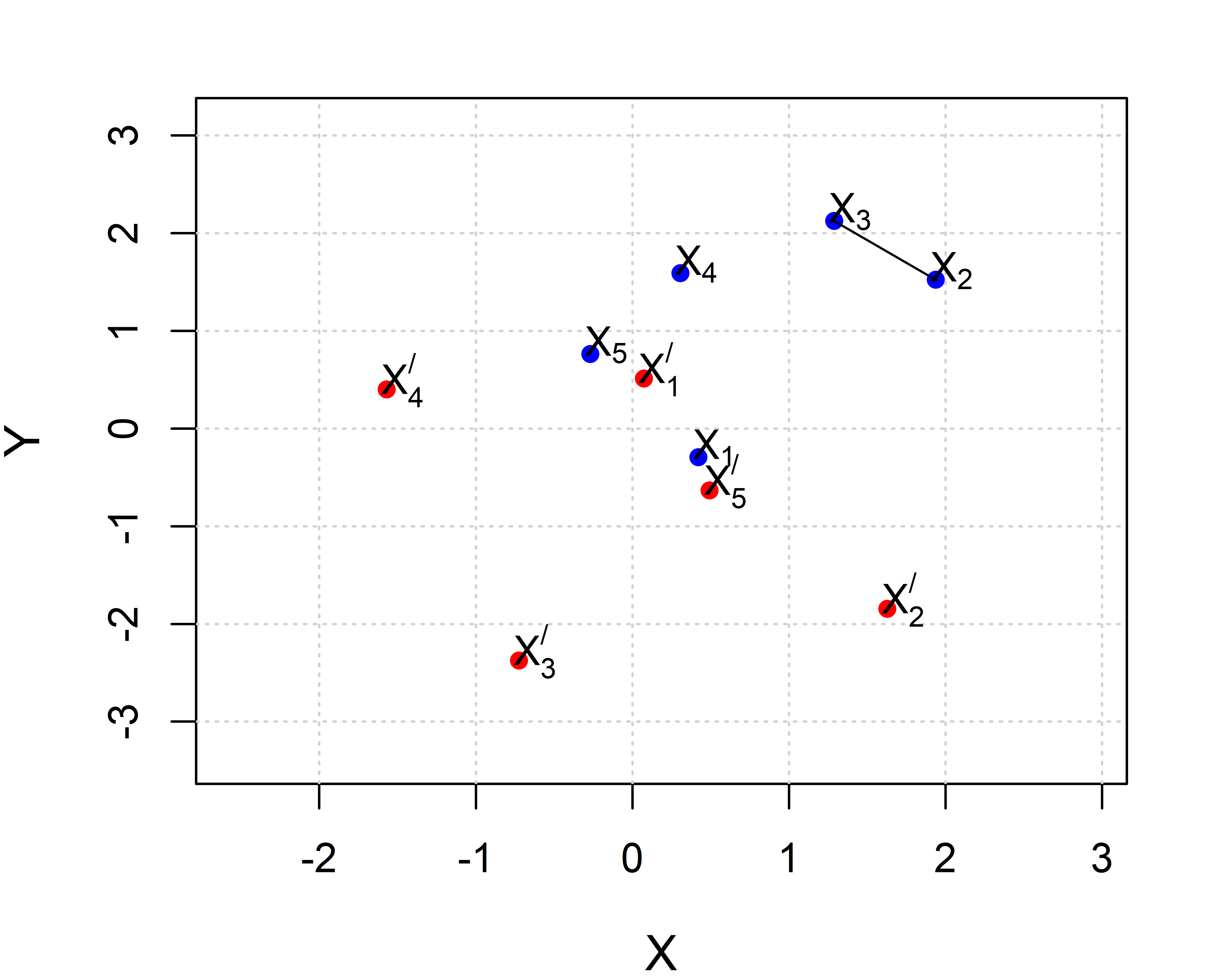} & \includegraphics[width=3in,height=2.1in]{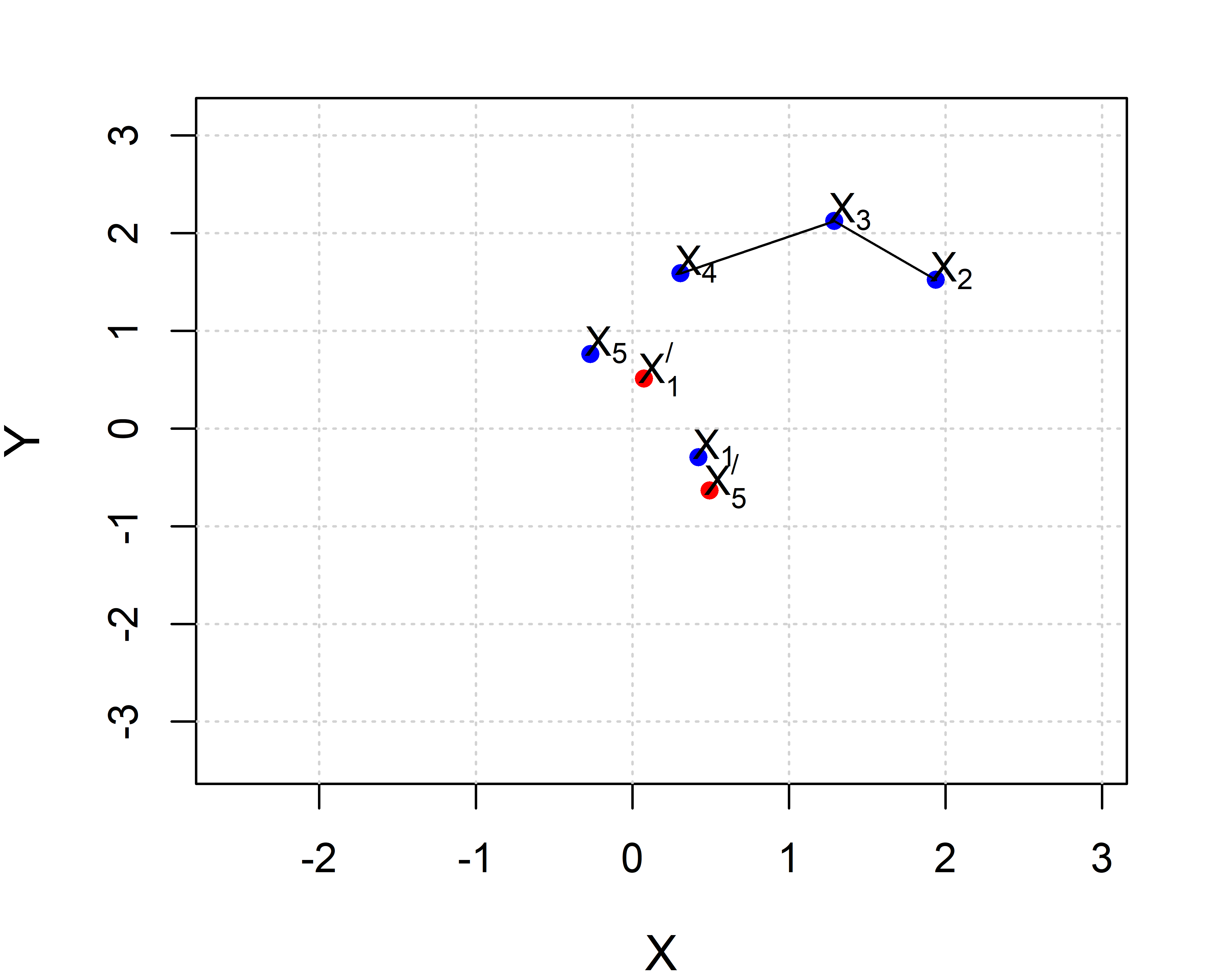}\\ 
			(c) Step 3& (d) Step 4 \\
			\includegraphics[width=3in,height=2.1in]{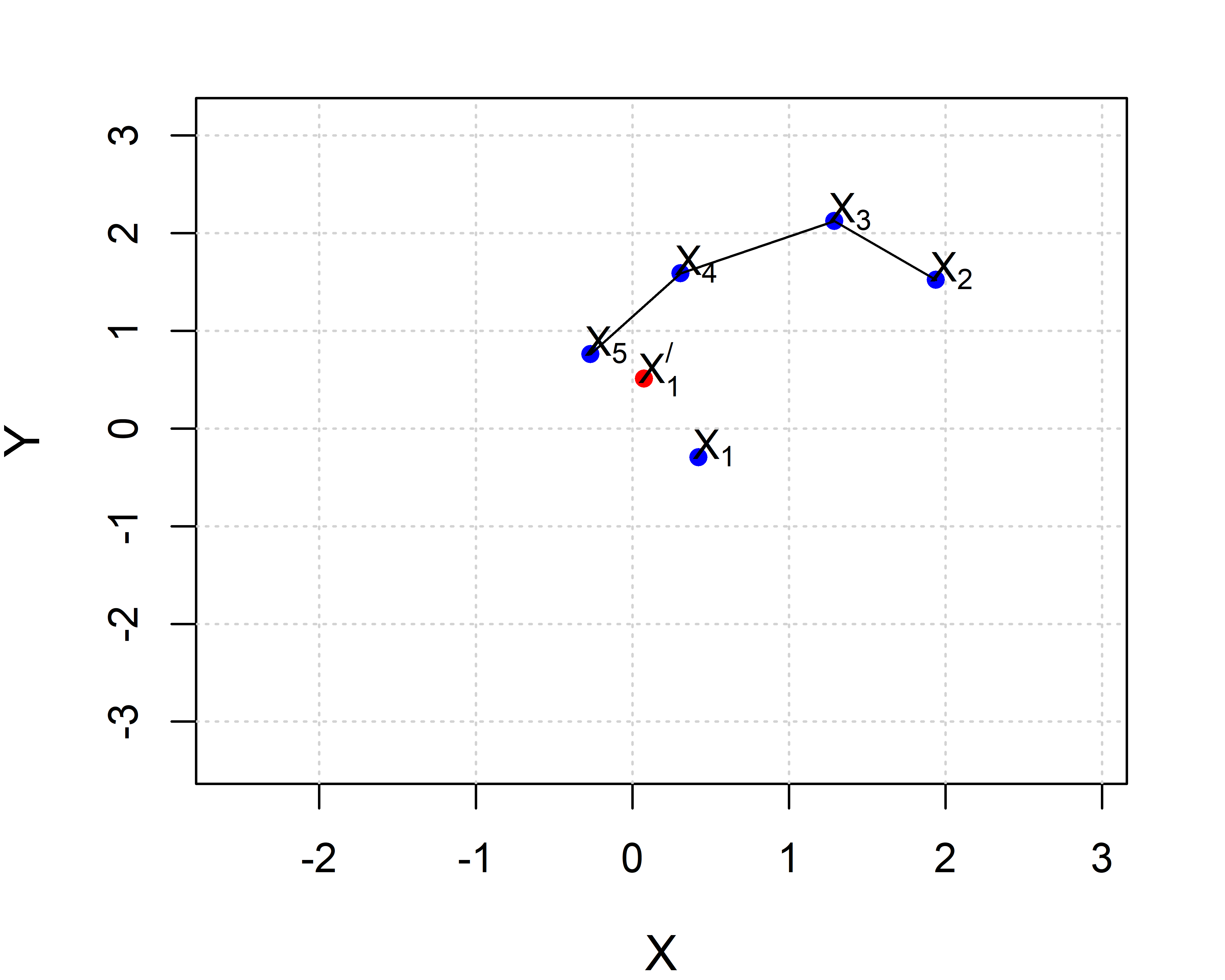} & \includegraphics[width=3in,height=2.1in]{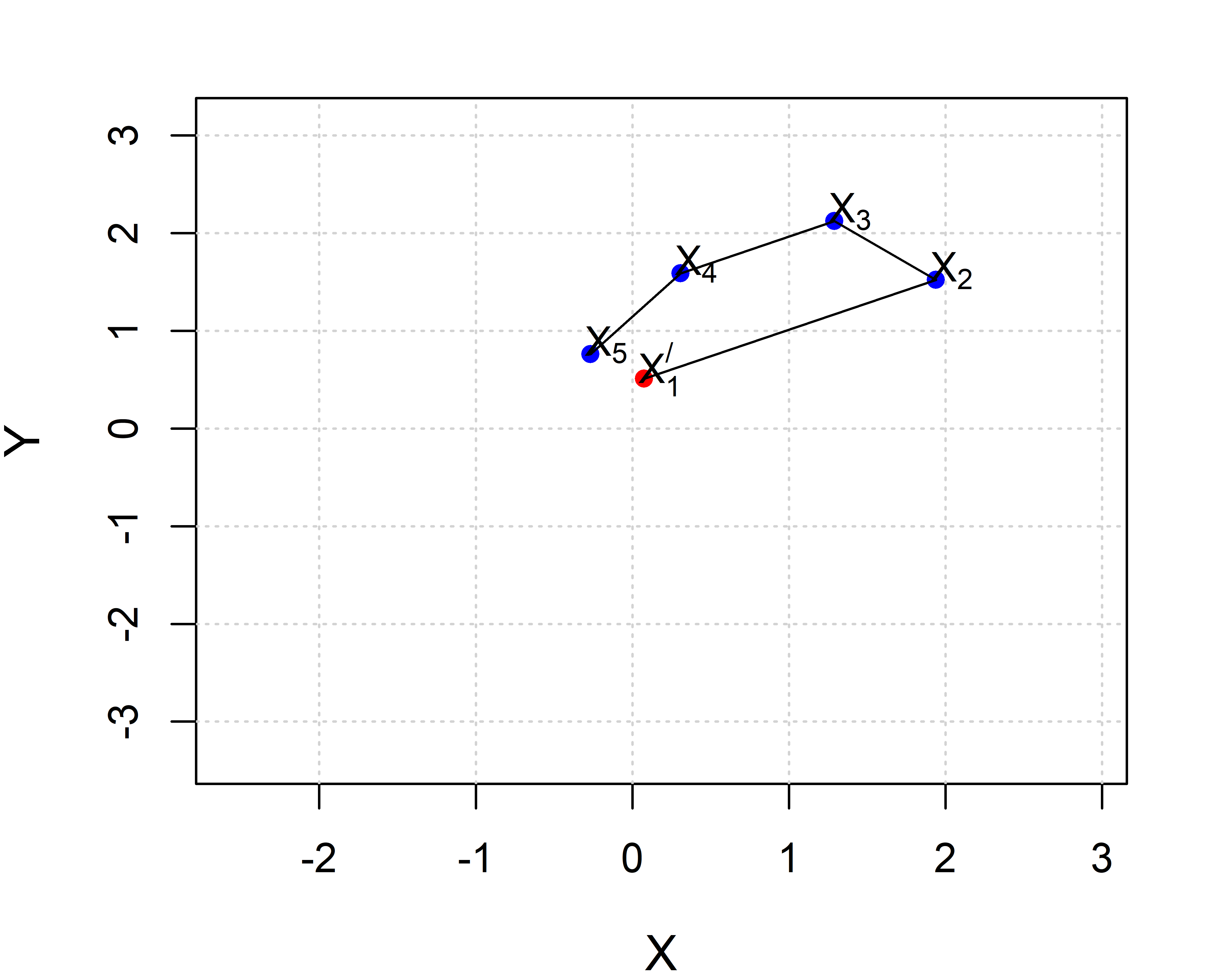}  
		\end{tabular}
	
	\vspace{-0.1in}
		\caption{Algorithm for constructing the shortest covering path based on  $\bm X_1,\ldots, \bm X_5\stackrel{i.i.d.}{\sim} \mathcal{N}(0,0,1,1,0.5)$ and their spherically symmetric variants $\bm X_1^\prime,\ldots,\bm X_5^\prime$
			when $\Theta$ is used as the cost matrix.
		}
		\label{fig:scp-construction}
	\end{figure}
	\begin{figure}[!b]
		\centering
		\setlength{\tabcolsep}{-2pt}
		\begin{tabular}{cc}
			(a) Difference of Sign Statistic ~~~& (b) Difference of Runs Statistic \\
			\includegraphics[width=3in,height=2.0in]{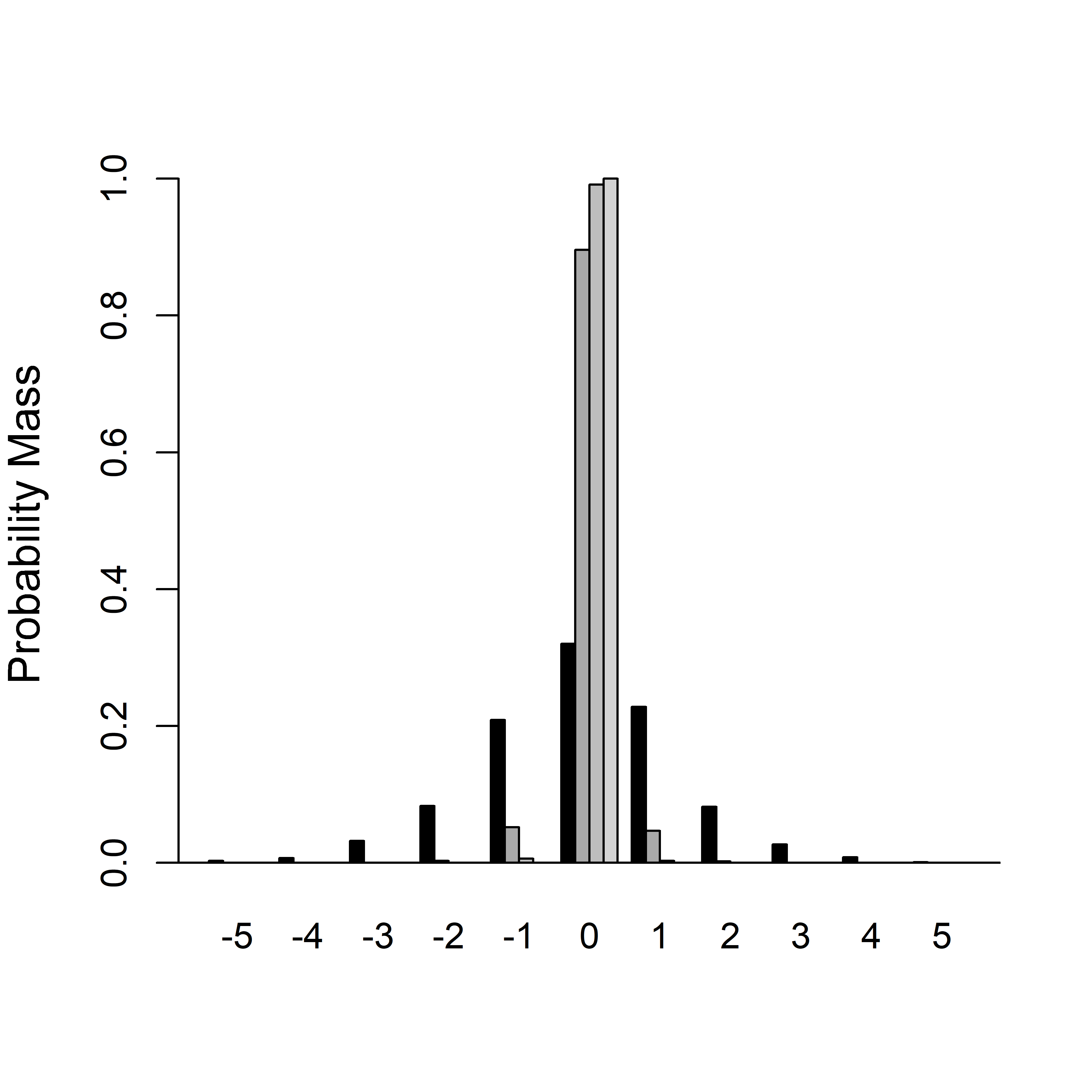} & \includegraphics[width=3in,height=2.0in]{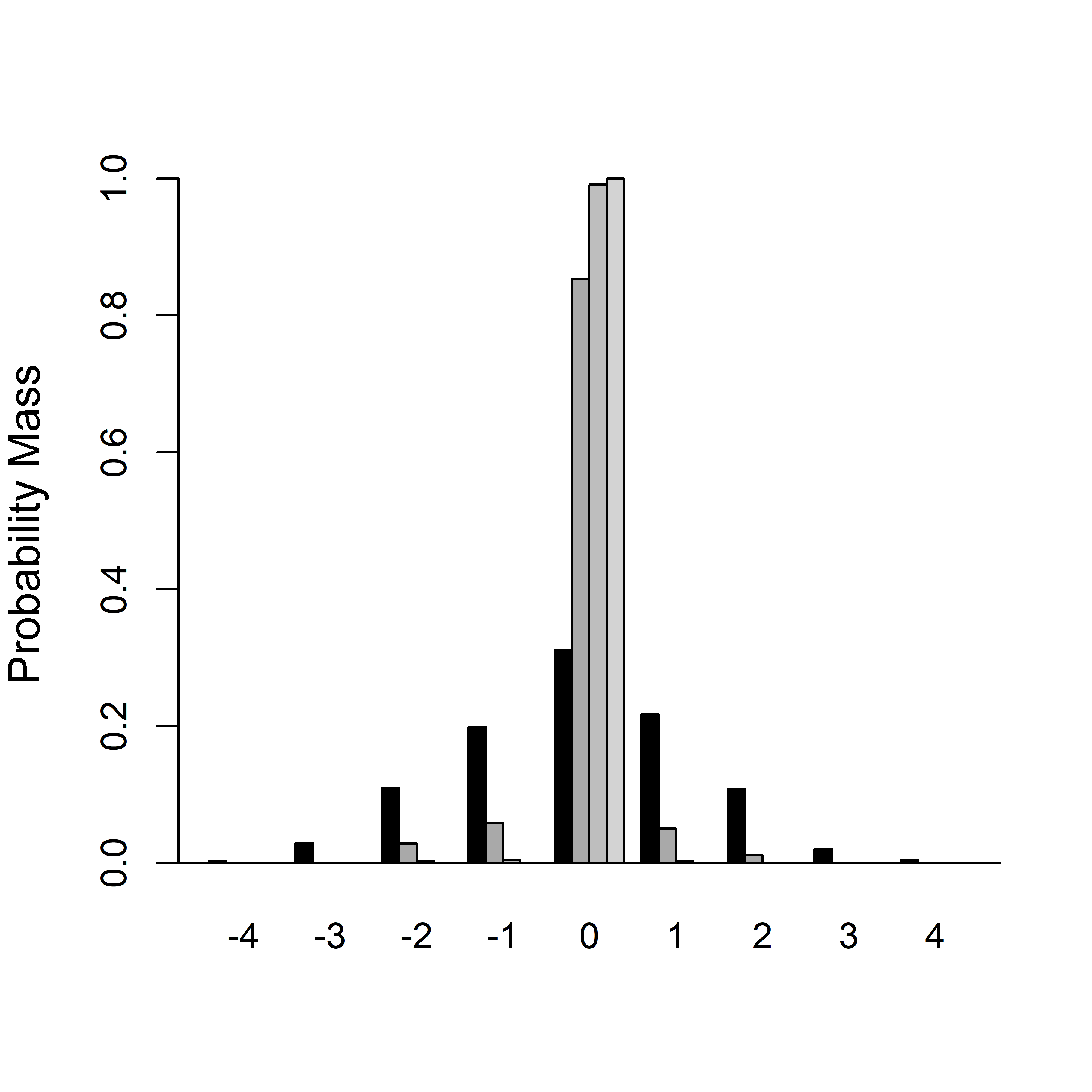}\\ 
		\end{tabular}
		\vspace{-0.25in}
		\caption{The barplot of the difference between (a) the sign statistic and (b) the runs statistic constructed based on $\mathcal{P}$ and ${\mathcal{P}}_0$, when $\bm X_1,\ldots, \bm X_5$ are generated independently from $\mathcal{N}(0,D)$ where $D = \textnormal{diag}(d,1,1,\ldots,1)$ for $d$ = 3 (black), 30 (dark gray), 300 (gray), and 3000 (light gray).}
		\label{fig:statistic-difference}
	\end{figure}

	Clearly, this is a heuristic algorithm and it may lead to a sub-optimal solution of \eqref{eq:optimization} in some cases.  However, it is clear that under $H_0$, $\bm S$ and $\bm R$ have the same distribution irrespective of whether they are computed along the actual shortest covering path $\mathcal{P}$ or along the shortest covering path computed using the algorithm (call it ${\mathcal P}_0$). This is because of the exchangeability of $\bm X_1,\bm X_2,\ldots,\bm X_n$ and their 
	symmetric variants. But, in some cases, the values of $T_S$ and $T_R$ may differ if they are computed along these two paths. To investigate this, we generate 5 observations from a $d$-variate normal distribution with mean at the center and covariance matrix being $\textnormal{diag}(d,1,1,\ldots,1)$.
	In this case it is possible to find the actual ${\mathcal P}$ by complete enumeration. We compute $T_S$ and $T_R$ along the paths $\mathcal{P}$ and ${\mathcal{P}}_0$ and calculate the corresponding differences. For each value of $d$ ($3,30,300,3000$), the experiment is repeated $100$ times, and the distributions of the differences are given by bar plots in Figure \ref{fig:statistic-difference}. We observe that both for $T_S$ and $T_R$, with increasing dimensions the difference concentrates around zero. So, in moderate and higher dimensions, the test statistic computed along ${\mathcal P}$ matches with that computed using
	${\mathcal P}_0$ in almost all cases. We have seen that in high dimensions, the actual shortest covering path ${\mathcal P}$ is supposed to contain all $\bm X_i$s. In almost all cases ${{\mathcal P}}_0$ also leads to a similar path, but the arrangements of the $\bm X_i$s along these two paths differ
	in some cases. Therefore, though our heuristic algorithm leads to a sub-optimal solution in terms of the cost of the path, in most of the cases, ${\mathcal P}_0$ and actual $\mathcal{P}$ lead to the same values of the test statistics. This justifies the use of our heuristic algorithm for finding the shortest covering path and computing the test statistics, especially for moderate or large dimensional data.

	\subsection{Inner products vs. cosine similarities}
	
	Since, $\|\bm X_i\| = \|\bm X_i^\prime\|$ for all $i=1,2,\ldots, n$, instead of using the cost based on squared inner products $(\bm X_i^{\top}\bm X_j)^2$, $(\bm X_i^{\top}\bm X_j^\prime)^2$ and $(\bm X_i^{\top}\bm X_j)^2$, one may be tempted to use a cost based on squared cosine similarities $C_{ij}^2 = \Big({\frac{\bm X_i^{\top}\bm X_j}{\|\bm X_i\|\|\bm X_j\|}}\Big)^2$, $C_{ij^\prime}^2 = \Big({\frac{\bm X_i^{\top}\bm X_j^{\prime}}{\|\bm X_i\|\|\bm X_j^{\prime}\|}}\Big)^2$ and $C_{i^\prime j^\prime}^2 = \Big({\frac{\bm X_i^{\prime}{\top}\bm X_j^\prime}{\|\bm X_i^\prime\|\|\bm X_j^\prime\|}}\Big)^2$ ($1\le i,j \le n$). If the underlying distribution is spherically symmetric,
	these scaled versions of inner products $C_{ij}, C_{ij^\prime}$ and $C_{i^\prime j^\prime}$ have the same distribution as ${\bm U_1}^{\top}{\bm U}_2$, where $\bm U_1$ and $\bm U_2$ are independent and identically distributed as uniform random vectors on ${\cal S}^{d-1}$. But, they have the same property for a class of angular symmetric distribution (i.e., $\bm X/\|\bm X\| \stackrel{D}{=} -\bm X/\|\bm X\|$) that are not spherically symmetric. In such cases, the resulting test fails to reject the null hypothesis. 
	
	To demonstrate this, we consider a simple example involving an angular symmetric distribution.
	We generate $200$ independent observations on a bivariate random vector $\bm X = R\bm U$ where $\bm U = (U_1,U_2)$ is uniformly distributed on the perimeter of the unit circle (i.e., $(U_1,U_2)=(\cos(\theta), \sin(\theta))$, where $\theta \sim U(0, 2\pi)$) and $R = R_1 \mathrm{I}\{U_1U_2>0\} + \mathrm{I}\{U_1U_2\leq0\}$, for $R_1$ being uniformly distributed over the interval $[1,5]$ independent of $\bm U$. Using these observations, we computed $T_S$ and $T_R$ based on squared inner products and those based on squared cosine similarities. The boxplots in Figure~\ref{fig:rose-histograms-22} show the distribution of these sign and runs statistics based on 1000 repetitions of the experiment. 
	Note that here $R$ and $\bm U$ are not independent. So, the distribution of $\bm X$ is not spherically symmetric. But the sign and runs statistics based on the squared cosine similarities could not figure it out. The distributions of these statistics were same as their corresponding null distributions.
	But for our proposed cost function, the sign statistic had higher values and the runs statistics had lower values leading to
	the rejection of the null hypothesis in most of the cases. This example clearly shows that any test based on $\bm X_i/\|\bm X_i\|$ ($i=1,2,\ldots,n$)
	fails to detect deviation from spherical symmetry if $\bm X_i/\|\bm X_i\|\sim$ Unif$(\mathcal{S}^{d-1})$ . It only tests whether the distribution of $\bm X/\|\bm X\|$ is uniform, but does not test for the independence between $\|\bm X\|$and $\bm X/\|\bm X\|$. The tests proposed by \cite{zou2014} and \cite{feng2017high} have a similar problem. This may be reason why these authors proposed their tests 
	assuming elliptic symmetry of the underlying distribution.
	
	\begin{figure}[h]
		\centering
		\setlength{\tabcolsep}{-2pt}
		\begin{tabular}{cccc}
			(a) Sign statistic& (b) Runs statistic
			\\
			\vspace{-0.15in}
			\includegraphics[width=3.25in,height=2.5in]{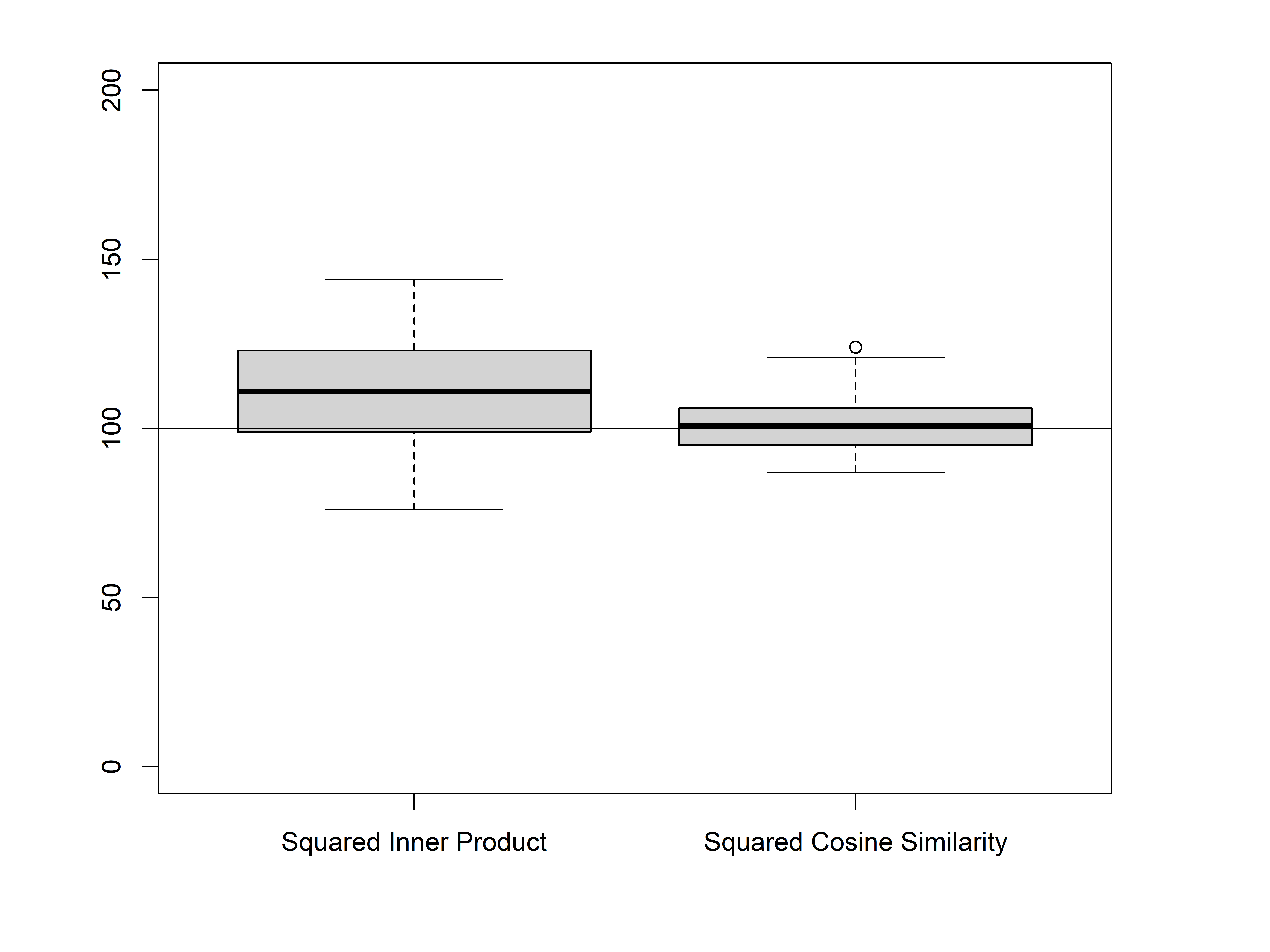} & \includegraphics[width=3.25in,height=2.5in]{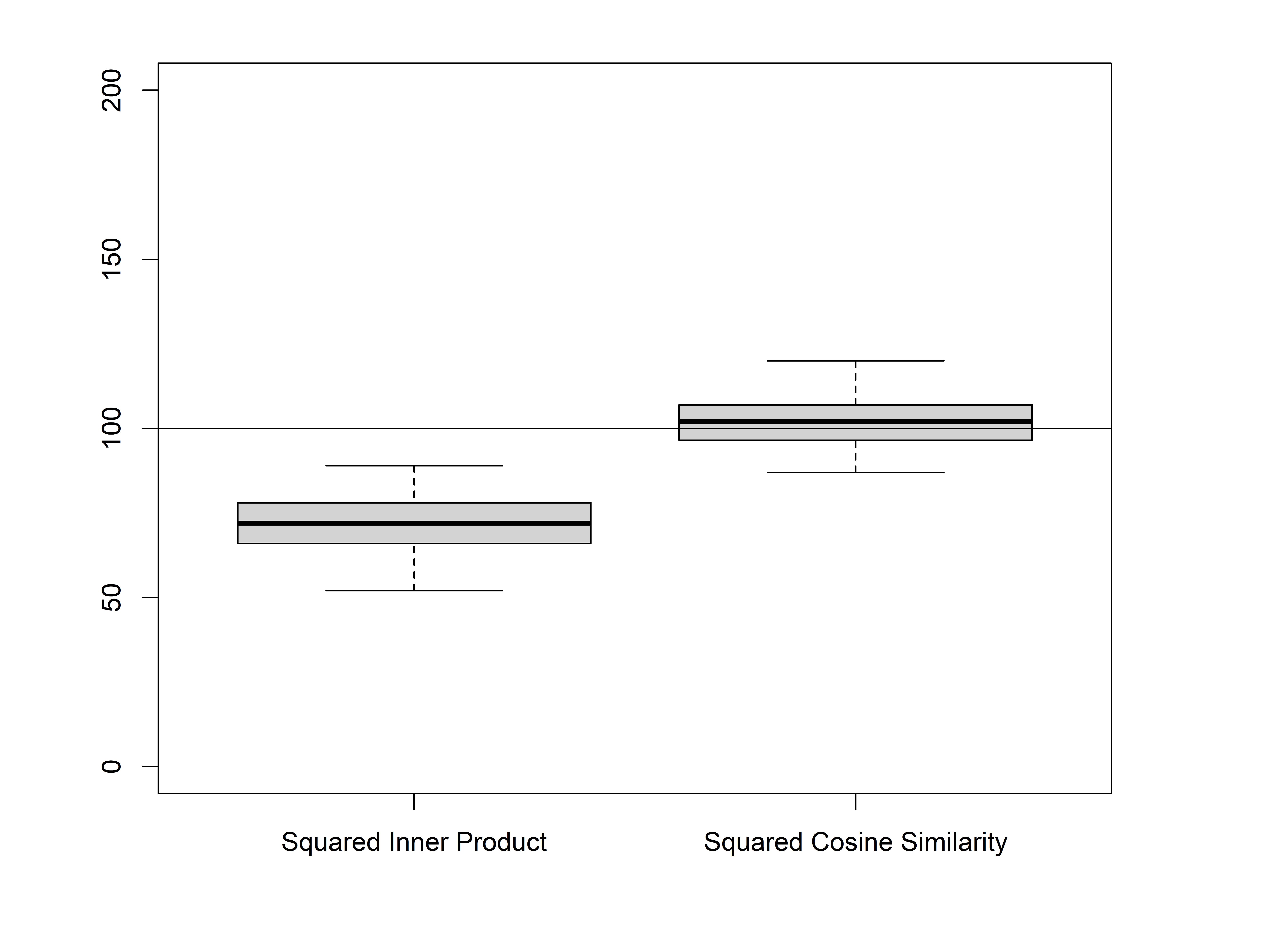}
		\end{tabular}
		\vspace{-0.15in}
		\caption{Boxplot of Sign and Runs statistics (based on 100 replications) when $200$ observations are generated from angular symmetric distribution. The solid line indicates the expectation of the test statistic
			under spherical symmetry.
		}
		\label{fig:rose-histograms-22}
	\end{figure}
	
	\section{High dimensional behavior of the proposed tests}
	\label{sec:tests}
	
	Let $\mathcal{D}= \{\bm X_1,\bm X_2,\ldots,\bm X_n\}$ be a data set consisting of $n$ independent observations form a $d$-dimensional distribution $\Pr$. We have already seen that if $\Pr$ is spherical, any function of the vectors of string signs $\bm S$ and string ranks ${\bm R}$ has the exact  distribution-free property (see Theorem \ref{thm:sign-rank-dist}). In particular, we consider the runs statistic $T_R$ and the sign statistic $T_S$. The later one can be viewed as a linear rank statistic of the form $T_{LR}=\sum_{i=1}^{n}S_{\pi_i} a(i)$, where the score function $a(\cdot)$ is given by $a(i) = 1$ or for all $i=1,2,\ldots,n$. From part (a) of Theorem \ref{thm:sign-rank-dist}, it is transparent that for any given score function $a(\cdot)$, the finite sample null distribution and the of $T_{LR}$ is same as the null distribution of the corresponding univariate linear rank statistic. So, its large sample distribution also matches with that of its univariate analog. 
	
	From the description of our tests, it is clear that they can be conveniently used for high-dimensional data even when the dimension is much larger than the sample size. In the next two sub-sections, we study the asymptotic behavior of these tests when the dimension grows to infinity while sample size either remains fixed or grows with the dimension. These two asymptotic regimes will be referred to as the high-dimension low sample size (HDLSS) and the high-dimension, high sample size (HDHSS) asymptotic regimes, respectively.

	\subsection{Behavior of the proposed tests for HDLSS data}\label{HDLSS}
	
	Following the seminal work by \cite{hall2005geometric}, investigation of the behavior of different statistical methods in HDLSS asymptotic regime has gained considerable interest. \cite{hall2005geometric} showed that under certain regularity conditions on moments and weak dependence of the measurement variables, observations from a high dimensional distribution tend to lie on the vertices of a regular simplex when the dimension diverges to infinity. \cite{ahn2007high,jung2009pca, yata2012} also provided another set of conditions based on the covariance matrix for a similar geometry of high dimensional data. 
	This geometric feature of the data cloud is used extensively to study the behavior of several one and two sample tests \citep[see, e.g., ][]{liu2011triangle,biswas2014distribution,biswas2015onesample,ghosh2016distribution,wei2016direction,tsukada2019high,kim2020robust,banerjee2025high} nonparametric classifiers   \citep[see, e.g.,][]{chan2009avgclassifier,dutta2016ghoshclassifier,pal2016high,roy2022hdlssclassification} and clustering methods \citep[see,e.g.][]{ahn2012clustering,sarkar2019perfect}. But, the HDLSS behavior of the tests of spherical symmetry is somewhat missing from the literature. There is a fundamental difficulty in this context. 
	Before discussing it, let us recall the following theorem by \cite{jung2009pca}.
		
	{\begin{thm}\label{thm:HDLSS-difficulty}
			Let $\bm X_1,\bm X_2$ be two independent copies of $\bm X \sim \Pr$, a $d$-dimensional distribution with the following properties.
			\begin{enumerate}
				\item[(A1)] $E(\bm X)={\bf 0}$
				\item[(A2)] Let 	$\sigmat = {\bf U}{\bf \Lambda}{\bf U}^\top$ be the spectral decomposition of $\sigmat=\var(\bm X)$, where ${\bf \Lambda}=\textnormal{diag}(\lambda_{1},\lambda_{2},\ldots,\lambda_d)$ is the diagonal matrix containing the eigenvalues $\lambda_1\geq \lambda_2\geq \cdots\geq \lambda_d$ of $\sigmat$, and ${\bf U}=[{\bf u}_1,{\bf u}_2,\ldots,{\bf u}_d]$ is the orthogonal matrix whose columns are the corresponding eigenvectors. The coordinates of $\bm Z = {\bf \Lambda}^{-1/2}{\bf U}^\top \bm X$ have uniformly bounded fourth moments and the $\rho$-mixing property under some permutation.
			\end{enumerate}
			Also assume that 
			$$\epsilon = \frac{\sum_{i=1}^d \lambda_{i}^2}{(\sum_{i=1}^d\lambda_{i})^2}\rightarrow 0\textnormal{ as } d\rightarrow\infty.$$ 
			Then $c_d^{-1} S_D \stackrel{P}{\rightarrow} \mathrm{\bf I}_n\textnormal{ as  } d\rightarrow\infty$ where $S_D = (\bm X_i^\top \bm X_j)_{1\leq i,j\leq n}$, $ \mathrm{\bf I}_n$ is the $n\times n$ identity matrix and $c_d = \sum_{i=1}^d \lambda_i$.
	\end{thm}}
	
    The condition $\epsilon \rightarrow 0$ as $d \rightarrow \infty$ is also known as the sphericity condition. As a consequence of Theorem \ref{thm:HDLSS-difficulty}, under the given conditions $c_d^{-1}\|\bm X_1\|^2$ and $c_d^{-1} (\bm X_1^\top\bm X_2)$ converges in probability to one and zero, respectively. 
	So, the data cloud from $\Pr$
	behaves like as if they are coming from a spherical distribution. Therefore, any test of spherical symmetry based on pairwise distances or inner products has asymptotic power close the nominal level $\alpha$ in high-dimensions. Hence, for good performance of a test of spherical symmetry in HDLSS situations, one needs to operate outside this sphericity condition. The following theorem gives us a direction in the context.
	
	{
		\begin{thm}\label{thm:HDLSS-consistency}
			Let $\bm X_1,\bm X_2,\ldots,\bm X_n$ be $n$ independent copies of $\bm X$, which follows a $d$-dimensional non-spherical distribution $\Pr$. Also assume that as $d$ diverges to infinity 
			\begin{align}\label{eq:HDLSS-condition}
				\P\left[\frac{d(\bm X_1^\top\bm X_2)^2}{\|\bm X_1\|^2\|\bm X_2\|^2}>M\right]\rightarrow 1~~~~\textnormal{for all $M>0$}.
			\end{align}
			Then $\bm S$, the vector string signs, converges to ${\bf 1}_n = (1,1,\ldots,1)$ in probability as $d$ diverges to infinity. 
		\end{thm}
		Since $\left\{\frac{d(\bm X_1^\top\bm X_2^\prime)^2}{\|\bm X_1\|^2\|\bm X_2^\prime\|^2}\right\}_{d\ge 1}$
		and $\left\{\frac{d({\bm X_1^\prime}^\top\bm X_2^\prime)^2}{\|{\bm X_1^\prime}\|^2\|\bm X_2^\prime\|^2}\right\}_{d\ge 1}$ are two tight sequences of random variables (see the proof of Theorem \ref{thm:HDLSS-consistency}), condition \eqref{eq:HDLSS-condition} ensures that  $\theta(\bm X_1,\bm X_2)$ becomes smaller than $\theta(\bm X_1,\bm X_2^\prime)$ and $\theta(\bm X_1^\prime,\bm X_2^\prime)$ with probability tending to $1$ as $d$ grows to infinity. One can show that condition 
		\eqref{eq:HDLSS-condition} holds for the spiked covariance model considered in \cite{jung2009pca}. So, as a corollary, we have the following result.
		
		\begin{cor}\label{cor:spiked-HDLSS}
			Let $\bm X_1,\bm X_2$ be two independent random variables from a $d$-dimensional distribution $\Pr$ satisfying (A1) and (A2) mentioned in Theorem \ref{thm:HDLSS-difficulty}. Also assume that
			\begin{itemize}
				\item[(a)] $\lambda_{1}/d^{\alpha}\rightarrow c_1$ for some $\alpha\geq 1$ and $c_1>0$,
				\item[(b)] $\sum_{i=2}^d \lambda_{i}^2/(\sum_{i=2}^d \lambda_{i})^2\rightarrow 0$ as $d\rightarrow \infty$ and $\sum_{i=2}^d \lambda_{i} = \mathcal{O}(d)$.
			\end{itemize}
			Then $\bm S$ converges to ${\bf 1}_n = (1,1,\ldots,1)$ in probability as $d$ diverges to infinity. 
		\end{cor}
		
		As a consequence of Theorem \ref{thm:HDLSS-consistency}, for any given sequence of scores $\{a(i)\}_{1\leq i\leq n}$, $T_{LR} = \sum_{i=1}^n S_{\pi_i} a(i)$ converges to $\sum_{i=1}^n a(i)$ in probability as $d$ diverges to infinity. So, if these scores $\{a(i)\}_{1\leq i\leq n}$ are non-negative, which is usually the case, $T_{LR}$ takes its largest value with probability tending to $1$ as $d$ diverges to infinity. In the case of the sign statistics $T_S$, we have $a(i)=1$ for all $i=1,2,\ldots,n$. Therefore, $P(T_S=n) \rightarrow 1$ as $d \rightarrow \infty$. Now, under $H_0$, we have $P(T_S \ge n)= 1/2^n$. So, for any fixed level $\alpha$ ($0<\alpha<1$), unless the sample size is very small (i.e. $2^n<1/\alpha$), the power of the proposed sign test converges to $1$
		as the dimension increases. Similarly, under the condition of Theorem  \ref{thm:HDLSS-consistency},
		the runs statistic $T_R$ converges to $1$ in probability. Now, under $H_0$, we have $P(T_R \le 1)=1/2^{n-1}$. So, if $2^{n-1}>1/\alpha$, we have the consistency of the proposed runs test of level $\alpha$ in the HDLSS asymptotic regime.
		
		Now, we consider three simple examples involving normal distributions to study the empirical performance of the proposed tests in high dimensions, when the nominal levels of the tests are taken as 0.05.
		
		\begin{exa}\label{exa:first}
			Observations are generated from a $d$-variate normal distribution with mean ${\bf 0}$ and variance covariance matrix $\bm \Sigma = ((\sigma_{ij}))$, where $\sigma_{ij}$ is $1$ if $i=j$ and $0.6$ if $i\not=j$.
		\end{exa}
		
		\begin{exa}\label{exa:second}
			Here we consider a $d$-variate normal distribution with mean ${\bf 0}$ and a diagonal variance covariance matrix $\bm \Sigma = ((\sigma_{ij}))$, where $\sigma_{ii}=1$ for $1\leq i\leq [d/2]$ and $\sigma_{ii} = 2$ for $i\geq [d/2]+1$. 
		\end{exa}
		
		\begin{exa}\label{exa:third}
			Here also, we deal with a $d$-variate normal distribution with mean ${\bf 0}$ and a diagonal variance covariance matrix. The matrix has the first diagonal element $d$ and rest equal to $1$. 
		\end{exa}
		
		For each example, we consider 10 different choices of $d$ ($d=2^i$ for $i=1,2,\ldots, 10$) but a fixed value of $n$ ($n=50$). Each experiment is repeated $500$ times, and the empirical power of a test is computed as the proportion of times it rejects $H_0$. The results are reported in Figure \ref{fig:sim-1}.
		
		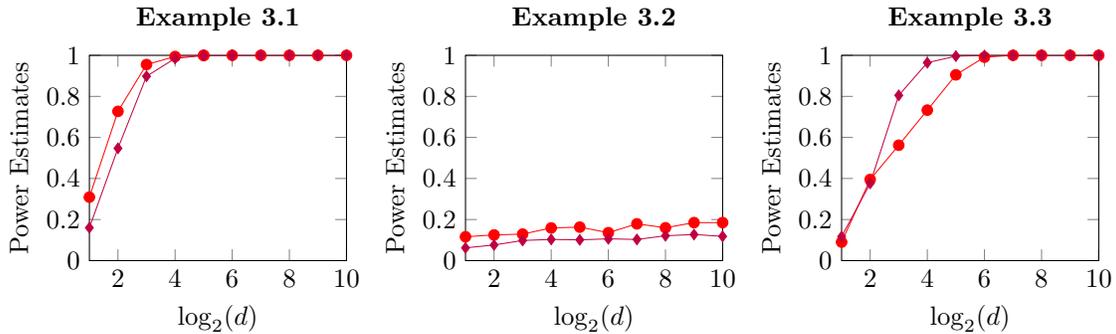
\begin{figure}[h]
			\centering
			\begin{tikzpicture}
				\begin{axis}[xmin = 1, xmax = 10, ymin = 0, ymax = 1, xlabel = {$\log_2(d)$}, ylabel = {Power Estimates}, title = {\bf Example \ref{exa:first}}]
					\addplot[color = red,   mark = *, step = 1cm,very thin]coordinates{(1,0.309)(2,0.727)(3,0.955)(4,0.994)(5,0.999)(5,1)(6,1)(7,1)(8,1)(9,1)(10,1)};
					
					
					\addplot[color = purple,   mark = diamond*, step = 1cm,very thin]coordinates{(1,0.16)(2,0.547)(3,0.898)(4,0.985)(5,0.999)(6,1)(6,1)(7,1)(8,1)(9,1)(10,1)};
					
				\end{axis}
			\end{tikzpicture}
			\begin{tikzpicture}
				\begin{axis}[xmin = 1, xmax = 10, ymin = 0, ymax = 1, xlabel = {$\log_2(d)$}, ylabel = {Power Estimates}, title = {\bf Example \ref{exa:second}}]
					\addplot[color = red, mark = *, step = 1cm,very thin]coordinates{(1,0.116)(2,0.125)(3,0.129)(4,0.159)(5,0.163)(6,0.136)(7,0.179)(8,0.16)(9,0.185)(10,0.185)};
					
					
					\addplot[color = purple,   mark = diamond*, step = 1cm,very thin]coordinates{(1,0.062)(2,0.076)(3,0.098)(4,0.103)(5,0.101)(6,0.106)(7,0.103)(8,0.121)(9,0.127)(10,0.118)};

				\end{axis}
			\end{tikzpicture}
			\begin{tikzpicture}
				\begin{axis}[xmin = 1, xmax = 10, ymin = 0, ymax = 1, xlabel = {$\log_2(d)$}, ylabel = {Power Estimates}, title = {\bf Example \ref{exa:third}}]
					\addplot[color = red, mark = *, step = 1cm,very thin]coordinates{(1,0.09)(2,0.395)(3,0.562)(4,0.733)(5,0.905)(6,0.991)(7,1)(8,1)(9,1)(10,1)};
					
					
					\addplot[color = purple,   mark = diamond*, step = 1cm,very thin]coordinates{(1,0.117)(2,0.376)(3,0.805)(4,0.964)(5,0.996)(6,0.999)(7,1)(8,1)(9,1)(10,1)};

				\end{axis}
			\end{tikzpicture}
			\caption{Observed power of the Sign test (\tikzsymbol[circle]{minimum width=2pt,fill=red}) 
				and Runs test (\tikzsymbol[diamond]{minimum width=2pt,fill=purple}) when $50$ observations are generated from the $d$-variate normal distributions (with $d=2^i, i=1,2,\ldots,10$) considered in Examples 3.1-3.3.}
			\label{fig:sim-1}
		\end{figure}
		
		One can check the normal distributions in Examples 3.1 and 3.3 satisfy conditions (a) and (b) of Corollary \ref{cor:spiked-HDLSS}. So, as expected, in these two cases, the powers of the sign and runs tests sharply raised to $1$. But in Example 3.2, we had a diametrically opposite picture, where both sign and runs tests failed to have satisfactory performance. Note that in this example, the sphericity condition (see Theorem \ref{thm:HDLSS-difficulty}) is satisfied, and this should be the reason behind the poor performance of these tests.
		
		\subsection{Behavior of the proposed tests in the HDHSS asymptotic regime}
		\label{sec:HDHSS}
		
		In thus section, we investigate the behavior of the proposed tests for high dimensional data sets having large number of observations. This type of data sets commonly arises in the field of biology, ecology, and medical sciences. Here we study the asymptotic behavior of the tests when the dimension and the sample size grow simultaneously, but their divergence rates of are arbitrary. 
		Since the null distributions of $T_{LR}$ (or $T_S$ in particular) and $T_R$ do not depend on the dimension of the data, their limiting null distributions in the HDHSS regime remains the same as they are in the classical asymptotic regime. 
		The asymptotic null
		distribution of the linear rank statistic $T_{LR}$ is given by the following theorem.
		
		\begin{thm}\label{thm:limit-null-distribution}
			Let $\bm X_1,\bm X_2,\ldots,\bm X_n$ be independent realizations of a $d$-dimensional random vector $\bm X \sim \Pr$. Assume that $\Pr$ is spherically symmetric and the  sequence of scores $\{a(i)\}_{1\leq i\leq n}$ satisfies the following conditions
			\begin{align}\label{eq:score-conditions}
				\sum_{i=1}^n a^2(i)\rightarrow \infty ~~~~\textnormal{and}~~~~ \max_{1\leq i\leq n}\frac{a^2(i)}{\sum_{i=1}^n a^2(i)}\rightarrow 0 ~~~~\mbox{as}~~ n\rightarrow\infty.
			\end{align}
			Then, as $n$ and $d$ both grow to infinity, we have
			$$\frac{T_{LR} - \frac{1}{2}\sum_{i=1}^n a(i)}{\sqrt{\sum_{i=1}^n a^2(i)}} \stackrel{D}{\rightarrow}\mathcal{N}\left(0,\frac{1}{4}\right).$$
		\end{thm}
		
		In particular we have $n^{-1/2}(T_S-n/2)\stackrel{D}{\rightarrow}\mathcal{N}(0,0.25)$. Theorem \ref{thm:limit-null-distribution} holds even when $d$ is fixed and $n$ diverges to infinity. So, irrespective of the value of $d$, when  the sample size is large, this test can be calibrated using the quantiles of the normal distribution. Now, we investigate the asymptotic behavior of $T_{LR}$ for non-spherical distribution. In this context, we have the following result.
		


		\begin{thm}\label{thm:limit-alt-convergence}
			Let $\bm X_1,\bm X_2,\ldots,\bm X_n$ be $n$ independent copies of a   $d$-dimensional random vector $\bm X \sim \Pr$. Assume that the sequence of scores $\{a(i)\}_{1\leq i\leq n}$ satisfies condition \eqref{eq:score-conditions} and as $n \rightarrow\infty$,
			\begin{align}\label{eq:consistency-cond}
				\frac{\sum_{i=1}^{n-1}a(i)a(i+1)}{\sum_{i=1}^n a^2(i)}\rightarrow C,
			\end{align}
			for some $C>0$. If $\Pr$ is a non-spherical distribution, then there exist finite constants $\sigma_{11}$ and $\sigma_{12}$.
			such that $$\limsup_{n,d\rightarrow\infty}\var\left[\frac{T_{LR} - \E[T_{LR}]}{\sqrt{\sum_{i=1}^n a^2(i)}}\right] = \sigma_{11} + 2C \sigma_{12}.$$
		\end{thm}
		
		\begin{rem}
			For any fixed dimension, one can derive the large sample distribution of $T_{LR}$ against a sequence of contiguous alternatives \citep[see Chapter 12][ for contiguous alternatives]{lehmanntesting} and prove its Pitman efficiency (see Theorem \ref{thm:pitman-efficiency} in Appendix). This is in sharp contrast to the results in \cite{bhattacharya2019general}, where the author proved that in multivariate setup, most of the graph based distribution-free two sample tests turn out to be inefficient in the classical asymptotic regime. 
		\end{rem}
		
		{
			Now note that for the sign statistic $T_S = \sum_{i=1}^n S_i$, 
			 we have $[T_S-E(T_S)]/\sqrt{n}=O_p(1)$ both under the null and alternative hypothesis. Hence, we have the probability convergence of $|T_S-E(T_S)|/n$ to $0$. One can also show that
			\begin{align*}\Big|\frac{1}{n}E(T_S) - \P\left[\theta(\bm Y_{S_1,1},\bm X_2)+\theta(\bm X_2,\bm Y_{S_3,3})\leq \theta(\bm Y_{\pi_1,1},\bm X_2^\prime)+\theta(\bm X_2^\prime,\bm Y_{S_3,3})\right]\Big| \rightarrow 0 \mbox { as } n,d \rightarrow \infty,
			\end{align*}
			where $\bm Y_i=S_i\bm X_i+(1-S_i)\bm X_i^\prime$ for $i=1,2,3$ (see Lemma \ref{lemma:sign-stat-convergence} for the proof). Let us define 
			$$p_S= \liminf_{d \rightarrow \infty} \P\left[\theta(\bm Y_{S_1,1},\bm X_2)+\theta(\bm X_2,\bm Y_{S_3,3})\leq \theta(\bm Y_{\pi_1,1},\bm X_2^\prime)+\theta(\bm X_2^\prime,\bm Y_{S_3,3})\right]\Big.$$
			as the limiting value of the probability of inclusion of any $\bm X_i$ ($i=1,2,\ldots,n$) (or the probability of non-inclusion of $\bm X_i^\prime$) in the shortest covering path $\mathcal{P}$. Under $H_0$, because of the exchangeability of
			$\bm X_i$ and $\bm X_i^\prime$, this probability turns out to be $0.5$. However, in view of Theorem \ref{thm:HDLSS-consistency} and Lemma \ref{lem:expectations},  under $H_1$, we expect $\theta(\bm X_1,\bm X_{2})$ to be  stochastically smaller than $\theta(\bm X_1,\bm X_{2}^\prime)$ and $\theta(\bm X_1^\prime,\bm X_{2}^\prime)$. So, $p_S$ is expected to be higher than $0.5$. Since we reject $H_0$ for higher values of $T_S$, under the condition $p_S>0.5$, the consistency of this sign test follows from Theorem \ref{thm:limit-alt-convergence}. }
		
		In the HDHSS setup, the limiting null distribution of the runs statistic $T_R$ remains the same as in the classical asymptotic regime. It is given by the following theorem.
		
		\begin{thm}\label{thm:large-dist-runs}
			Let $\bm X_1,\bm X_2,\ldots,\bm X_n$ be $n$ independent realizations of a $d$-dimensional random vector $\bm X \sim \Pr$. If $\Pr$ is spherically symmetric, as $n,d\rightarrow\infty$, we have
			$$\frac{T_R - (n+1)/2}{\sqrt{n}} \stackrel{D}{\rightarrow}\mathcal{N}\left(0,\frac{1}{4}\right)$$
		\end{thm}
		
		This asymptotic null distribution of $T_R$ ramins the same even for any fixed $d$ as $n$ grows to infinity. So, when the sample size is large, whatever be the dimension of the data, this runs test can be calibrated using the quatiles of a Gaussian distribution, and for any fixed nominal level $\alpha$, the cut-off remains the same for all values of $d$. Now, one may be curious to know about the asymptotic behavior of $T_R$ for non-spherical distributions when $n$ and $d$ both diverges to infinity. This is specified in the following theorem.
		
		{\begin{thm}\label{thm:large-dist-runs-alt}
				Let $\bm X_1,\bm X_2,\ldots,\bm X_n$ be $n$ independent copies of a $d$-dimensional random vector $\bm X \sim \Pr$. If $\Pr$ is non-spherical, there exists a positive constant $\sigma^2$ such that
				$$\limsup_{n,d\rightarrow\infty}\var\left[\frac{T_R - \E[T_R]}{\sqrt{n}}\right] = \sigma^2.$$
				Also, if the condition \eqref{eq:HDLSS-condition} is satisfied, $T_R$ converges in probability to one as $d$ diverges to infinity. 
		\end{thm}}
		
		From, Theorem \ref{thm:large-dist-runs-alt}, we have the probability convergence of $|T_R - \E(T_R)|/n$ to $0$ as $n,d \rightarrow \infty$.
		We know that under $H_0$, the limiting value of $E(T_R)/n$ and hence that of $T_R/n$ is $0.5$ (follows from Theorem \ref{thm:large-dist-runs}). But,
		as we have seen before, under the alternative most of the string signs are expected to be one, and as a result, $p_R=\limsup_{n,d \rightarrow\infty} E(T_R/n)$ is expected to be small. Since we reject $H_0$ for small values of $T_R$, when $p_r<0.5$, the runs test turns out to be consistent in the HDHSS regime.
		

				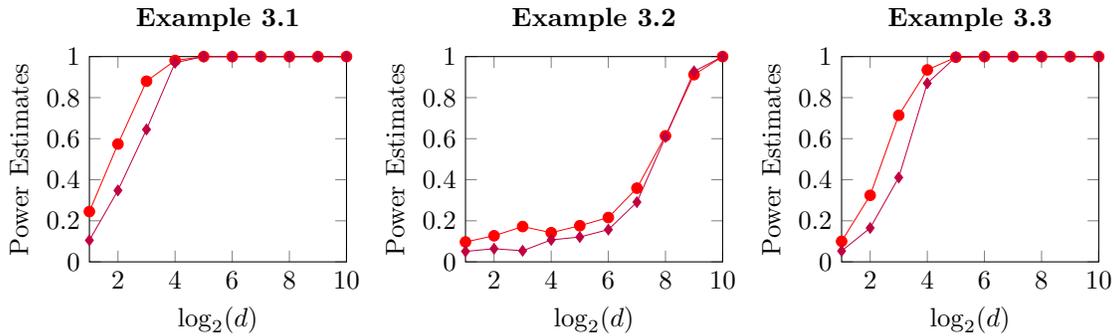
\begin{figure}[h]
					\centering
					\begin{tikzpicture}
						\begin{axis}[xmin = 1, xmax = 10, ymin = 0, ymax = 1, xlabel = {$\log_2(d)$}, ylabel = {Power Estimates}, title = {\bf Example \ref{exa:first}}]
							\addplot[color = red,   mark = *, step = 1cm,very thin]coordinates{(1,0.245)(2,0.574)(3,0.88)(4,0.981)(5,1)(6,1)(7,1)(8,1)(9,1)(10,1)};

							\addplot[color = purple,   mark = diamond*, step = 1cm,very thin]coordinates{(1,0.105)(2,0.348)(3,0.645)(4,0.97)(5,1)(6,1)(7,1)(8,1)(9,1)(10,1)};
							
						\end{axis}
					\end{tikzpicture}
					\begin{tikzpicture}
						\begin{axis}[xmin = 1, xmax = 10, ymin = 0, ymax = 1, xlabel = {$\log_2(d)$}, ylabel = {Power Estimates}, title = {\bf Example \ref{exa:second}}]
							\addplot[color = red, mark = *, step = 1cm,very thin]coordinates{(1,0.097)(2,0.127)(3,0.172)(4,0.142)(5,0.176)(6,0.216)(7,0.359)(8,0.614)(9,0.912)(10,1)};
							
							
							\addplot[color = purple,   mark = diamond*, step = 1cm,very thin]coordinates{(1,0.051)(2,0.064)(3,0.054)(4,0.107)(5,0.121)(6,0.157)(7,0.291)(8,0.608)(9,0.928)(10,1)};

						\end{axis}
					\end{tikzpicture}
					\begin{tikzpicture}
						\begin{axis}[xmin = 1, xmax = 10, ymin = 0, ymax = 1, xlabel = {$\log_2(d)$}, ylabel = {Power Estimates}, title = {\bf Example \ref{exa:third}}]
							\addplot[color = red, mark = *, step = 1cm,very thin]coordinates{(1,0.1)(2,0.324)(3,0.714)(4,0.935)(5,0.997)(6,1)(7,1)(8,1)(9,1)(10,1)};
							
							
							\addplot[color = purple,   mark = diamond*, step = 1cm,very thin]coordinates{(1,0.053)(2,0.165)(3,0.411)(4,0.87)(5,0.998)(6,1)(7,1)(8,1)(9,1)(10,1)};
							
						\end{axis}
					\end{tikzpicture}
					
					\caption{Observed power of the Sign test (\tikzsymbol[circle]{minimum width=2pt,fill=red}) and Runs test (\tikzsymbol[diamond]{minimum width=2pt,fill=purple}) when $n = d + 20$ observations are generated from the $d$-variate normal distributions (with $d=2^i, i=1,2,\ldots,10$) considered in Examples 3.1-3.3.}
					\label{fig:sim-1.1}
				\end{figure}

				We study the performance of sign and runs tests in Examples 3.1-3.3 for $10$ values of $d$ (i.e., $d=2^i$ for $i=1,2,\ldots,10$) as in Section \ref{HDLSS}, but this time instead of considering a fixed value of $n$, we take $n=d+20$ so that it also grows with the dimension. Each example is repeated to compute the empirical powers of the tests as before. In this three examples, we have $p_S$ larger than 0.5 (1, 0.577 and 1, respectively) and $p_R$ smaller than 0.5 (0, 0.422 and 0, respectively). So, as expected, in all there cases, the sign and the runs tests performed well for large values of $n$ and $d$. In Examples 3.1 and 3.3, powers of these tests raised sharply as before. But unlike what happened in the HDLSS set up, powers of these tests increased to $1$ in Example 3.2 as well. Since we get more information as the sample size size increases, such results are quite expected.
				
				\section{Further modifications of the proposed tests}
				\label{sec:improvements}
				
				Due to the sphericity condition in the HDLSS setup, our proposed sign and runs tests failed to have satisfactory performance in Example 3.2 even though the diagonal elements of the covariance matrix were different. This motivated us to look for further modifications of these tests. Note that the computation of $T_S$ and $T_R$ and the performance of the resulting tests depend on the construction of the shortest covering path using $\Theta$ as the cost matrix. One can notice that the ordering of  $\theta(\bm X_1,\bm X_2)$, $\theta(\bm X_1,\bm X_2^\prime)$ and $\theta(\bm X_1^\prime,\bm X_2^\prime)$ depends on that of $(\bm X_1^\top \bm X_2)^2$, $(\bm X_1^\top \bm X_2^\prime)^2$ 
				and $({\bm X_1^\prime}^\top \bm X_2^\prime)^2$. Now, we can break $(\bm X_1^\top \bm X_2)^2$ into two parts containing square terms and cross-product terms as
				\begin{align*}
					(\bm X_1^\top \bm X_2)^2 = \sum_{i=1}^d X_{1i}^2X_{2i}^2 + \sum_{1\leq i\not=j\leq d} X_{1i}X_{2i} X_{1j}X_{2j}.
				\end{align*}
				If $\bm X_1$ and $\bm X_2$ are $d$-variate i.i.d. random variables with mean ${\bf 0}$ and dispersion matrix $\bm \Sigma=((\sigma_{ij}))$, expectations of these two parts are $a_d=\sum_{i=1}^{d} \sigma_{ii}^2$ and $c_d=\sum_{i =1}^{d}\sum_{j (\neq i)=1}^{d} \sigma_{ij}^2$, respectively. However, 
				for $(\bm X_1^\top\bm X_2^{\prime})^2$
				(and also for $({\bm X_1^\prime}^\top\bm X_2^{\prime})^2$), these two expected values are $b_d=(\sum_{i=1}^{d}\sigma_{ii})^2/d$ and $0$, respectively (see Lemma \ref{lem:expectations}). So, the differences turn out to be $v_d=a_d-b_d=\sum_{i=1}^{d}[\sigma_{ii}-(\frac{1}{d}\sum_{i=1}^{d} \sigma_{ii})]^2$ and $c_d$, respectively. While the first one measures the variation among the diagonal elements of $\bm \Sigma$, the second one tells us how different the off-diagonal elements of $\bm\Sigma$ are from $0$. In Example 3.1, where we get signal from the second part (i.e. $c_d>0$), the proposed tests worked well. However, in Example 3.2, we
				have no signal from the second part. So, the difference in cross-products serves as a noise and its order is higher ($\mathcal{O}_P(d^2)$) than the order of the signal obtained from the difference in square terms (i.e., $v_d = \mathcal{O}(d)$). Therefore, the proposed test could not have satisfactory performance. In Example 3.3 also, the difference in cross-products serves as a noise of order $\mathcal{O}_P(d^2)$, but unlike Example 3.2, here the signal $v_d$ is of the order $O(d^2)$. So, the power of the sign and runs tests increased with the dimension. However, from the above discussion, it is clear that if there is no signal from the off-diagonal part (i.e. $c_d=0$), our tests can perform better if we can get rid of this noise term involving cross-products. One possible option is to ignore the cross-product terms and construct the shortest covering path based on a different cost function. 
				We can consider an edge-weighted complete graph
				${\mathcal K}_{2n}$ on $2n$ vertices'
				$\bm Z_1,\bm Z_2,\ldots, \bm Z_{2n}$ }as before but use 
			${\widetilde \theta}(\bm Z_i,\bm Z_j)=\exp\{-\frac{1}{d}\sum_{q=1}^{d}Z_{iq}^2Z_{jq}^2\}$ as the cost of the edge joining $\bm Z_i$ and $\bm Z_j$ ($1\le i<j\le 2n$).
			Using this cost function, we can construct the shortest covering path as before and define the sign and rank vectors along that path as in Section \ref{sec:definition}. So, we look for new sign and rank vectors ${\bm {\widetilde S}}$  and ${\bm {\widetilde R}}= {\bm {\widetilde \Pi}}^{-1}$, where
			\begin{align}\label{eq:optimization-1}
				\vspace{-0.05in}
				({\bm {\widetilde S}},{\bm {\widetilde \Pi}}) = \argmin_{\substack{\bm s\in\{0,1\}^n\\\bm\pi\in \mathcal{S}_n}} \Big[\sum_{i=1}^{n-1} {\widetilde \theta}(\bm Y_{s_{\pi_i},\pi_i},\bm Y_{s_{\pi_{i+1}},\pi_{i+1}})\Big].
			\end{align}
			
			\vspace{-0.05in}
			\noindent
			It is easy to check that the null distribution of $({\bm {\widetilde S}}.{\bm{\widetilde{R}}})$  matches with that of $(\bm S, \bm R)$ (see part (a) of Theorem \ref{thm:modified-rank-tests}). In fact Theorem \ref{thm:sign-rank-dist} holds for ${\bm {\widetilde S}}=({\widetilde S}_1,{\widetilde S}_2,\ldots,{\widetilde S}_n)$, ${\bm{\widetilde{R}}}=({\widetilde R}_1,{\widetilde R}_2,\ldots,{\widetilde R}_n)$ and ${\bm {\widetilde \Pi}}=({\widetilde \pi}_1,{\widetilde \pi}_2,\ldots,{\widetilde \pi}_n)$. So, the null distributions of the corresponding sign statistic ${\widetilde T}_S=\sum_{i=1}^{n} {\widetilde S}_i$ (or any linear rank statistic ${\widetilde T}_{LR}=\sum_{i=1}^{n} {\widetilde S}_i a({\widetilde R}_i)=\sum_{i=1}^{n} {\widetilde S}_{{\widetilde \pi}_i} a(i)$) and runs statistic ${\widetilde T}_R=1
			+\sum_{i=1}^{n-1} \mathrm I({\widetilde S}_{\tilde{\pi}_i}\neq {\widetilde S}_{\tilde{\pi}_{i+1}})$ match with the corresponding univariate statistics, and the cut-offs can be obtained as before from the statistical tables
			available for the univariate sign (or linear rank)
			and runs tests.
			
			\begin{figure}[h]
				\centering
				\begin{tikzpicture}
					\begin{axis}[xmin = 1, xmax = 10, ymin = 0, ymax = 1, xlabel = {$\log_2(d)$}, ylabel = {Power Estimates}, title = {\bf Example \ref{exa:first}}]
						\addplot[color = red,   mark = *, step = 1cm,very thin]coordinates{(1,0.028)(2,0.039)(3,0.041)(4,0.029)(5,0.036)(6,0.028)(7,0.032)(8,0.038)(9,0.038)(10,0.033)};
						
						
						\addplot[color = purple,   mark = diamond*, step = 1cm,very thin]coordinates{(1,0.036)(2,0.046)(3,0.042)(4,0.047)(5,0.047)(6,0.038)(7,0.038)(8,0.035)(9,0.047)(10,0.039)};
						
					\end{axis}
				\end{tikzpicture}
				\begin{tikzpicture}
					\begin{axis}[xmin = 1, xmax = 10, ymin = 0, ymax = 1, xlabel = {$\log_2(d)$}, ylabel = {Power Estimates}, title = {\bf Example \ref{exa:second}}]
						\addplot[color = red,   mark = *, step = 1cm,very thin]coordinates{(1,0.103)(2,0.131)(3,0.173)(4,0.246)(5,0.408)(6,0.62)(7,0.859)(8,0.97)(9,1)(10,1)};
						
						
						\addplot[color = purple,   mark = diamond*, step = 1cm,very thin]coordinates{(1,0.076)(2,0.107)(3,0.144)(4,0.233)(5,0.38)(6,0.588)(7,0.803)(8,0.965)(9,0.999)(10,1)};
						
					\end{axis}
				\end{tikzpicture}
				\begin{tikzpicture}
					\begin{axis}[xmin = 1, xmax = 10, ymin = 0, ymax = 1, xlabel = {$\log_2(d)$}, ylabel = {Power Estimates}, title = {\bf Example \ref{exa:third}}]
						\addplot[color = red,   mark = *, step = 1cm,very thin]coordinates{(1,0.092)(2,0.437)(3,0.855)(4,0.988)(5,0.999)(6,1)(7,1)(8,1)(9,1)(10,1)};
						
						
						\addplot[color = purple,   mark = diamond*, step = 1cm,very thin]coordinates{(1,0.088)(2,0.349)(3,0.847)(4,0.987)(5,1)(6,1)(7,1)(8,1)(9,1)(10,1)};
						
					\end{axis}
				\end{tikzpicture}
				\caption{Observed power of the tests based on ${\widetilde T}_S$ (\tikzsymbol[circle]{minimum width=2pt,fill=red}) 
					and ${\widetilde T}_R$  (\tikzsymbol[diamond]{minimum width=2pt,fill=purple}) when $50$ observations are generated from the $d$-variate normal distributions (with $d=2^i$ for $i=1,2,\ldots,10$) considered in Examples 3.1-3.3.}
				\label{fig:sim-2}
			\end{figure}
			
			Figure \ref{fig:sim-2} shows the performance of the sign and runs tests based on ${\widetilde T}_S$ and
			${\widetilde T}_R$ in Examples 3.1-3.3. In Example 3.3, where the signal comes from the diagonal part, they performed better than our previous sign and runs tests. They performed well also in Example 3.2, where the previous sign and runs tests did not have satisfactory performance.
			However, in Example 3.1, where we do have no signals from the diagonal part, this new sign and runs tests had poor performance. 
			
			This high-dimensional behavior of the tests based of ${\widetilde T}_S$ and ${\widetilde T}_R$ can be further explained by part (b) of Theorem~\ref{thm:modified-rank-tests}. But for that we need the following technical assumption.

				\begin{itemize}
					\item[(A3)] Let $\bm X_1,\bm X_2$ be two independent copies of $\bm X \sim \Pr$ and $\bm X_1^\prime,\bm X_2^\prime$ be their spherically symmetric variants. There exists an $\alpha>0$ such that for $W = d^{-\alpha}~\sum_{i=1}^d (\bm X_{1})_j^2(\bm X_{2})_j^2$, $d^{-\alpha}~ \sum_{i=1}^d (\bm X_{1})_j^2(\bm X_{2}^\prime)_j^2$ and $d^{-\alpha}~\sum_{i=1}^d (\bm X_{1}^\prime)_j^2(\bm X_{2}^\prime)_j^2$, $\big| W - \E[W]\big|\stackrel{P}{\rightarrow} 0$ as $d\rightarrow \infty$, and at least one of the limit is non-zero. 
				\end{itemize}
				Similar assumptions were also considered by \cite{hall2005geometric, jung2009pca, yata2012, sarkar2019perfect, banerjee2025high, dutta2016multiscale} for studying high dimensional behaviour of different statistical methods. This assumption is satisfied in Examples 3.2 and 3.3 for $\alpha=1$ and $2$, respectively. Under this assumption we have the following result. 				
				\begin{thm}\label{thm:modified-rank-tests}
					Let $\bm X_1,\bm X_2,\ldots,\bm X_n$ be $n$ independent realizations $\bm X$ following a $d$-dimensional continuous distribution $\Pr$. Then, we have the following results.
					\begin{itemize}
						\item[(a)] If $\Pr$ is spherically symmetric, ${\bm {\widetilde S}}\sim$ Unif$(\{0,1\}^n)$, ${\bm {\widetilde R}} \sim$ Unif$({\mathcal S}_n)$, and they are independent.
						\item[(b)] Let ${\bm \Sigma}$ denote the covariance matrix of $\bm X$ and ${\bf D} = \textnormal{diag}({\bm \Sigma})$. Suppose that $\Pr$ is not spherically symmetric and it satisfies Assumption (A3). If
						\begin{align}\label{eq:magic}
							\liminf_{d\rightarrow\infty}\left\{\frac{1}{d^\alpha}\textnormal{Tr}({\bf D}^2) - \frac{1}{d^{1+\alpha}}\Big(\textnormal{Tr}({\bf D})\Big)^2\right\}>0,
						\end{align} then the sign vector $\widetilde{\bm {S}}$ converges to ${\bf 1}_n = (1,1,\ldots,1)$ in probability as $d$ diverges to infinity. 
					\end{itemize} 
			\end{thm}
			Note that using Jensen's inequality we have {$d~\textnormal{Tr}({\bf D}^2) \geq (\textnormal{Tr}({\bf D}))^2$, or $\frac{1}{d^\alpha}\textnormal{Tr}({\bf D}^2) - \frac{1}{d^{1+\alpha}}\Big(\textnormal{Tr}({\bf D})\Big)^2\geq 0$, where the equality holds if and only if {all diagonal elements of ${\bf D}$} are equal. So, the condition \eqref{eq:magic} holds when the variance {among the diagonal elements of ${\bf D}$} remains bounded away from $0$ as the dimension grows to infinity. This condition is satisfied in Examples 3.2 and 3.3 but not in Example 3.1. This explains the difference in the performance of the tests based on ${\widetilde T}_S$ and ${\widetilde T}_R$ in these three examples.
			
			Therefore, in the HDLSS setup, while the tests based on $T_S$ and $T_R$ may fail to detect weak signals in the diagonal part (for instance, in Example 3.2, where all diagonal elements of the covariance matrix are not same but the sphericity condition holds), those based on ${\widetilde T}_S$ and ${\widetilde T}_R$ cannot detect signals present in the off-diagonal part (for instance in Example 3.1, where we have non-zero off-diagonals elements). To overcome these limitations, one can think of combining the strengths of these two types of tests. For instance, we can use $T^M_{S} = \max\{T_{S},{\widetilde T}_{S}\}$ and $T^M_{R} = \min\{T_{R},{\widetilde T}_{R}\}$ as the test statistics to boost the performance of the sign and runs tests for a larger class of alternatives. Naturally, we reject the null hypothesis of spherical symmetry for large values of the modified sign statistic $T^M_{S}$ or small values of the modified runs statistic $ T^M_{R}$. The cut-offs can be computed using the an appropriate resampling method, but
			to keep our tests simple and computationally efficient, here  we use Bonferroni's method for calibration. Note that under the null hypothesis, $T_{S}$ and ${\widetilde T}_{S}$ (respectively, $T_{R}$ and ${\widetilde T}_{R}$) are identically distributed. So, the cut-off for $T^M_{S}$ (respectively, $T^M_{R}$) can be easily obtained from statistical tables and packages. One can use tests based on 
			$T^S_{S} = T_{S} + {\widetilde T}_{S}$ and $T^S_{R} = T_{R} + {\widetilde T}_{R}$ as well, but in those cases, Bonferroni's method cannot be used, and the user needs to go for the resampling algorithm for calibration. So, these tests become computationally expensive and we do not consider them in this article.
			
			Figure \ref{fig:sim-3} shows the performance of the modified sign and runs tests in Examples 3.1-3.3. In all three examples, they had excellent performance. So, it seems reasonable to use $T^M_S$ or $T^M_R$ as the test statistic.
			The following theorem establishes the consistency of the resulting tests in the HDLSS asymptotic regime.
			
			\begin{figure}[h]
				\centering
				\begin{tikzpicture}
					\begin{axis}[xmin = 1, xmax = 10, ymin = 0, ymax = 1, xlabel = {$\log_2(d)$}, ylabel = {Power Estimates}, title = {\bf Example \ref{exa:first}}]
						\addplot[color = red,   mark = *, step = 1cm,very thin]coordinates{(1,0.159)(2,0.53)(3,0.864)(4,0.98)(5,0.998)(6,1)(7,1)(8,1)(9,1)(10,1)};
						
						
						\addplot[color = purple,   mark = diamond*, step = 1cm,very thin]coordinates{(1,0.106)(2,0.454)(3,0.839)(4,0.97)(5,0.999)(6,1)(7,1)(8,1)(9,1)(10,1)};
						
					\end{axis}
				\end{tikzpicture}
				\begin{tikzpicture}
					\begin{axis}[xmin = 1, xmax = 10, ymin = 0, ymax = 1, xlabel = {$\log_2(d)$}, ylabel = {Power Estimates}, title = {\bf Example \ref{exa:second}}]
						\addplot[color = red,   mark = *, step = 1cm,very thin]coordinates{(1,0.1)(2,0.134)(3,0.152)(4,0.213)(5,0.341)(6,0.567)(7,0.801)(8,0.956)(9,1)(10,1)};
						
						
						\addplot[color = purple,   mark = diamond*, step = 1cm,very thin]coordinates{(1,0.071)(2,0.107)(3,0.145)(4,0.194)(5,0.311)(6,0.47)(7,0.737)(8,0.949)(9,0.998)(10,1)};
						
					\end{axis}
				\end{tikzpicture}
				\begin{tikzpicture}
					\begin{axis}[xmin = 1, xmax = 10, ymin = 0, ymax = 1, xlabel = {$\log_2(d)$}, ylabel = {Power Estimates}, title = {\bf Example \ref{exa:third}}]
						\addplot[color = red,   mark = *, step = 1cm,very thin]coordinates{(1,0.089)(2,0.456)(3,0.882)(4,0.99)(5,1)(6,1)(7,1)(8,1)(9,1)(10,1)};
						
						
						\addplot[color = purple,   mark = diamond*, step = 1cm,very thin]coordinates{(1,0.079)(2,0.369)(3,0.854)(4,0.992)(5,1)(6,1)(7,1)(8,1)(9,1)(10,1)};
						
					\end{axis}
				\end{tikzpicture}
				\caption{Observed power of the Modified sign test (\tikzsymbol[circle]{minimum width=2pt,fill=red}) 
					and the Modified runs test (\tikzsymbol[diamond]{minimum width=2pt,fill=purple}) when $50$ observations are generated from the $d$-variate normal distributions (with $d=2^i$ for $i=1,2,\ldots,10$) in Examples 3.1-3.3.}
				\vspace{-0.05in}
				\label{fig:sim-3}
			\end{figure}
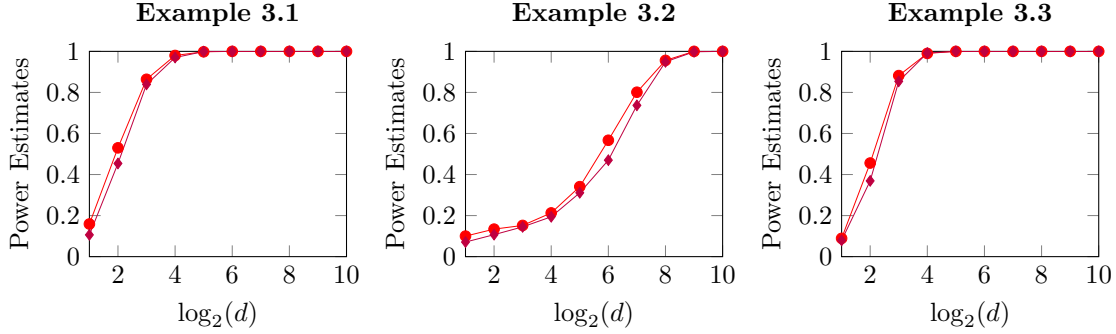
			
			{\begin{thm}\label{thm:HDLSS-boosted-tests}
					Let $\bm X_1,\bm X_2,\ldots,\bm X_n$ be $n$ independent realizations of $\bm X$, which follows a $d$-dimensional distribution $\Pr$ which is not spherically symmetric. Let ${\bm \Sigma}$ be the covariance matrix of $\bm X$. Define ${\bm D} = \textnormal{diag}({\bm \Sigma})$. Also assume that one of the following conditions.
					\begin{itemize}
						\item[(a)] for all $M>0$, $\P\left[\frac{d(\bm X_1^\top\bm X_2)^2}{\|\bm X_1\|^2\|\bm X_2\|^2}>M\right]\rightarrow 1$ as $d$ diverges to infinity,
						\item[(b)] Assumption (A3) holds and  $\liminf_{d\rightarrow\infty}\left\{\frac{1}{d^\alpha}\textnormal{Tr}({\bf D}^2) - \frac{1}{d^{1+\alpha}}\Big(\textnormal{Tr}({\bf D})\Big)^2\right\}>0$.
					\end{itemize}
					Then $T^M_{S}$ converges in probability to its maximum value $n$ and $T^M_{R}$ converges in probability to its minimum value $1$ as $d$ diverges to infinity. 
			\end{thm}}
			
			\begin{rem}
				Under $H_0$, we have $P(T^M_S \ge n) \le P(T_S = n)+P({\widetilde T}_S = n)= 1/2^{n-1}$. So, for a any fixed level $\alpha$ ($0<\alpha<1$), if $n$ exceeds 
				$-\log_2(\alpha)+1$, 
				the power of the modified sign test converges to $1$
				as the dimension increases. Under $H_0$, we also have $P(T^M_R \le 1) \le P(T_R = 1)+P({\widetilde T}_R = 1)= 1/2^{n-2}$. So, for the consistency of the modified runs test in the HDLSS asymptotic regime, we need $n>-\log_2(\alpha)+2$.
			\end{rem}
			
			From our discussion in this section, it is clear that if the covariance matrix of the underlying distribution differs from a scalar multiple of the identity matrix, our modified sign and runs tests may work well in the HDLSS regime. Now, one may be curious to know what happens if the underlying distribution is not spherically symmetric but the covariance matrix is a scalar multiple of the identity matrix, for instance if the coordinate variables are independent and identically distributed (i.i.d.)  but the distribution is not spherical. To investigate it, we consider two simple examples.   
			
			\begin{exa}\label{exa:iid-uniform}
				Here we deal with a uniform distribution on the $d$-dimensional hypercube $[-1,1]^d$.
			\end{exa}
			
			\begin{exa}\label{exa:iid-coordinates}
				We generate observations on $\bm X = (X_1,\ldots, X_d)$, where the coordinate variables are independent and identically distributed as Laplace ($0,1$) variates.
			\end{exa}
			
			For each of these examples, we consider a sample of size $50$ and use six different value of $d$ ($d=2^i$ for $i=1,\ldots, 6$). Each experiment is carried out $100$ times to compute the empirical powers of different tests and they are reported in Figure \ref{fig:sim-3.1}. 
			
			\begin{figure}[h]
				\centering
				\begin{tikzpicture}
					\begin{axis}[xmin = 1, xmax = 6, ymin = 0, ymax = 1, xlabel = {$\log_2(d)$}, ylabel = {Power Estimates}, title = {\bf Example \ref{exa:iid-uniform}}]
						
						\addplot[color = blue,   mark = *, step = 1cm,very thin]coordinates{(1,0.087)(2,0.085)(3,0.078)(4,0.073)(5,0.061)(6,0.069)};
						
						\addplot[color = blue,   mark = diamond*, step = 1cm,very thin]coordinates{(1,0.066)(2,0.075)(3,0.06)(4,0.05)(5,0.041)(6,0.056)};
						
						\addplot[color = red,   mark = *, step = 1cm,very thin]coordinates{(1,0.029)(2,0.028)(3,0.027)(4,0.018)(5,0.017)(6,0.021)};
						
						\addplot[color = red,   mark = diamond*, step = 1cm,very thin]coordinates{(1,0.085)(2,0.191)(3,0.371)(4,0.596)(5,0.768)(6,0.848)};

						
					\end{axis}
				\end{tikzpicture}
				\begin{tikzpicture}
					\begin{axis}[xmin = 1, xmax = 6, ymin = 0, ymax = 1, xlabel = {$\log_2(d)$}, ylabel = {Power Estimates}, title = {\bf Example \ref{exa:iid-coordinates}}]
						
						\addplot[color = blue,   mark = *, step = 1cm,very thin]coordinates{(1,0.077)(2,0.072)(3,0.09)(4,0.075)(5,0.094)(6,0.075)};
						
						\addplot[color = blue,   mark = diamond*, step = 1cm,very thin]coordinates{(1,0.063)(2,0.04)(3,0.061)(4,0.048)(5,0.045)(6,0.045)};
						
						\addplot[color = red,   mark = *, step = 1cm,very thin]coordinates{(1,0.081)(2,0.087)(3,0.158)(4,0.263)(5,0.476)(6,0.651)};
						
						\addplot[color = red,   mark = diamond*, step = 1cm,very thin]coordinates{(1,0.06)(2,0.051)(3,0.072)(4,0.072)(5,0.135)(6,0.183)};

						
					\end{axis}
				\end{tikzpicture}
				\caption{Power of the sign test (\tikzsymbol[circle]{minimum width=2pt,fill=blue}), runs test, (\tikzsymbol[diamond]{minimum width=2pt,fill=blue}), modified sign test (\tikzsymbol[circle]{minimum width=2pt,fill=red}), and modified runs test (\tikzsymbol[diamond]{minimum width=2pt,fill=red}) in Examples \ref{exa:iid-uniform}
					and \ref{exa:iid-coordinates} for $n=50$ and $d = 2^i$ for $i=1,\ldots, 6$. }
				\label{fig:sim-3.1}
			\end{figure}
			
     		The sign test and the runs test discussed in the previous section have very poor performance in these examples. In Example \ref{exa:iid-uniform}, the modified sign test also failed but the modified runs test had an excellent performance. However, in Example \ref{exa:iid-coordinates}, the modified sign test performed well. The power of the modified runs test also increased with the dimension though the rate of increment was relatively slower. This result shows that the modification helped to improve the performance of the tests substantially in some cases.
			
			To understand the reason behind this behavior of modified sign and runs tests, we look at the distributions of $\widetilde{\theta}(\bm X_i,\bm X_j),$ $\widetilde{\theta}(\bm X_i, \bm X_j^\prime)$ and $\widetilde{\theta}(\bm X_i^\prime, \bm X_j^\prime)$s, which are shown in Figure \ref{fig:density-estimate-lp-1} for $d=1000$. From our previous discussions at the beginning of this section, one can show that under Assumption (A3), all of them converge to the same value as $d$ increases. But Figure \ref{fig:density-estimate-lp-1} shows us an interesting phenomenon. Note that the left tails of these distributions play an important role in our methods as the shortest covering path construction algorithm starts the pair of observations corresponding to the smallest value of ${\widetilde \theta}$ and then the other pairs are joined subsequently. Figure \ref{fig:density-estimate-lp-1} shows that in  Example \ref{exa:iid-uniform}, we are likely to start with a pair of the form $(\bm X_i^\prime,\bm X_j^\prime)$, and subsequently join more observations from the set $\{\bm X_1^\prime, \bm X_2^\prime,\ldots,\bm X_n^\prime\}$. As a result,
			both the runs statistic and the sign statistic take 
			smaller values. So, the runs test based on ${\widetilde T}_R$ and hence the modified runs test work well, but the sign test that rejects $H_0$ for larger values of ${\widetilde T}_S$ performs poorly, and so does the modified sign test. But we observed an opposite picture in Example \ref{exa:iid-coordinates}. Here the shortest covering path is likely to start with a pair of the form $(\bm X_i,\bm X_j)$, and we are expected to have a dominance of the original observations on the path. So, the sign test based on ${\widetilde T}_S$ and hence the modified sign test performed well. The power of the modified runs test also showed an increasing trend but its performance was relatively inferior compared to the modified sign test.
			
			\begin{figure}[h]
				\centering
				\setlength{\tabcolsep}{-2pt}
				\begin{tabular}{cc}
					(a) Example 4.1 & (b) Example 4.2\\
					\includegraphics[width=3.25in,height=2.5in]{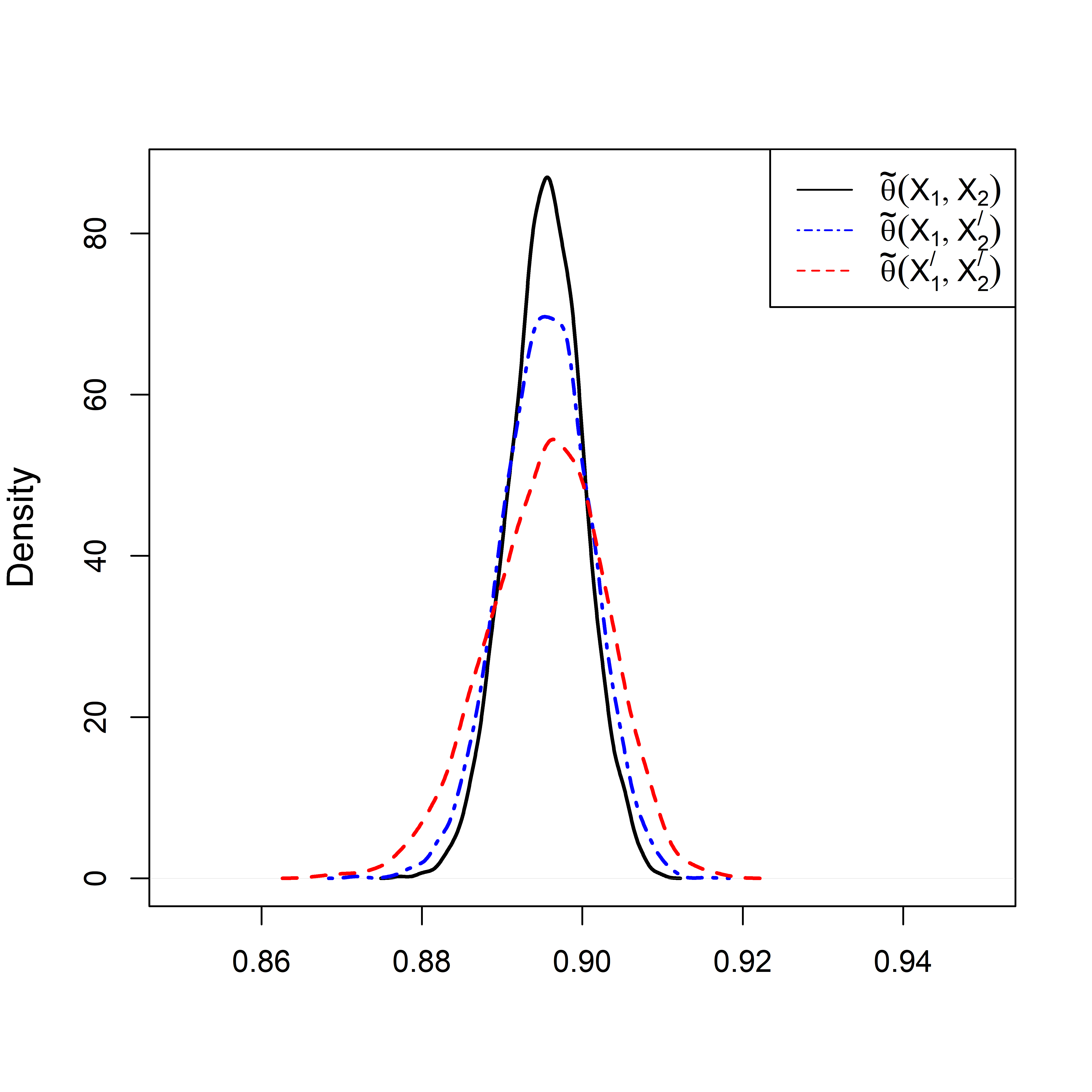} &
					\includegraphics[width=3.25in,height=2.5in]{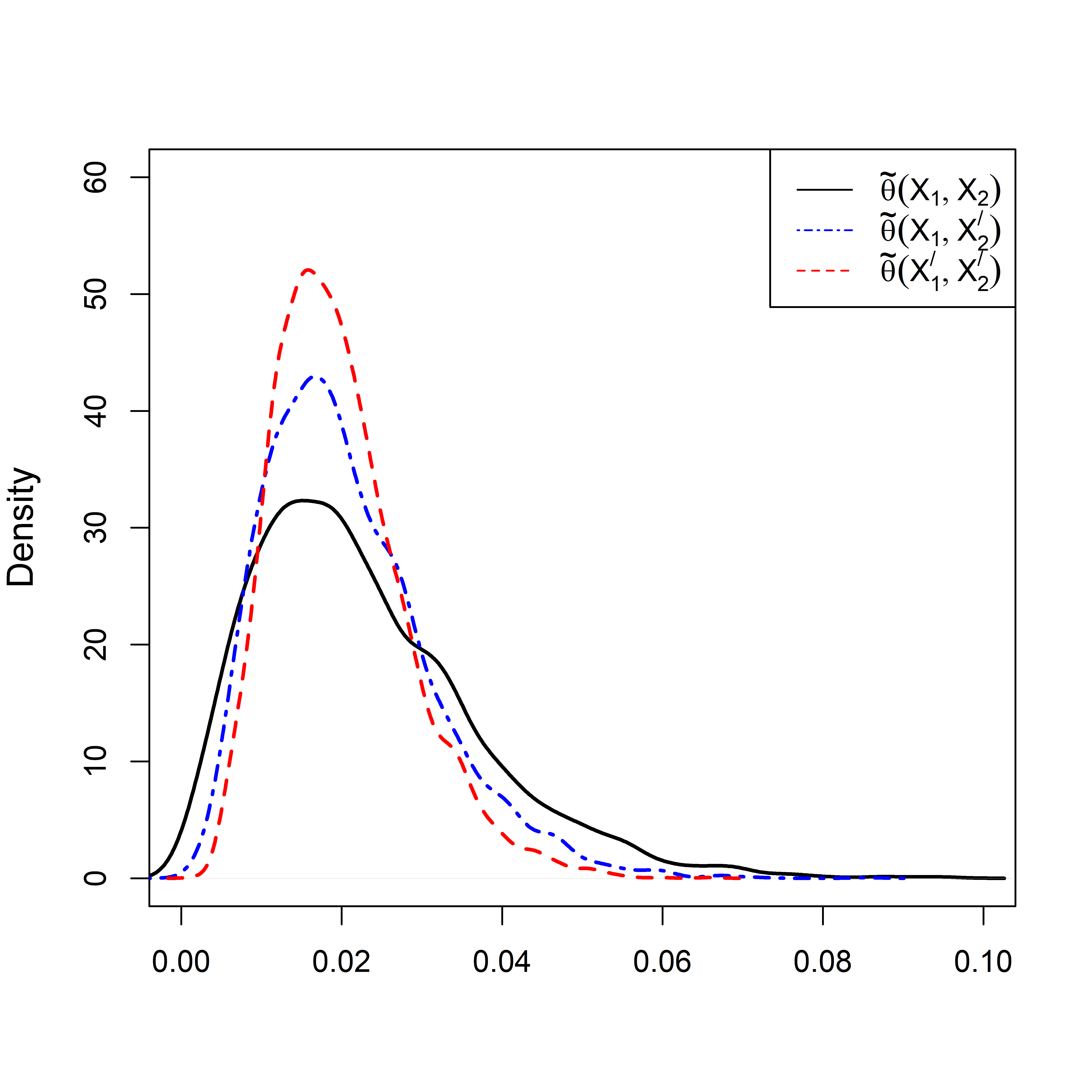}
				\end{tabular}
				\vspace{-0.25in}
				\caption{The density estimates of  $\tilde{\theta}(.,.)$ when $\bm X_1,\ldots, \bm X_n$ are generated independently as described in Examples \ref{exa:iid-uniform} and
				 \ref{exa:iid-coordinates}.}
				\label{fig:density-estimate-lp-1}
			\end{figure}
			
			However, powers of all these tests show increasing trends when the sample size also increases with the dimension (see Figure \ref{fig:sim-3.2}). 
			In both examples, modified sign and runs tests had much better performance than what we observed in Figure \ref{fig:sim-3.1}. We also observed the same for the sign and runs tests based on $T_S$ and $T_R$. We have already studied the HDHSS behavior of these two tests in Section 3. Now, we briefly investigate the HDHSS behavior of our modified tests.
			
			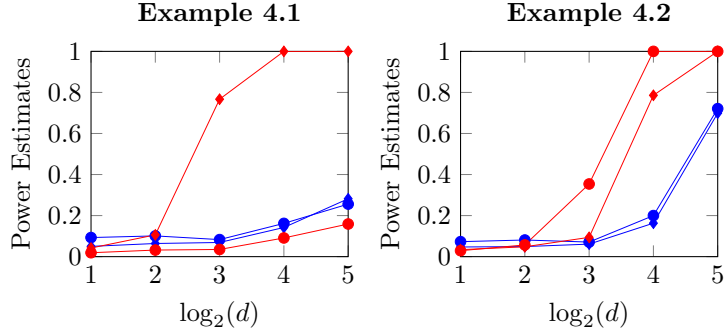
\begin{figure}[h]
				\centering
				\begin{tikzpicture}
					\begin{axis}[xmin = 1, xmax = 5, ymin = 0, ymax = 1, xlabel = {$\log_2(d)$}, ylabel = {Power Estimates}, title = {\bf Example \ref{exa:iid-uniform}}]
						
						\addplot[color = blue,   mark = *, step = 1cm,very thin]coordinates{(1,0.093)(2,0.101)(3,0.083)(4,0.162)(5,0.256)};
						
						\addplot[color = blue,   mark = diamond*, step = 1cm,very thin]coordinates{(1,0.05)(2,0.064)(3,0.069)(4,0.142)(5,0.282)};
						
						\addplot[color = red,   mark = *, step = 1cm,very thin]coordinates{(1,0.019)(2,0.032)(3,0.035)(4,0.091)(5,0.159)};
						
						\addplot[color = red,   mark = diamond*, step = 1cm,very thin]coordinates{(1,0.042)(2,0.106)(3,0.767)(4,1)(5,1)};

						
					\end{axis}
				\end{tikzpicture}
				\begin{tikzpicture}
					\begin{axis}[xmin = 1, xmax = 5, ymin = 0, ymax = 1, xlabel = {$\log_2(d)$}, ylabel = {Power Estimates}, title = {\bf Example \ref{exa:iid-coordinates}}]
						
						\addplot[color = blue,   mark = *, step = 1cm,very thin]coordinates{(1,0.073)(2,0.081)(3,0.071)(4,0.2)(5,0.721)};
						
						\addplot[color = blue,   mark = diamond*, step = 1cm,very thin]coordinates{(1,0.047)(2,0.049)(3,0.061)(4,0.163)(5,0.699)};
						
						\addplot[color = red,   mark = *, step = 1cm,very thin]coordinates{(1,0.029)(2,0.057)(3,0.354)(4,1)(5,1)};
						
						\addplot[color = red,   mark = diamond*, step = 1cm,very thin]coordinates{(1,0.034)(2,0.048)(3,0.094)(4,0.786)(5,1)};

						
					\end{axis}
				\end{tikzpicture}
				\caption{Power of the sign test (\tikzsymbol[circle]{minimum width=2pt,fill=blue}), runs test, (\tikzsymbol[diamond]{minimum width=2pt,fill=blue}), modified sign test (\tikzsymbol[circle]{minimum width=2pt,fill=red}), and modified runs test (\tikzsymbol[diamond]{minimum width=2pt,fill=red}) in Examples \ref{exa:iid-uniform} and \ref{exa:iid-coordinates} for $n= d^2 + 20 $ and $d = 2^i$ for $i=1,\ldots, 5$. }
				\label{fig:sim-3.2}
			\end{figure}

			Note that the null distributions of the test based on ${\widetilde T}_S$ and ${\widetilde T}_R$ are identical to those based on ${T}_S$ and ${T}_R$, respectively. Theorem \ref{thm:limit-null-distribution} and Theorem \ref{thm:large-dist-runs} give the asymptotic null distributions of ${T}_S$ and ${T}_R$ when the dimension and the sample size both tend to infinity (these results hold even when the dimension is fixed and the sample size diverges
			to infinity). One can show that  ${\widetilde T}_S$  and ${\widetilde T}_R$ have the same asymptotic behavior. However, studying the asymptotic behavior of the modified tests turns out to be a bit complicated due to the dependence between these two sign statistics and that between two runs statistics. The following theorem summarizes the asymptotic null behavior of the modified test statistics.
			
			\begin{thm}
				\label{thm:HDHSS-modified-tests-null}
				Let $\bm X_1,\ldots, \bm X_n$ be independent realizations of a $d$-dimensional random vector $\bm X\sim \Pr$. Assume that $\Pr$ is spherically symmetric, then, we have the following results:
				\begin{itemize}
					\item[(a)] Assume that $\sigma_s^2 = \lim_{d\rightarrow\infty} \P[S_1=1; \tilde{S}_1=1]$ exists. Then, $n$ and $d$ both grow to infinity, we have
    				\begin{align*}
						\frac{\max\{T_S, \tilde{T}_S\} - \frac{n}{2}}{\sqrt{n}} \stackrel{D}{\rightarrow} \max\{Z_1,Z_2\}.
					\end{align*}
				    where $(Z_1,Z_2)$ follows a bivariate normal distribution with mean at the origin, equal variance $\frac{1}{4}$ and covariance $\big(\sigma_s^2-\frac{1}{4}\big)$.
					
					\item[(b)] Assume that $\sigma_r^2 = \lim_{d\rightarrow\infty} \textnormal{Cov}\big(\mathrm{I}\{S_{\pi_1}\not = S_{\pi_2}\}, \mathrm{I}\{\tilde{S}_{\tilde{\pi}_1} \not = \tilde{S}_{\tilde{\pi}_2}\}\big)$ exists. Then,  $n$ and $d$ both grow to infinity, we have
					\begin{align*}
						 \frac{\min\{T_R, \tilde{T}_R\} - \frac{n+1}{2}}{\sqrt{n}} \stackrel{D}{\rightarrow} \min\{Z_1^\prime,Z_2^\prime\}
					\end{align*}
					$(Z_1^\prime, Z_2^\prime)$ also follows a bivariate normal distribution with same marginals as $(Z_1,Z_2)$, but its covariance is $\big(\sigma_r^2-\frac{1}{4}\big)$.
				\end{itemize}
			\end{thm}
			
			While the cut-off for the tests based on $T_S$ or ${\widetilde T}_S$ (note that they have the same cut-off) can be computed easily, finding the cut-off for the modified sign test is difficult to obtain from this asymptotic null distribution unless one finds a consistent estimator for the covariance. So, here also, we use the Bonferroni's method for implementing the modified sign test. The same strategy is used for the modified runs test as well.
			From the above discussion it is quite clear that if the test based on $T_S$ (respectively, $T_R$) or that based on ${\widetilde T}_S$ (respectively, ${\widetilde T}_R$), at least one of them is consistent, the modified sign test (respectively, the modified runs test) turns out to be consistent in the HDHSS set up. This result is stated in the following theorem.
			
			\begin{thm}
				\label{thm:HDHSS-modified-tests-alt}
				Let $\bm X_1,\ldots, \bm X_n$ be independent copies of a $d$-dimensional random vector $\bm X\sim \Pr$. If $\Pr$ is not spherically symmetric, we have the following results as $n$ and $d$ both diverge to infinity.
				\begin{itemize}
					\item[(a)] If $\max\big\{\liminf\limits_{n,d \rightarrow \infty} E[T_{S}/n], \liminf\limits_{n,d \rightarrow \infty} \E[\tilde{T}_{S}/n]\big\}>0.5$ the power of the modified sign test converges to $1$.
					\item[(b)] If $\min\big\{\limsup\limits_{n,d \rightarrow \infty} \E[T_R], \limsup\limits_{n,d \rightarrow \infty} \E[\tilde{T}_R/n]\big\}<0.5$, the power of the modified runs test converges $1$.
				\end{itemize}
			\end{thm}   

			\section{Analysis of simulated and real data sets}
			\label{sec:empirical-analysis}
			In this section, we analyze some high-dimensional simulated and real data sets to compare the empirical performance of our proposed tests with the test based on optimal transport \cite{huang2023multivariate}, and that based on density functions \cite{diks1999}. 
			 These two tests are consistent in the classical asymptotic regime under general alternatives, computationally efficient and can be conveniently used even when the dimension is larger than the sample size. They will be referred to as the OT test and the DT test, 
			 respectively. Henceforth by sign and runs tests, we shall refer to the tests proposed in Section \ref{sec:definition} and the modified versions considered in Section \ref{sec:improvements} will be referred to as the modified sign and modified runs tests, respectively.  Throughout this section, all test are considered to have the nominal level $\alpha = 0.05$. Note that the OT test 
			 is distribution free under $H_0$. However, following the suggestion of the authors, we use the asymptotic distribution of the test statistic for calibration. For the DT test, we use a resampling algorithm as proposed in that article, where the cut-off is computed based on 500 independent random iterations of the resampling algorithm. For the DT test, we also introduce the bandwidth $(0.25)^2 \hat\sigma_0^2$ (where $\hat\sigma_0^2 = \frac{1}{n(n-1)} \sum_{1\leq i< j\leq n} \|\bm X_i-\bm X_j\|^2$) in the test statistic for scale adjustment in higher dimensions. Without that adjustment, it performs poorly for high dimensional data. Each experiment is repeated 1000 times to compute the power of the tests by the proportion of times they reject $H_0$. The R codes of all these tests are available in the supplementary material.
			
			
			First we consider some examples (see Examples 5.1-5.6) involving high dimension, low sample size data. In each of these examples, we generate samples of size 50 and carry out our experiment for $10$ different choices $d$ ($d=2^i$ for $i=1,2,\ldots,10$) as before. In Examples \ref{exa:hd-1} and \ref{exa:hd-1A}, we deal with elliptic distributions with equi-correlated structure. 
			Example \ref{exa:hd-1} is same as Example \ref{exa:first}, and in Example \ref{exa:hd-1A}, we replace the normal distribution by Cauchy distribution. Descriptions of these two examples are given below.
			
			\begin{exa}\label{exa:hd-1} We
				consider a $d$-variate normal distribution with the location at the origin and the scatter matrix $\bm\Sigma = 0.4 \mathrm{\bf I}_d + 0.6 \mathrm{\bf J}_d$, where $\mathrm{\bf I}_d$ is the $d\times d$ identity matrix and $\mathrm{\bf J}_d$ is the $d\times d$ matrix with all entries equal to one.
						\end{exa}
					
			\begin{exa}\label{exa:hd-1A} Observations are generated from
				a $d$-variate Cauchy distribution with the same location and scatter matrix as in Example \ref{exa:hd-1}.			\end{exa}

			\begin{figure}[h]
				\vspace{-0.1in}
				\centering
				\begin{tikzpicture}[]
					\begin{axis}[xmin = 1, xmax = 10, ymin = 0, ymax = 1, xlabel = {$\log_2(d)$}, ylabel = {Power Estimates}, title = {\bf Example \ref{exa:hd-1}}]
						
						\addplot[color = ForestGreen,   mark = square*, mark size = 2pt, step = 1cm,very thin]coordinates{(1,0.39)(2,0.718)(3,0.824)(4,0.904)(5,0.983)(6,0.999)(7,1)(8,1)(9,1)(10,1)};

						\addplot[color = blue,   mark = *, mark size = 2pt, step = 1cm,very thin]coordinates{(1,0.309)(2,0.727)(3,0.955)(4,0.994)(5,0.999)(6,1)(7,1)(8,1)(9,1)(10,1)};
						
						\addplot[color = blue,   mark = diamond*, mark size = 2pt, step = 1cm,very thin]coordinates{(1,0.16)(2,0.547)(3,0.898)(4,0.985)(5,0.999)(6,1)(7,1)(8,1)(9,1)(10,1)};
						
						\addplot[color = red,   mark = *, mark size = 2pt, step = 1cm,very thin]coordinates{(1,0.159)(2,0.53)(3,0.864)(4,0.98)(5,0.998)(6,1)(7,1)(8,1)(9,1)(10,1)};
						
						\addplot[color = red,   mark = diamond*, mark size = 2pt, step = 1cm,very thin]coordinates{(1,0.106)(2,0.454)(3,0.839)(4,0.97)(5,0.999)(6,1)(7,1)(8,1)(9,1)(10,1)};
						
						
						\addplot[color = black,   mark = triangle*, mark size = 2pt, step = 1cm,very thin]coordinates{(1,0.054)(2,0.079)(3,0.105)(4,0.137)(5,0.166)(6,0.186)(7,0.236)(8,0.259)(9,0.281)(10,0.264)};
						
						\addplot[dashed, color = black,   mark = , step = 1cm,very thin]coordinates{(1,0.05)(2,0.05)(3,0.05)(4,0.05)(5,0.05)(6,0.05)(7,0.05)(8,0.05)(9,0.05)(10,0.05)};
						
					\end{axis}
				\end{tikzpicture}
				\begin{tikzpicture}[]
					\begin{axis}[xmin = 1, xmax = 10, ymin = 0, ymax = 1, xlabel = {$\log_2(d)$}, ylabel = {Power Estimates}, title = {\bf Example \ref{exa:hd-1A}}]

						\addplot[color = ForestGreen,   mark = square*, mark size = 2pt, step = 1cm,very thin]coordinates{(1,0.199)(2,0.532)(3,0.764)(4,0.854)(5,0.892)(6,0.913)(7,0.924)(8,0.937)(9,0.94)(10,0.945)};

						\addplot[color = blue,   mark = *, mark size = 2pt, step = 1cm,very thin]coordinates{(1,0.169)(2,0.432)(3,0.727)(4,0.947)(5,0.987)(6,1)(7,1)(8,1)(9,1)(10,1)};
						
						\addplot[color = blue,   mark = diamond*, mark size = 2pt, step = 1cm,very thin]coordinates{(1,0.117)(2,0.353)(3,0.748)(4,0.949)(5,0.994)(6,1)(7,1)(8,1)(9,1)(10,1)};
						
						\addplot[color = red,   mark = *, mark size = 2pt, step = 1cm,very thin]coordinates{(1,0.081)(2,0.248)(3,0.573)(4,0.876)(5,0.975)(6,0.997)(7,1)(8,1)(9,1)(10,1)};
						
						\addplot[color = red,   mark = diamond*, mark size = 2pt, step = 1cm,very thin]coordinates{(1,0.081)(2,0.285)(3,0.66)(4,0.917)(5,0.989)(6,0.999)(7,1)(8,1)(9,1)(10,1)};
						
						
						\addplot[color = black,   mark = triangle*, mark size = 2pt, step = 1cm,very thin]coordinates{(1,0.049)(2,0.053)(3,0.088)(4,0.117)(5,0.156)(6,0.19)(7,0.229)(8,0.246)(9,0.26)(10,0.274)};
						
						\addplot[dashed, color = black,   mark = , step = 1cm,very thin]coordinates{(1,0.05)(2,0.05)(3,0.05)(4,0.05)(5,0.05)(6,0.05)(7,0.05)(8,0.05)(9,0.05)(10,0.05)};
						
					\end{axis}
				\end{tikzpicture}
				\caption{Power of the sign test (\tikzsymbol[circle]{minimum width=2pt,fill=blue}), the runs test (\tikzsymbol[diamond]{minimum width=2pt,fill=blue}), the modified sign test (\tikzsymbol[circle]{minimum width=2pt,fill=red}), the modified runs test (\tikzsymbol[diamond]{minimum width=2pt,fill=red}), the OT test ($\textcolor{black}{\blacktriangle}$) and the DT test ($\textcolor{ForestGreen}{\blacksquare}$) in Examples \ref{exa:hd-1} and \ref{exa:hd-1A}. The dashed line indicates the nominal level $\alpha = 0.05$.}
				\label{fig:sim-6}
			\end{figure}
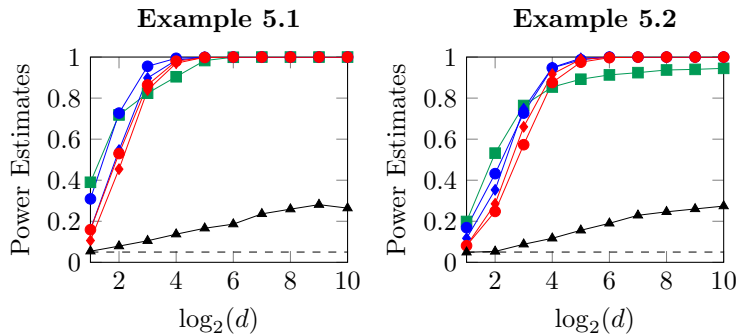
			
			Figure \ref{fig:sim-6} shows that in these two examples, the OT test had much lower power than all other tests considered here. The rest of the tests had comparable performance in Example~\ref{exa:hd-1}. They also had satisfactory performance in Example~\ref{exa:hd-1A}. However, in this example, our proposed tests performed better than the DT test in higher dimensions. 
			
			Next we consider two examples, where all off-diagonal elements of the dispersion matrix are $0$, but the diagonal elements 
			are not equal. Example \ref{exa:hd-2} is similar to Example \ref{exa:third} but here we have a weaker signal against spherical symmetry. Example \ref{exa:hd-3} is same as Example \ref{exa:second}, where the sign and runs tests had power close to the nominal level, but their modified versions had much better performance. Brief descriptions of these two examples are given below.
			
			\begin{exa}\label{exa:hd-2}
				We consider a $d$-dimensional normal distribution with the mean vector ${\bf 0}_d=(0,0,\ldots,0)$ and the diagonal covariance matrix $\bm \Sigma = \textnormal{diag}(d^{0.3},1,1,\ldots, 1)$.
			\end{exa}
			
			\begin{exa}\label{exa:hd-3} Here also, we consider a $d$-dimensional normal distribution with the mean vector ${\bf 0}_d$ and a diagonal covariance matrix $\bm \Sigma$, which has the first $[d/2]$ diagonal elements equal to $1$ and the rest equal to $2$.    
			\end{exa}
			
			\begin{figure}[h]
				\centering
				\begin{tikzpicture}[]
					\begin{axis}[xmin = 1, xmax = 10, ymin = 0, ymax = 1, xlabel = {$\log_2(d)$}, ylabel = {Power Estimates}, title = {\bf Example \ref{exa:hd-2}}]
						\addplot[color = ForestGreen,   mark = square*, mark size = 2pt, step = 1cm,very thin]coordinates{(1,0.05)(2,0.07)(3,0.066)(4,0.068)(5,0.055)(6,0.074)(7,0.095)(8,0.105)(9,0.111)(10,0.089)};
						
						\addplot[color = blue,   mark = *, mark size = 2pt, step = 1cm,very thin]coordinates{(1,0.074)(2,0.077)(3,0.081)(4,0.11)(5,0.132)(6,0.138)(7,0.135)(8,0.107)(9,0.127)(10,0.095)};
						
						\addplot[color = blue,   mark = diamond*, mark size = 2pt, step = 1cm,very thin]coordinates{(1,0.054)(2,0.052)(3,0.06)(4,0.066)(5,0.087)(6,0.083)(7,0.093)(8,0.077)(9,0.057)(10,0.07)};
						
						\addplot[color = red,   mark = *, mark size = 2pt, step = 1cm,very thin]coordinates{(1,0.046)(2,0.052)(3,0.076)(4,0.12)(5,0.138)(6,0.235)(7,0.27)(8,0.329)(9,0.391)(10,0.432)};
						
						\addplot[color = red,   mark = diamond*, mark size = 2pt, step = 1cm,very thin]coordinates{(1,0.058)(2,0.071)(3,0.086)(4,0.118)(5,0.166)(6,0.188)(7,0.278)(8,0.299)(9,0.387)(10,0.416)};
						
						
						\addplot[color = black,   mark = triangle*, mark size = 2pt, step = 1cm,very thin]coordinates{(1,0.045)(2,0.056)(3,0.052)(4,0.047)(5,0.056)(6,0.078)(7,0.076)(8,0.054)(9,0.062)(10,0.058)};
						
						\addplot[dashed, color = black,   mark = , step = 1cm,very thin]coordinates{(1,0.05)(2,0.05)(3,0.05)(4,0.05)(5,0.05)(6,0.05)(7,0.05)(8,0.05)(9,0.05)(10,0.05)};
						
					\end{axis}
				\end{tikzpicture}
				\begin{tikzpicture}[]
					\begin{axis}[xmin = 1, xmax = 10, ymin = 0, ymax = 1, xlabel = {$\log_2(d)$}, ylabel = {Power Estimates}, title = {\bf Example \ref{exa:hd-3}}]
						
						\addplot[color = ForestGreen,   mark = square*, mark size = 2pt, step = 1cm,very thin]coordinates{(1,0.087)(2,0.099)(3,0.115)(4,0.096)(5,0.099)(6,0.135)(7,0.163)(8,0.252)(9,0.243)(10,0.209)};

						\addplot[color = blue,   mark = *, mark size = 2pt, step = 1cm,very thin]coordinates{(1,0.098)(2,0.116)(3,0.136)(4,0.163)(5,0.166)(6,0.16)(7,0.159)(8,0.179)(9,0.195)(10,0.187)};
						
						\addplot[color = blue,   mark = diamond*, mark size = 2pt, step = 1cm,very thin]coordinates{(1,0.046)(2,0.08)(3,0.093)(4,0.093)(5,0.102)(6,0.117)(7,0.116)(8,0.117)(9,0.104)(10,0.122)};
						
						\addplot[color = red,   mark = *, mark size = 2pt, step = 1cm,very thin]coordinates{(1,0.096)(2,0.079)(3,0.142)(4,0.206)(5,0.359)(6,0.557)(7,0.798)(8,0.958)(9,0.994)(10,1)};
						
						\addplot[color = red,   mark = diamond*, mark size = 2pt, step = 1cm,very thin]coordinates{(1,0.066)(2,0.096)(3,0.12)(4,0.181)(5,0.293)(6,0.523)(7,0.744)(8,0.96)(9,1)(10,1)};
						
						
						\addplot[color = black,   mark = triangle*, mark size = 2pt, step = 1cm,very thin]coordinates{(1,0.045)(2,0.074)(3,0.071)(4,0.07)(5,0.055)(6,0.045)(7,0.059)(8,0.051)(9,0.042)(10,0.048)};
						
						\addplot[dashed, color = black,   mark = , step = 1cm,very thin]coordinates{(1,0.05)(2,0.05)(3,0.05)(4,0.05)(5,0.05)(6,0.05)(7,0.05)(8,0.05)(9,0.05)(10,0.05)};
						
					\end{axis}
				\end{tikzpicture}
				
				\caption{Power of the sign test (\tikzsymbol[circle]{minimum width=2pt,fill=blue}), the runs test (\tikzsymbol[diamond]{minimum width=2pt,fill=blue}), the modified sign test (\tikzsymbol[circle]{minimum width=2pt,fill=red}), the modified runs test (\tikzsymbol[diamond]{minimum width=2pt,fill=red}), the OT test ($\textcolor{black}{\blacktriangle}$) 
					and the DT test ($\textcolor{ForestGreen}{\blacksquare}$) in Examples \ref{exa:hd-2} and \ref{exa:hd-3}. The dashed line indicates the nominal level $\alpha = 0.05$.}
				\label{fig:sim-7}
			\end{figure}
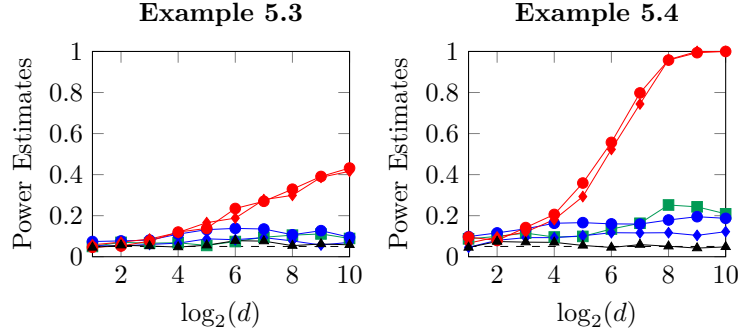
			
			In these examples, the DT test and the OT test had poor performance (see  Figure \ref{fig:sim-7}). The sign and the runs test also performed poorly, but their modified versions worked well, especially in Example \ref{exa:hd-3}. This superiority of the modified tests was quite expected in view of our discussion in Section \ref{sec:improvements}.
			
			Now consider two examples, where the covariance matrix of the underlying distribution is a constant multiple of identity matrix, but the distribution is not spherical. Recall that we consider two such examples (see Examples \ref{exa:iid-uniform} and \ref{exa:iid-coordinates}) in Section
			\ref{sec:improvements}. Here we revisit them as Examples \ref{exa:iid-unif-sim} and \ref{exa:iid-coordinates-sim}, respectively.
			
			\begin{exa}\label{exa:iid-unif-sim}
				We consider a $d$-dimensional distribution with i.i.d. Unif$(-1,1)$ coordinates.
			\end{exa}
			
			\begin{exa}\label{exa:iid-coordinates-sim}
				Observations are generated from a $d$-dimensional distribution, where the coordinate variables are i.i.d. with p.d.f. $f(x) =\frac{1}{2} \exp\{-|x|\}$.
			\end{exa}

			 In Figure \ref{fig:sim-8}, we observed that in Example \ref{exa:iid-unif-sim} while modified runs test had excellent performance, all other competing tests had powers close to the nominal level.
			 In \ref{exa:iid-coordinates-sim}, the modified sign test had the best performance followed by the modified runs tests. But all other tests performed poorly.  The reasons for such performance of the modified tests
			 has already been discussed in Section \ref{sec:improvements}. 	
			 
			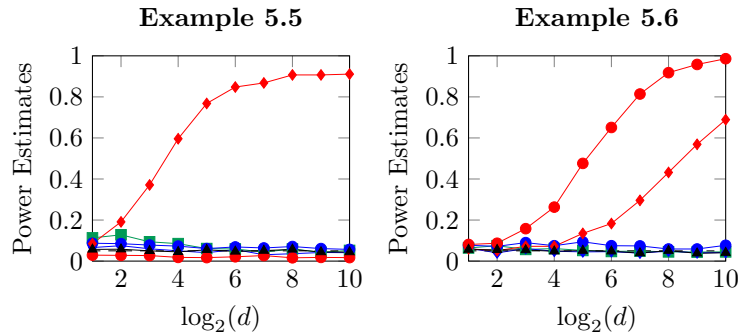
\begin{figure}[h]
				\centering
				\begin{tikzpicture}[]
					\begin{axis}[xmin = 1, xmax = 10, ymin = 0, ymax = 1, xlabel = {$\log_2(d)$}, ylabel = {Power Estimates}, title = {\bf Example \ref{exa:iid-unif-sim}}]
						\addplot[color = ForestGreen,   mark = square*, mark size = 2pt, step = 1cm,very thin]coordinates{(1,0.113)(2,0.128)(3,0.094)(4,0.085)(5,0.061)(6,0.059)(7,0.044)(8,0.059)(9,0.043)(10,0.052)};
						
						\addplot[color = blue,   mark = *, mark size = 2pt, step = 1cm,very thin]coordinates{(1,0.087)(2,0.085)(3,0.078)(4,0.073)(5,0.061)(6,0.069)(7,0.064)(8,0.071)(9,0.061)(10,0.057)};
						
						\addplot[color = blue,   mark = diamond*, mark size = 2pt, step = 1cm,very thin]coordinates{(1,0.066)(2,0.075)(3,0.06)(4,0.05)(5,0.041)(6,0.056)(7,0.031)(8,0.036)(9,0.042)(10,0.041)};
						
						\addplot[color = red,   mark = *, mark size = 2pt, step = 1cm,very thin]coordinates{(1,0.029)(2,0.028)(3,0.027)(4,0.018)(5,0.017)(6,0.021)(7,0.027)(8,0.016)(9,0.018)(10,0.017)};
						
						\addplot[color = red,   mark = diamond*, mark size = 2pt, step = 1cm,very thin]coordinates{(1,0.085)(2,0.191)(3,0.371)(4,0.596)(5,0.768)(6,0.848)(7,0.868)(8,0.907)(9,0.907)(10,0.911)};
						
						
						\addplot[color = black,   mark = triangle*, mark size = 2pt, step = 1cm,very thin]coordinates{(1,0.057)(2,0.058)(3,0.05)(4,0.045)(5,0.055)(6,0.047)(7,0.051)(8,0.057)(9,0.044)(10,0.044)};
						
						\addplot[dashed, color = black,   mark = , step = 1cm,very thin]coordinates{(1,0.05)(2,0.05)(3,0.05)(4,0.05)(5,0.05)(6,0.05)(7,0.05)(8,0.05)(9,0.05)(10,0.05)};
						
					\end{axis}
				\end{tikzpicture}
				\begin{tikzpicture}[]
					\begin{axis}[xmin = 1, xmax = 10, ymin = 0, ymax = 1, xlabel = {$\log_2(d)$}, ylabel = {Power Estimates}, title = {\bf Example \ref{exa:iid-coordinates-sim}}]
						
						\addplot[color = ForestGreen,   mark = square*, mark size = 2pt, step = 1cm,very thin]coordinates{(1,0.064)(2,0.075)(3,0.057)(4,0.061)(5,0.052)(6,0.047)(7,0.049)(8,0.045)(9,0.044)(10,0.045)};

						\addplot[color = blue,   mark = *, mark size = 2pt, step = 1cm,very thin]coordinates{(1,0.077)(2,0.072)(3,0.09)(4,0.075)(5,0.094)(6,0.075)(7,0.074)(8,0.059)(9,0.059)(10,0.077)};
						
						\addplot[color = blue,   mark = diamond*, mark size = 2pt, step = 1cm,very thin]coordinates{(1,0.063)(2,0.04)(3,0.061)(4,0.048)(5,0.045)(6,0.046)(7,0.039)(8,0.049)(9,0.04)(10,0.042)};
						
						\addplot[color = red,   mark = *, mark size = 2pt, step = 1cm,very thin]coordinates{(1,0.081)(2,0.087)(3,0.158)(4,0.263)(5,0.476)(6,0.651)(7,0.814)(8,0.918)(9,0.958)(10,0.986)};
						
						\addplot[color = red,   mark = diamond*, mark size = 2pt, step = 1cm,very thin]coordinates{(1,0.06)(2,0.051)(3,0.072)(4,0.072)(5,0.135)(6,0.183)(7,0.296)(8,0.432)(9,0.569)(10,0.689)};
						
						
						\addplot[color = black,   mark = triangle*, mark size = 2pt, step = 1cm,very thin]coordinates{(1,0.056)(2,0.055)(3,0.054)(4,0.048)(5,0.051)(6,0.048)(7,0.039)(8,0.051)(9,0.038)(10,0.041)};
						
						\addplot[dashed, color = black,   mark = , step = 1cm,very thin]coordinates{(1,0.05)(2,0.05)(3,0.05)(4,0.05)(5,0.05)(6,0.05)(7,0.05)(8,0.05)(9,0.05)(10,0.05)};
						
					\end{axis}
				\end{tikzpicture}
				
				\caption{Power of the sign test (\tikzsymbol[circle]{minimum width=2pt,fill=blue}), the runs test (\tikzsymbol[diamond]{minimum width=2pt,fill=blue}), the modified sign test (\tikzsymbol[circle]{minimum width=2pt,fill=red}), the modified runs test (\tikzsymbol[diamond]{minimum width=2pt,fill=red}), the OT test ($\textcolor{black}{\blacktriangle}$) 
				and the DT test ($\textcolor{ForestGreen}{\blacksquare}$) 
				in Examples \ref{exa:iid-unif-sim} and \ref{exa:iid-coordinates-sim}. The dashed line indicates the nominal level $\alpha = 0.05$.}
				\label{fig:sim-8}
			\end{figure}
			
			

			\begin{figure}[h]
				\centering
				\begin{tikzpicture}[]
					\begin{axis}[xmin = 1, xmax = 5, ymin = 0, ymax = 1, xlabel = {$\log_2(d)$}, ylabel = {Power Estimates}, title = {\bf Example \ref{exa:hd-2}}]
						\addplot[color = ForestGreen,   mark = square*, mark size = 2pt, step = 1cm,very thin]coordinates{(1,0.056)(2,0.071)(3,0.078)(4,0.217)(5,0.928)};

						\addplot[color = blue,   mark = *, mark size = 2pt, step = 1cm,very thin]coordinates{(1,0.078)(2,0.069)(3,0.074)(4,0.221)(5,0.679)};
						
						\addplot[color = blue,   mark = diamond*, mark size = 2pt, step = 1cm,very thin]coordinates{(1,0.043)(2,0.051)(3,0.072)(4,0.178)(5,0.688)};
						
						\addplot[color = red,   mark = *, mark size = 2pt, step = 1cm,very thin]coordinates{(1,0.026)(2,0.042)(3,0.093)(4,0.359)(5,0.93)};
						
						\addplot[color = red,   mark = diamond*, mark size = 2pt, step = 1cm,very thin]coordinates{(1,0.025)(2,0.05)(3,0.106)(4,0.346)(5,0.968)};
						
						
						\addplot[color = black,   mark = triangle*, mark size = 2pt, step = 1cm,very thin]coordinates{(1,0.045)(2,0.056)(3,0.053)(4,0.059)(5,0.069)};
						
						\addplot[dashed, color = black,   mark = , step = 1cm,very thin]coordinates{(1,0.05)(2,0.05)(3,0.05)(4,0.05)(5,0.05)};
						
					\end{axis}
				\end{tikzpicture}
				\begin{tikzpicture}[]
					\begin{axis}[xmin = 1, xmax = 5, ymin = 0, ymax = 1, xlabel = {$\log_2(d)$}, ylabel = {Power Estimates}, title = {\bf Example \ref{exa:hd-3}}]
						
						\addplot[color = ForestGreen,   mark = square*, mark size = 2pt, step = 1cm,very thin]coordinates{(1,0.08)(2,0.1)(3,0.149)(4,0.6)(5,1)};

						\addplot[color = blue,   mark = *, mark size = 2pt, step = 1cm,very thin]coordinates{(1,0.119)(2,0.113)(3,0.149)(4,0.383)(5,0.965)};
						
						\addplot[color = blue,   mark = diamond*, mark size = 2pt, step = 1cm,very thin]coordinates{(1,0.063)(2,0.071)(3,0.09)(4,0.389)(5,0.983)};
						
						\addplot[color = red,   mark = *, mark size = 2pt, step = 1cm,very thin]coordinates{(1,0.051)(2,0.087)(3,0.2)(4,0.744)(5,1)};
						
						\addplot[color = red,   mark = diamond*, mark size = 2pt, step = 1cm,very thin]coordinates{(1,0.057)(2,0.074)(3,0.181)(4,0.794)(5,1)};
						
						
						\addplot[color = black,   mark = triangle*, mark size = 2pt, step = 1cm,very thin]coordinates{(1,0.063)(2,0.068)(3,0.047)(4,0.054)(5,0.057)};
						
						\addplot[dashed, color = black,   mark = , step = 1cm,very thin]coordinates{(1,0.05)(2,0.05)(3,0.05)(4,0.05)(5,0.05)};
						
					\end{axis}
				\end{tikzpicture}

				\begin{tikzpicture}[]
					\begin{axis}[xmin = 1, xmax = 5, ymin = 0, ymax = 1, xlabel = {$\log_2(d)$}, ylabel = {Power Estimates}, title = {\bf Example \ref{exa:iid-unif-sim}}]
						\addplot[color = ForestGreen,   mark = square*, mark size = 2pt, step = 1cm,very thin]coordinates{(1,0.085)(2,0.111)(3,0.155)(4,0.311)(5,0.694)};
						
						\addplot[color = blue,   mark = *, step = 1cm,very thin]coordinates{(1,0.093)(2,0.101)(3,0.083)(4,0.162)(5,0.256)};
					
					\addplot[color = blue,   mark = diamond*, step = 1cm,very thin]coordinates{(1,0.05)(2,0.064)(3,0.069)(4,0.142)(5,0.282)};
					
					\addplot[color = red,   mark = *, step = 1cm,very thin]coordinates{(1,0.019)(2,0.032)(3,0.035)(4,0.091)(5,0.159)};
					
					\addplot[color = red,   mark = diamond*, step = 1cm,very thin]coordinates{(1,0.042)(2,0.106)(3,0.767)(4,1)(5,1)};
					
						
						\addplot[color = black,   mark = triangle*, mark size = 2pt, step = 1cm,very thin]coordinates{(1,0.064)(2,0.049)(3,0.046)(4,0.05)(5,0.051)};
						
						\addplot[dashed, color = black,   mark = , step = 1cm,very thin]coordinates{(1,0.05)(2,0.05)(3,0.05)(4,0.05)(5,0.05)};
						
					\end{axis}
				\end{tikzpicture}
				\begin{tikzpicture}[]
					\begin{axis}[xmin = 1, xmax = 5, ymin = 0, ymax = 1, xlabel = {$\log_2(d)$}, ylabel = {Power Estimates}, title = {\bf Example \ref{exa:iid-coordinates-sim}}]
						
						\addplot[color = ForestGreen,   mark = square*, mark size = 2pt, step = 1cm,very thin]coordinates{(1,0.068)(2,0.064)(3,0.066)(4,0.129)(5,0.374)};

					\addplot[color = blue,   mark = *, step = 1cm,very thin]coordinates{(1,0.073)(2,0.081)(3,0.071)(4,0.2)(5,0.721)};
					
					\addplot[color = blue,   mark = diamond*, step = 1cm,very thin]coordinates{(1,0.047)(2,0.049)(3,0.061)(4,0.163)(5,0.699)};
					
					\addplot[color = red,   mark = *, step = 1cm,very thin]coordinates{(1,0.029)(2,0.057)(3,0.354)(4,1)(5,1)};
					
					\addplot[color = red,   mark = diamond*, step = 1cm,very thin]coordinates{(1,0.034)(2,0.048)(3,0.094)(4,0.786)(5,1)};
						
						
						\addplot[color = black,   mark = triangle*, mark size = 2pt, step = 1cm,very thin]coordinates{(1,0.065)(2,0.042)(3,0.042)(4,0.055)(5,0.051)};
						
						\addplot[dashed, color = black,   mark = , step = 1cm,very thin]coordinates{(1,0.05)(2,0.05)(3,0.05)(4,0.05)(5,0.05)};
						
					\end{axis}
				\end{tikzpicture}
				
				\vspace{-0.1in}
				\caption{Power of the sign test (\tikzsymbol[circle]{minimum width=2pt,fill=blue}), the runs test (\tikzsymbol[diamond]{minimum width=2pt,fill=blue}), the modified sign test (\tikzsymbol[circle]{minimum width=2pt,fill=red}), the modified runs test (\tikzsymbol[diamond]{minimum width=2pt,fill=red}), the OT test ($\textcolor{black}{\blacktriangle}$) 
				and the DT test ($\textcolor{ForestGreen}{\blacksquare}$) as a function of the dimension $d$ when we generate $n=d^2+20$ many observations from Examples \ref{exa:hd-2} 
					-\ref{exa:iid-coordinates-sim}. The dashed line indicates the nominal level $\alpha = 0.05$.}
				\label{fig:sim-8.1}
			\end{figure}
			
			 In Examples 5.3-5.6, though the sign and runs tests had powers close to the nominal level in high dimensions, they can have better performance if the sample size also increases with the dimension at a suitable rate. To demonstrate this, we revisit these four examples but this time we consider the sample size $n = d^2 + 20$ that increases with the dimension $d$. The results are reported in Figure \ref{fig:sim-8.1}. 
			Unlike before, except for the OT test,
			powers of all tests showed increasing trends as the dimension increases. In Examples \ref{exa:hd-2} and \ref{exa:hd-3}, though the modified tests had a clear edge,
			all these tests had satisfactory performance 
			in high dimensions. 
		In Example \ref{exa:iid-unif-sim}, the modified runs test outperformed all other tests as before. The DT test had the second best performance, while the sign and runs tests performed better than the modified sign test. However, in Example \ref{exa:iid-coordinates-sim} the modified
		sign test had the best performance followed by the modified runs test. In this example, the sign and runs tests higher powers than the DT test in high dimensions.
	
			\subsection{Analysis of `Earthquake' data}
			For further evaluation of the performance of our tests, we analyze the `Earthquakes' data available at the Time Series Machine Learning website (\href{https://www.timeseries classification.com/}{https://www.timeseries- classification.com/dataset.php}). Data were collected from Northern California Earthquake Data Center, which were donated by Prof. Anthony Bagnall. Here each datum is the hourly average of readings on the Richter scale during 1967 and 2003. The single time series was then transformed into multi-dimensional objects by segmenting the time series by intervals of 512 hours. Any reading over 5 on the Richter scale is defined as a major event. However, such events are often followed by aftershocks. Hence, a segment of the time series is considered to be a positive case if there is a major event in that segment that is not preceded by another major event for at least 512 hours. Any reading below 4 that is preceded by at least 20 non-zero readings in the previous 512 hours is considered as a negative case. After this initial processing, this dataset has 512 hourly readings on 368 negative cases and 93  positive cases. 
			We consider these two groups containing (a) Positive cases and (b) Negative cases separately and test whether their underlying distributions are spherically symmetric. 
			
			However, if we use the full data set for testing, any test will either accept or reject the null hypothesis. Based on that single experiment, it is difficult to compare among different test procedures. Therefore, to compare our tests with the other methods, we adopt a sub-sampling approach, where we take a random sub-sample containing $p$ ($0<p<1$) proportion of observations and apply the tests on that the sub-sample. 
			For each of the 5 values of $p$ (0.2, 0.4, 0.6,0.8 amd 0.95), this experiment is carried out 1000 times to compute the power of the tests by the proportion of times they reject $H_0$. The results are reported in Figure \ref{fig:real-1}.
			
			\begin{figure}[h]
				\centering
				\begin{tikzpicture}[]
					\begin{axis}[xmin = 0.2, xmax = 0.95, ymin = 0, ymax = 1, xlabel = {$p$}, ylabel = {Power Estimates}, title = {\bf (a) Positive Case}]
						\addplot[color = ForestGreen,   mark = square*, mark size = 2pt, step = 1cm,very thin]coordinates{(0.2,0.1)(0.4,0.184)(0.6,0.287)(0.8,0.455)(0.95,0.738)};
						
						\addplot[color = blue,   mark = *, mark size = 2pt, step = 1cm,very thin]coordinates{(0.2,0.15)(0.4,0.123)(0.6,0.076)(0.8,0.099)(0.95,0.092)};
						
						\addplot[color = blue,   mark = diamond*, mark size = 2pt, step = 1cm,very thin]coordinates{(0.2,0.019)(0.4,0.046)(0.6,0.06)(0.8,0.058)(0.95,0.07)};
						
						\addplot[color = red,   mark = *, mark size = 2pt, step = 1cm,very thin]coordinates{(0.2,0.109)(0.4,0.387)(0.6,0.526)(0.8,0.819)(0.95,0.91)};
						
						\addplot[color = red,   mark = diamond*, mark size = 2pt, step = 1cm,very thin]coordinates{(0.2,0.063)(0.4,0.092)(0.6,0.201)(0.8,0.259)(0.95,0.325)};
						
						
						\addplot[color = black,   mark = triangle*, mark size = 2pt, step = 1cm,very thin]coordinates{(0.2,0.066)(0.4,0.057)(0.6,0.073)(0.8,0.034)(0.95,0.003)};
						
						
						
						\addplot[dashed, color = black,   mark = , step = 1cm,very thin]coordinates{(0.2,0.05)(0.4,0.05)(0.6,0.05)(0.8,0.05)(1,0.05)};
						
					\end{axis}
				\end{tikzpicture}
				\begin{tikzpicture}[]
					\begin{axis}[xmin = 0.2, xmax = 0.95, ymin = 0, ymax = 1, xlabel = {$p$}, ylabel = {Power Estimates}, title = {\bf (b) Negative Case}]
						\addplot[color = ForestGreen,   mark = square*, mark size = 2pt, step = 1cm,very thin]coordinates{(0.2,0.615)(0.4,0.994)(0.6,1)(0.8,1)(0.95,1)};
						
						\addplot[color = blue,   mark = *, mark size = 2pt, step = 1cm,very thin]coordinates{(0.2,0.08)(0.4,0.15)(0.6,0.227)(0.8,0.353)(0.95,0.408)};
						
						\addplot[color = blue,   mark = diamond*, mark size = 2pt, step = 1cm,very thin]coordinates{(0.2,0.05)(0.4,0.103)(0.6,0.173)(0.8,0.271)(0.95,0.333)};
						
						\addplot[color = red,   mark = *, mark size = 2pt, step = 1cm,very thin]coordinates{(0.2,1)(0.4,1)(0.6,1)(0.8,1)(0.95,1)};
						
						\addplot[color = red,   mark = diamond*, mark size = 2pt, step = 1cm,very thin]coordinates{(0.2,0.892)(0.4,0.999)(0.6,1)(0.8,1)(0.95,1)};
						
						
						\addplot[color = black,   mark = triangle*, mark size = 2pt, step = 1cm,very thin]coordinates{(0.2,0.24)(0.4,0.564)(0.6,0.915)(0.8,0.999)(0.95,1)};
						
						
						
						\addplot[dashed, color = black,   mark = , step = 1cm,very thin]coordinates{(0.2,0.05)(0.4,0.05)(0.6,0.05)(0.8,0.05)(1,0.05)};

					\end{axis}
				\end{tikzpicture}
				\caption{Powers of the sign test (\tikzsymbol[circle]{minimum width=2pt,fill=blue}), the runs test (\tikzsymbol[diamond]{minimum width=2pt,fill=blue}), the modified sign test (\tikzsymbol[circle]{minimum width=2pt,fill=red}), the modified runs test (\tikzsymbol[diamond]{minimum width=2pt,fill=red}), the OT test ($\textcolor{black}{\blacktriangle}$) 
					and the DT test ($\textcolor{ForestGreen}{\blacksquare}$) based on varying proportions of observations ($p$) from the positive and the negative cases in the `Earthquakes' dataset. The dashed line indicates the nominal level $\alpha = 0.05$.}
				\label{fig:real-1}
			\end{figure}
			
			\begin{figure}[!b]
				\centering
				\includegraphics[width=0.44\linewidth]{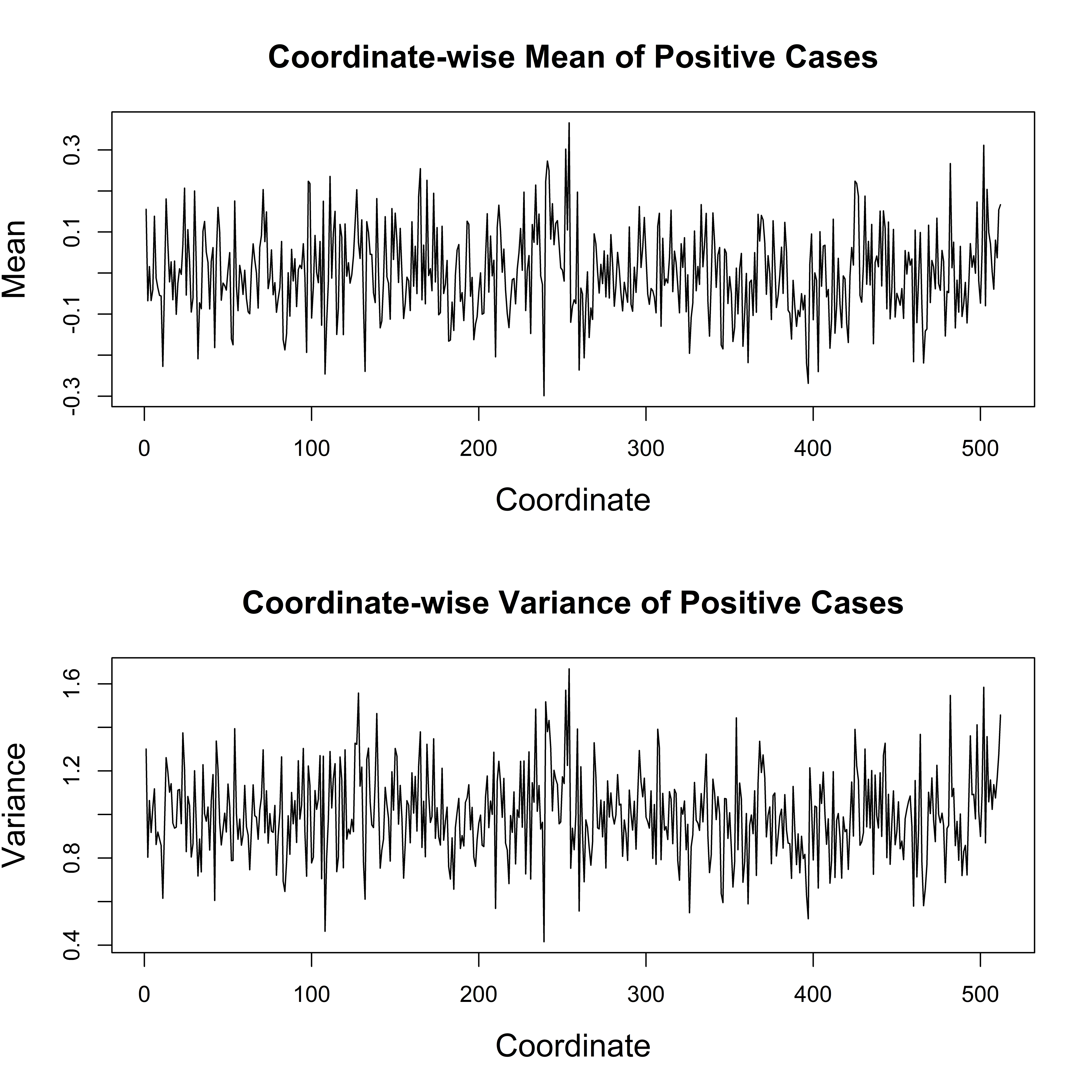}
				\includegraphics[width=0.44\linewidth]{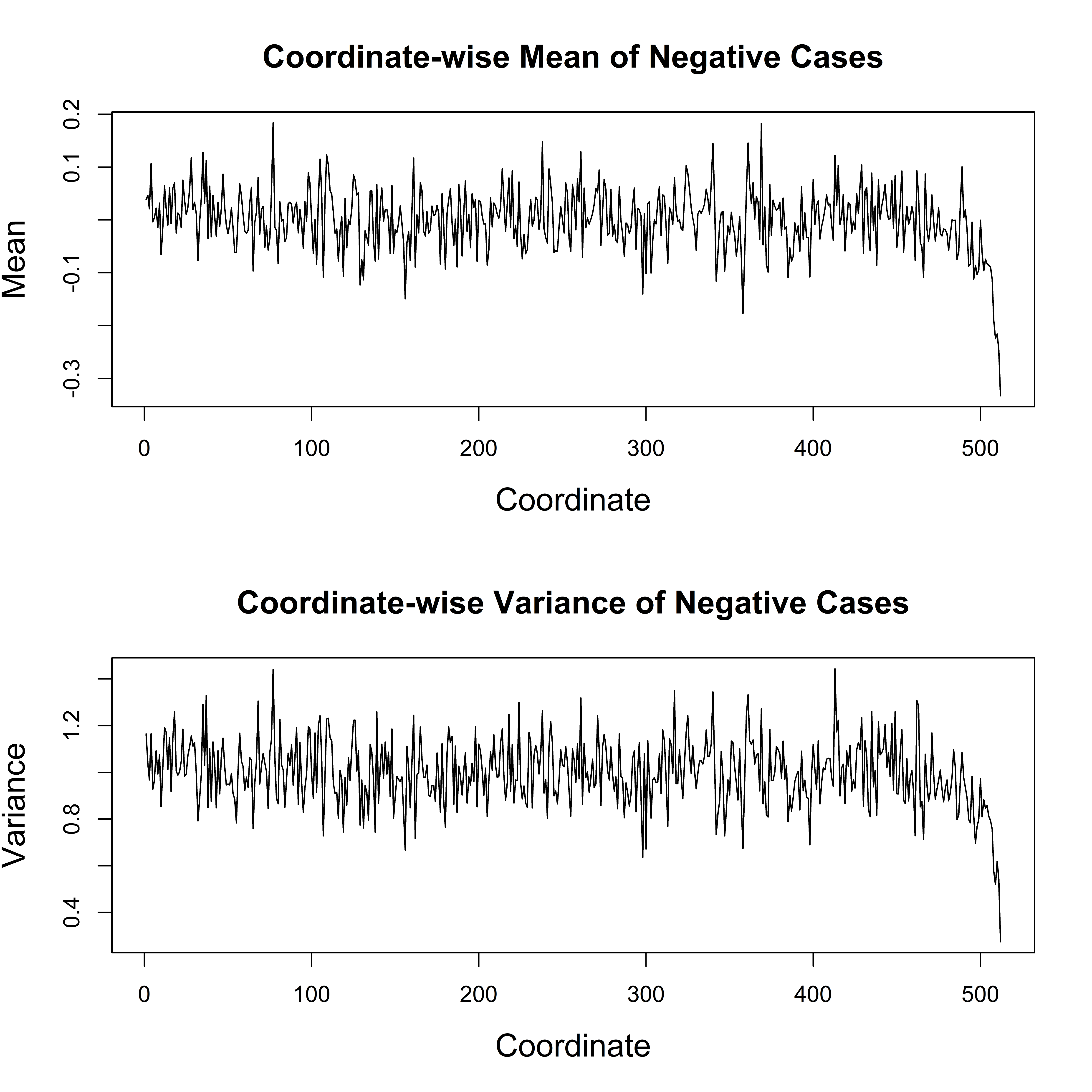}
				\caption{Coordinate-wise mean and variance of the feature vectors in the `Earthquakes' data set divided into two groups of positive and negative cases.}
				\label{fig:mean-var}
			\end{figure}

			We observed an interesting phenomenon in this dataset. For the group containing Positive Cases, only our modified tests, and the DT test are able to detect the deviation from spherical asymmetry. 
			In this example, the modified sign test outperformed all other tests, while the powers of the DT test and the modified runs test slowly increased with $p$. For the group of Negative Cases,
			powers of all tests steadily increased with $p$. Here also, the modified sign test significantly outperformed its all other competitors. The modified runs test (based on $T_S$ and $T_R$) had the second best performance closely followed by he DT test.  The OT test exhibited satisfactory performance. However our sign test and runs test had relatively low powers.
			
			The above result indicates that the distribution of the Negative Cases deviates more from spherical symmetry compared to the distribution of the Positive Cases. This is confirmed by Figure \ref{fig:mean-var}, which shows the plots of the coordinate-wise sample mean and sample variance for the two groups.
			For the Positive Cases, the sample mean is more or less stationary about zero, but for the Negative Cases, the sample mean as well as sample variance have a sharp drop at the right end. This sharp drop can be a potential reason behind the high powers of the tests in Figure \ref{fig:real-1} (b). However, this sharp drop in the mean and variance of the data could be a subjective bias at the data curation step. Therefore, to eliminate such possible bias, we truncate the feature vectors (both for Positive and Negative cases) by removing 32 features from the end and carry out our experiment with the first $480$ coordinates (which corresponds to $20$ days hourly readings on the Richter scale). Powers of different tests are computed based on 1000 random sub-samples as before. Our findings are reported in Figure \ref{fig:real-2}. 
			
			\begin{figure}[h]
				\centering
				\begin{tikzpicture}[]
					\begin{axis}[xmin = 0.2, xmax = 0.95, ymin = 0, ymax = 1, xlabel = {$p$}, ylabel = {Power Estimates}, title = {\bf (a) Positive Case}]
						\addplot[color = ForestGreen,   mark = square*, mark size = 2pt, step = 1cm,very thin]coordinates{(0.2,0.098)(0.4,0.116)(0.6,0.125)(0.8,0.083)(0.95,0.042)};
						
						\addplot[color = blue,   mark = *, mark size = 2pt, step = 1cm,very thin]coordinates{(0.2,0.144)(0.4,0.109)(0.6,0.076)(0.8,0.127)(0.95,0.109)};
						
						\addplot[color = blue,   mark = diamond*, mark size = 2pt, step = 1cm,very thin]coordinates{(0.2,0.019)(0.4,0.04)(0.6,0.053)(0.8,0.066)(0.95,0.083)};
						
						\addplot[color = red,   mark = *, mark size = 2pt, step = 1cm,very thin]coordinates{(0.2,0.107)(0.4,0.337)(0.6,0.482)(0.8,0.776)(0.95,0.869)};
						
						\addplot[color = red,   mark = diamond*, mark size = 2pt, step = 1cm,very thin]coordinates{(0.2,0.042)(0.4,0.096)(0.6,0.171)(0.8,0.242)(0.95,0.302)};
						
						
						\addplot[color = black,   mark = triangle*, mark size = 2pt, step = 1cm,very thin]coordinates{(0.2,0.055)(0.4,0.052)(0.6,0.058)(0.8,0.021)(0.95,0.0)};
						
						
						
						\addplot[dashed, color = black,   mark = , step = 1cm,very thin]coordinates{(0.2,0.05)(0.4,0.05)(0.6,0.05)(0.8,0.05)(1,0.05)};
						
					\end{axis}
				\end{tikzpicture}
				\begin{tikzpicture}[]
					\begin{axis}[xmin = 0.2, xmax = 0.95, ymin = 0, ymax = 1, xlabel = {$p$}, ylabel = {Power Estimates}, title = {\bf (b) Negative Case}]
						\addplot[color = ForestGreen,   mark = square*, mark size = 2pt, step = 1cm,very thin]coordinates{(0.2,0.371)(0.4,0.773)(0.6,0.99)(0.8,1)(0.95,1)};
						
						\addplot[color = blue,   mark = *, mark size = 2pt, step = 1cm,very thin]coordinates{(0.2,0.071)(0.4,0.151)(0.6,0.214)(0.8,0.296)(0.95,0.369)};
						
						\addplot[color = blue,   mark = diamond*, mark size = 2pt, step = 1cm,very thin]coordinates{(0.2,0.05)(0.4,0.086)(0.6,0.139)(0.8,0.221)(0.95,0.308)};
						
						\addplot[color = red,   mark = *, mark size = 2pt, step = 1cm,very thin]coordinates{(0.2,0.998)(0.4,1)(0.6,1)(0.8,1)(0.95,1)};
						
						\addplot[color = red,   mark = diamond*, mark size = 2pt, step = 1cm,very thin]coordinates{(0.2,0.845)(0.4,0.997)(0.6,1)(0.8,1)(0.95,1)};
						
						
						\addplot[color = black,   mark = triangle*, mark size = 2pt, step = 1cm,very thin]coordinates{(0.2,0.069)(0.4,0.057)(0.6,0.039)(0.8,0.018)(0.95,0)};
						
						
						
						\addplot[dashed, color = black,   mark = , step = 1cm,very thin]coordinates{(0.2,0.05)(0.4,0.05)(0.6,0.05)(0.8,0.05)(1,0.05)};

					\end{axis}
				\end{tikzpicture}
				\caption{Power of the sign test (\tikzsymbol[circle]{minimum width=2pt,fill=blue}), the runs test (\tikzsymbol[diamond]{minimum width=2pt,fill=blue}), the modified sign test (\tikzsymbol[circle]{minimum width=2pt,fill=red}), the modified runs test (\tikzsymbol[diamond]{minimum width=2pt,fill=red}), the OT test ($\textcolor{black}{\blacktriangle}$) 
					and the DT test ($\textcolor{ForestGreen}{\blacksquare}$) based on varying proportions of observations ($p$) from the positive and the negative cases in the truncated `Earthquakes' dataset. The dashed line indicates the nominal level $\alpha = 0.05$.}
				\label{fig:real-2}
			\end{figure}
			
			Here also, the modified sign test significantly outperformed its competitors both for Positive and Negative cases. For the Positive cases, unlike before, the power of the DT test did not show any increasing pattern. Here, only the modified sign and runs tests had powers increasing with $p$. For the Negative Cases, the OT test had very poor performance. Our sign and runs tests (based on $T_S$ and $T_R$) also failed to achieve satisfactory performance. However, the other tests showed a similar pattern as observed in Figure \ref{fig:real-1} (b). 
			

			\section{Tests of spherical symmetry about an unknown center}
			\label{sec:centering-effect}
			So far, we have considered the null hypothesis that specifies the center of symmetry of the underlying distribution. Without loss of generality, the origin was taken as the specified center. However, if null hypothesis does not specify the center, it calls for a test of spherical symmetry about an unknown center ${\bm \mu}$. In such cases, one can think of estimating the center from the data and test for the spherical symmetry about that estimated center ${\bm {\hat \mu}}$. For instance, we can use the sample mean as ${\bm {\hat \mu}}$, subtract it from the original observations for centering, and then apply our tests on the centered data. Of course, one can also use other robust estimates of ${\bm \mu}$ like the MCD estimate
			\citep[see, e.g.,][]{rousseeuw1999fast} or the MVE estimate \citep[see, e.g.,][]{van2009minimum} for this purpose. Depth based estimates like the spatial median \citep[see, e.g.][]{chaudhuri1996, kolchinskii1998} can be used as well. This method of centering works well when the sample size is large compared to the dimension of the data. But, in high-dimensional problem, it may lead to an inflated type I error, especially when the dimension exceeds the sample size. This is illustrated using the following example.
			
			\begin{exa}\label{exa:noncentral-1}
				We generate $n$ observations from a $d$-variate normal distribution with mean $\bf{1}_d$ and variance covariance matrix $\mathrm {\bf I}_d$,
				where $\mathrm {\bf 1}_d$ denotes the $d$-dimensional vector $(1,1,\ldots,1)$ with elements equal to $1$.
			\end{exa}
			
			\begin{figure}[h]
				\centering
				\begin{tikzpicture}[]
					\begin{axis}[xmin = 100, xmax = 500, ymin = 0, ymax = 0.25, xlabel = {Sample Size}, ylabel = {Power Estimates}, title = {\bf (a) $d = 20$}]
						
						\addplot[color = ForestGreen,   mark = square*, mark size = 2pt, step = 1cm,very thin]coordinates{(100,0.022)(200,0.023)(300,0.022)(400,0.025)(500,0.029)};
						
						\addplot[color = blue,   mark = *, mark size = 2pt, step = 1cm,very thin]coordinates{(100,0.065)(200,0.051)(300,0.068)(400,0.056)(500,0.063)};
						
						\addplot[color = blue,   mark = diamond*, mark size = 2pt, step = 1cm,very thin]coordinates{(100,0.045)(200,0.04)(300,0.04)(400,0.038)(500,0.052)};
						
						\addplot[color = red,   mark = *, mark size = 2pt, step = 1cm,very thin]coordinates{(100,0.04)(200,0.043)(300,0.042)(400,0.033)(500,0.062)};
						
						\addplot[color = red,   mark = diamond*, mark size = 2pt, step = 1cm,very thin]coordinates{(100,0.046)(200,0.052)(300,0.041)(400,0.047)(500,0.056)};
						
						
						\addplot[color = black,   mark = triangle*, mark size = 2pt, step = 1cm,very thin]coordinates{(100,0)(200,0)(300,0)(400,0)(500,0)};
						
						
						
						\addplot[dashed, color = black,   mark = , step = 1cm,very thin]coordinates{(100,0.05)(200,0.05)(300,0.05)(400,0.05)(500,0.05)};
						
					\end{axis}
				\end{tikzpicture}
				\begin{tikzpicture}[]
					\begin{axis}[xmin = 1, xmax = 10, ymin = 0, ymax = 0.25, xlabel = {$\log_2(d)$}, ylabel = {Power Estimates}, title = {\bf (b) $n = 50$}]

						\addplot[color = ForestGreen,   mark = square*, mark size = 2pt, step = 1cm,very thin]coordinates{(1,0.013)(2,0.028)(3,0.021)(4,0.03)(5,0.021)(6,0.016)(7,0)(8,0)(9,0)(10,0)};
						
						\addplot[color = blue,   mark = *, mark size = 2pt, step = 1cm,very thin]coordinates{(1,0.045)(2,0.037)(3,0.06)(4,0.06)(5,0.072)(6,0.056)(7,0.074)(8,0.066)(9,0.093)(10,	0.156)};
						
						\addplot[color = blue,   mark = diamond*, mark size = 2pt, step = 1cm,very thin]coordinates{(1,0.036)(2,0.039)(3,0.043)(4,0.041)(5,0.05)(6,0.053)(7,0.047)(8,0.051)(9,0.074)(10,0.095)};
						
						\addplot[color = red,   mark = *, mark size = 2pt, step = 1cm,very thin]coordinates{(1,0.041)(2,0.04)(3,0.041)(4,0.04)(5,0.042)(6,0.041)(7,0.057)(8,0.041)(9,0.054)(10,	0.097)};
						
						\addplot[color = red,   mark = diamond*, mark size = 2pt, step = 1cm,very thin]coordinates{(1,0.045)(2,0.04)(3,0.041)(4,0.045)(5,0.039)(6,0.044)(7,0.05)(8,0.044)(9,0.058)(10,0.077)};
						
						
						\addplot[color = black,   mark = triangle*, mark size = 2pt, step = 1cm,very thin]coordinates{(1,0.001)(2,0.003)(3,0)(4,0)(5,0)(6,0)(7,0)(8,0)(9,0)(10,0)};
						
						
						
						\addplot[dashed, color = black,   mark = , step = 1cm,very thin]coordinates{(1,0.05)(2,0.05)(3,0.05)(4,0.05)(5,0.05)(6,0.05)(7,0.05)(8,0.05)(9,0.05)(10,0.05)};
						
					\end{axis}
				\end{tikzpicture}
				\caption{Type I errors of the sign test (\tikzsymbol[circle]{minimum width=2pt,fill=blue}), the runs test (\tikzsymbol[diamond]{minimum width=2pt,fill=blue}), the modified sign test (\tikzsymbol[circle]{minimum width=2pt,fill=red}), the modified runs test (\tikzsymbol[diamond]{minimum width=2pt,fill=red}), the OT test ($\textcolor{black}{\blacktriangle}$) 
				and the DT test ($\textcolor{ForestGreen}{\blacksquare}$) in Example \ref{exa:noncentral-1} when (a) the sample size increases while the dimension is kept fixed at $20$ and (b) when the dimension increases while the sample size is kept fixed at $50$. The dashed line indicates the nominal level $\alpha = 0.05$.}
				\label{fig:sim-9}
			\end{figure}
			
			To evaluate the effect of centering on different tests, we use the spatial median as ${\bm {\hat \mu}}$ and apply the tests on the centered data. Instead of sample mean, we choose the spatial median because of its better robustness properties. First we look at the Type I errors of different tests as functions of the sample size when the dimension is kept fixed at $20$ (see In Figure \ref{fig:sim-9} (a)). In this case, Type I errors of all tests become close to the nominal level $\alpha=0.05$ as the sample size increases. With the increasing sample size, since ${\bm {\hat \mu}}$ becomes close to ${\bm \mu}$, such a phenomenon is quite expected. Next we look at their Type I errors when the sample size is kept fixed at $50$, and the dimension varies (see Figure \ref{fig:sim-9} (b)). In higher dimension, the Type I error rates of our tests became higher than the nominal level. 
			On the other hand, DT and OT tests had Type I errors converging to $0$ as the dimension increases. Note that if the  dimension is higher compared to the sample size, 
			${\bm {\hat \mu}}$ may be very different from the actual center $\bm \mu$, and this leads to the loss of the exchangeability property of the observed data points and their spherically symmetric variants (i.e., the difference between the distributions of $\bm X-{\bm {\hat\mu}}$ and $\|\bm X-{\bm {\hat\mu}}\|\bm U$ increases with the dimension), which increases the Type I error of our tests.  To take care of this problem, we can use an idea based on sample splitting, which is motivated by the result stated below.

			\begin{lemma}
				\label{lemma:key-lemma}
				Suppose that $\bm X_1$ and $\bm X_2$ are two independent copies of $\bm X \sim \Pr$, which is symmetric about ${\bm \mu}$. Then $\Pr$ is spherically symmetric about $\bm \mu$ if and only if the distribution of $\bm X_1-\bm X_2$ is spherically symmetric about the origin. 
			\end{lemma}

		\begin{figure}[!b]
			\centering
			\begin{tikzpicture}[]
				\begin{axis}[xmin = 1, xmax = 10, ymin = 0, ymax = 1, xlabel = {$\log_2(d)$}, ylabel = {Power Estimates}, title = {\bf (a) Centered Data}]
					
					\addplot[color = ForestGreen,   mark = square*, mark size = 2pt, step = 1cm,very thin]coordinates{(1,0.038)(2,0.112)(3,0.156)(4,0.238)(5,0.415)(6,0.673)(7,0.929)(8,0.991)(9,0.998)(10,1)};
					
					\addplot[color = blue,   mark = *, mark size = 2pt, step = 1cm,very thin]coordinates{(1,0.108)(2,0.194)(3,0.388)(4,0.696)(5,0.941)(6,0.984)(7,0.999)(8,1)(9,1)(10,1)};
					
					\addplot[color = blue,   mark = diamond*, mark size = 2pt, step = 1cm,very thin]coordinates{(1,0.076)(2,0.136)(3,0.302)(4,0.629)(5,0.876)(6,0.974)(7,0.998)(8,1)(9,1)(10,1)};
					
					\addplot[color = red,   mark = *, mark size = 2pt, step = 1cm,very thin]coordinates{(1,0.053)(2,0.096)(3,0.233)(4,0.52)(5,0.837)(6,0.956)(7,0.987)(8,0.999)(9,1)(10,1)};
					
					\addplot[color = red,   mark = diamond*, mark size = 2pt, step = 1cm,very thin]coordinates{(1,0.068)(2,0.096)(3,0.232)(4,0.528)(5,0.807)(6,0.951)(7,0.996)(8,0.999)(9,1)(10,1)};
					
					
					\addplot[color = black,   mark = triangle*, mark size = 2pt, step = 1cm,very thin]coordinates{(1,0)(2,0.001)(3,0)(4,0)(5,0)(6,0)(7,0)(8,0)(9,0)(10,0)};
					
					
					
					\addplot[dashed, color = black,   mark = , step = 1cm,very thin]coordinates{(1,0.05)(2,0.05)(3,0.05)(4,0.05)(5,0.05)(6,0.05)(7,0.05)(8,0.05)(9,0.05)(10,0.05)};
					
				\end{axis}
			\end{tikzpicture}
			\begin{tikzpicture}[]
				\begin{axis}[xmin = 1, xmax = 10, ymin = 0, ymax = 1, xlabel = {$\log_2(d)$}, ylabel = {Power Estimates}, title = {\bf (b) Spample Splitting}]

					\addplot[color = ForestGreen,   mark = square*, mark size = 2pt, step = 1cm,very thin]coordinates{(1,0.06)(2,0.092)(3,0.137)(4,0.193)(5,0.327)(6,0.57)(7,0.84)(8,0.961)(9,1)(10,1)};
					
					\addplot[color = blue,   mark = *, mark size = 2pt, step = 1cm,very thin]coordinates{(1,0.079)(2,0.14)(3,0.288)(4,0.551)(5,0.74)(6,0.868)(7,0.945)(8,0.991)(9,0.995)(10,0.999)};
					
					\addplot[color = blue,   mark = diamond*, mark size = 2pt, step = 1cm,very thin]coordinates{(1,0.042)(2,0.076)(3,0.166)(4,0.355)(5,0.563)(6,0.751)(7,0.908)(8,0.968)(9,0.982)(10,0.998)};
					
					\addplot[color = red,   mark = *, mark size = 2pt, step = 1cm,very thin]coordinates{(1,0.057)(2,0.096)(3,0.179)(4,0.392)(5,0.601)(6,0.776)(7,0.892)(8,0.97)(9,0.981)(10,	0.996)};
					
					\addplot[color = red,   mark = diamond*, mark size = 2pt, step = 1cm,very thin]coordinates{(1,0.025)(2,0.051)(3,0.097)(4,0.211)(5,0.397)(6,0.595)(7,0.789)(8,0.916)(9,0.962)(10,0.985)};
					
					
					\addplot[color = black,   mark = triangle*, mark size = 2pt, step = 1cm,very thin]coordinates{(1,0.048)(2,0.055)(3,0.058)(4,0.104)(5,0.1)(6,0.15)(7,0.173)(8,0.214)(9,0.245)(10,0.258)};
					
					
					
					\addplot[dashed, color = black,   mark = , step = 1cm,very thin]coordinates{(1,0.05)(2,0.05)(3,0.05)(4,0.05)(5,0.05)(6,0.05)(7,0.05)(8,0.05)(9,0.05)(10,0.05)};
					
				\end{axis}
			\end{tikzpicture}
			\caption{Powers of the sign test (\tikzsymbol[circle]{minimum width=2pt,fill=blue}), the runs test (\tikzsymbol[diamond]{minimum width=2pt,fill=blue}), the modified sign test (\tikzsymbol[circle]{minimum width=2pt,fill=red}), the modified runs test (\tikzsymbol[diamond]{minimum width=2pt,fill=red}), the OT test ($\textcolor{black}{\blacktriangle}$) 
				and the DT test ($\textcolor{ForestGreen}{\blacksquare}$) in Example \ref{exa:centering-power} when (a) the samples are centered using the spatial median and (b) when we use differences of the observations based on sample splitting. The dashed line indicates the nominal level $\alpha = 0.05$.}
			\label{fig:sim-9.1}
		\end{figure}
	
			Therefore, if the location $\bm \mu$ is unknown and the sample size $n$ is even (discard one observation, if needed), we can use our tests on the transformed data $\{\bm Z_1,\bm Z_2,\ldots, \bm Z_{n/2}\}$, where $\bm Z_i = \bm X_i-\bm X_{n/2+i}$ for $i=1,2,\ldots,n/2$. The resulting tests will have the exact distribution-free property and their asymptotic properties can be established using arguments similar to those in Section \ref{sec:tests} and \ref{sec:improvements}. Therefore, to avoid repetition, we omit those discussions here. To demonstrate the empirical performance of the resulting method, we consider the following example.
			
			\begin{exa}\label{exa:centering-power}
				We generate $50$ observations from a $d$-variate normal distribution with mean $\bf{1}_d$ and variance covariance matrix $0.7\mathrm {\bf I}_d + 0.3 \mathrm {\bf J}_d$.
			\end{exa}  
			
				We compute the power of the tests (a) when the observations are centered using the spatial median and (b) when the above idea based sample splitting is used. The results are given Figure \ref{fig:sim-9.1}. For the centered data, our tests have  higher powers than OT and DT tests. We have seen that when we use centering based on the spatial median, the DT test becomes conservative in high dimension, while our tests have a tendency to have inflated Type I error. This may be one of the reasons for the significant difference in their powers. The OT test had a very poor performance, it had almost zero power in all dimensions. This may also be due to the conservativeness of this test as observed in Figure 	\ref{fig:sim-9}. When we adopt the sample splitting idea and use the tests on the differences of the observations, the power of our tests became slightly lower than what we observed before, but still they had an edge over the DT test. 
				Surprisingly, this method helped the OT test to gain some power.
				Though its performance was inferior to other competitors, unlike before, we observed an increasing tend in its power as the dimension increases.
			
			\section{Concluding Remarks}
			
			In this article, we proposed some distribution-free methods for testing spherical symmetry of a multivariate distribution, which can be conveniently used for high-dimensional data even when the dimension is larger than the sample size. Under appropriate regularity conditions, we proved the consistency of these tests in the HDLSS and HDHSS asymptotic regimes and demonstrated their utility using several simulated and real data sets. They outperformed the state-of-the-art methods in a wide variety of high-dimensional examples.
			
			Recall that our tests were constructed using signs and ranks computed along the shortest covering path on the augmented dataset. Instead of shortest covering path, one can consider other graph based methods as well. For instance, one can consider a tree with $n-1$ edges that has the minimum cost and covers either a data point or its spherically symmetric counterpart. Though there are algorithms for constructing the minimum spanning tree, finding such a tree that covers $n$ out of $2n$ vertices (each representing one observation in the augmented data set) turns out to be an NP-complete problem. A heuristic method based on Prim's algorithm \citep{primalgo} can be used there as well, and after constructing the tree, sign statistic can be defined in the same way. This sign statistic will also have the distribution free property and its null distribution will match with that of the univariate sign statistic. The runs statistic can be computed using the idea of \cite{friedman1979multivariate}, but unfortunately the resulting test won't be distribution-free in two or higher dimensions. Similarly, one can also construct a test based on nearest neighbor type coincidences \citep[see, e.g.,][]{henze1988multivariate,schilling1986multivariate}, but that won't be distribution-free as well, and one needs to use an resampling algorithm for calibration, which will increase the computing cost. 
			
			For constructing our modified sign and runs tests, here we have used Bonferroni's method for size correction. Instead, one can also use the methods available for controlling the false discovery rate based on $p$-values \citep{benjamini1995controlling,benjamini2001control} or $e$-values \citep{wang2022false}. However, these methods did not make any visible difference in the performance of our proposed tests. 
			
			Since the main focus of this article is on high dimensional test for spherical symmetry, we did not pay much attention to the usual large sample behavior of the proposed tests in the classical asymptotic regime. However, from our discussion, it is clear that the asymptotic null distributions of the linear rank statistic and the runs statistic are same as given by Theorem \ref{thm:limit-null-distribution} and \ref{thm:large-dist-runs}. Large sample consistency of the resulting tests can be proved as well. Theorem \ref{thm:pitman-efficiency}  the Appendix also establishes the Pitman efficiency of the linear rank test, but such a result for the runs test is yet to derived.
			
			Another interesting problem would be the construction of a test for elliptic symmetry of a high dimensional probability distribution. If the sample size is large compared to dimension of the data, we can estimate the location and scatter of the underlying distribution for standardization and apply the tests of spherical symmetry on the standardized data. But this idea does not work when the dimension is larger compared to the sample size. One needs to come up with an alternative method to take care of this issue.

			\bibliographystyle{apalike}
			\bibliography{ref.bib}

			\appendix

			\section{Appendix A: Proofs}
			
			\subsection{Proofs of the results stated in Section \ref{sec:definition}}
			
			\begin{proof}[\bf Proof of Lemma \ref{lemma:distance-characterization}]
				It is easy to see that if $\Pr$ is spherically symmetric $(\bm X_1,\bm X_1^\prime)\stackrel{D}{=}(\bm X_1^\prime,\bm X_1)$. Therefore, for any measurable function $h(.,.)$, $h(\bm X_1,\bm X_2) \stackrel{D}{=} h(\bm X_1,\bm X_2^\prime) \stackrel{D}{=} h(\bm X_1^\prime,\bm X_2^\prime)$
				holds trivially. To prove the only if part, first note that
				\begin{align*}
					\P[h(\bm X_1,\bm X_2)\leq t] = \int_{\{h(\bm x,\bm y)\leq t\}} \hspace{-0.2in} p(\bm x) p(\bm y) \mathrm d\bm x\mathrm d \bm y & = \int_{\{h(\bm x-\bm y,0)\leq t\}} \hspace{-0.2in} p(\bm x) p(\bm y) \mathrm d \bm x \mathrm d \bm y
					 = \int_{\{h(\bm v,0)\leq t\}} \left[\int p(\bm y+\bm v) p(\bm y) \mathrm d \bm y\right] \mathrm d \bm v.
				\end{align*}
				By the square integrability of $p~(\geq 0)$ we can say that
				$$u(\bm v) := \int p(\bm y+\bm v) p(\bm y) \mathrm d \bm y ~~~~~~~~~~\left[\leq \Big(\int p^2(\bm y+\bm v)\mathrm d\bm y\Big)^{1/2} \Big(\int p^2(\bm y)\mathrm d\bm y\Big)^{1/2} = \int p^2(\bm y)\mathrm d\bm y<\infty\right]$$
				is locally integrable on $\R^d$ (i.e., integrable on all compact subsets of $\R^d$). Hence, by Theorem 7.15 from \cite{wheeden1977measure}, we conclude that almost every point in $\R^d$ is a Lebesgue point of $u(\cdot)$. Also, by our assumption (c) and Theorem 7.16 from \cite{wheeden1977measure}, we obtain
				\begin{align*}
					\lim_{t\downarrow 0}\frac{\int_{\{h(\bm v,0)\leq t\}} |u(\bm v)-u(0)| \mathrm d \bm v}{\int_{\{h(\bm v,0)\leq t\}} \mathrm d\bm v} = 0,
				\end{align*}
				or in other words
				\begin{align*}
					\lim_{t\downarrow 0}\frac{\P[h(\bm X_1,\bm X_2)\leq t]}{\int_{\{h(\bm v,0)\leq t\}} \mathrm d\bm v} = \lim_{t\downarrow 0}\frac{\int_{\{h(\bm v,0)\leq t\}} u(\bm v) \mathrm d \bm v}{\int_{\{h(\bm v,0)\leq t\}} \mathrm d\bm v} = u(0) = \int p^2(\bm y)\mathrm d \bm y.
				\end{align*}
				Now, the random variable $\bm X_1^\prime$ has the density $p^\prime(\bm x) =  \int p(\bm H^\top \bm x)\mathrm d\nu(\bm H)$ where $\nu$ is the Haar measure on the set of all $d\times d$ orthogonal matrices \citep[see Lemma A.2][]{banerjee2024consistent}. One can show that if $p$ is square integrable, so is $p^\prime$. Therefore, using the same argument, we can show that
				\begin{align*}
					\lim_{t\downarrow 0}\frac{\P[h(\bm X_1,\bm X_2^\prime)\leq t]}{\int_{\{h(\bm v,0)\leq t\}} \mathrm d\bm v} = \int p(\bm y) p^\prime(\bm y) \mathrm d \bm y~~~~\textnormal{and}~~~~\lim_{t\downarrow 0}\frac{\P[h(\bm X_1^\prime,\bm X_2^\prime)\leq t]}{\int_{\{h(\bm v,0)\leq t\}} \mathrm d\bm v} = \int p^{\prime2}(\bm y)\mathrm d\bm y. 
				\end{align*}
				However, under the assumption $h(\bm X_1,\bm X_2)\stackrel{D}{=}h(\bm X_1,\bm X_2^\prime)\stackrel{D}{=}h(\bm X_1^\prime,\bm X_2^\prime)$ we must have
				\begin{align*}\P[h(\bm X_1^\prime,\bm X_2^\prime)\le t] = \P[h(\bm X_1^\prime,\bm X_2^\prime)\le t] =\P[h(\bm X_1^\prime,\bm X_2^\prime)\le t]~~~\forall t\in \R.
				\end{align*}
				So, combining our results we get
				\begin{equation}\label{eq:prob-limits}
					\int p^2(\bm y) \mathrm d \bm y = \int p^{\prime2}(\bm y) \mathrm d \bm y =  \int p(\bm y) p^\prime(\bm y) \mathrm d \bm y.
				\end{equation}
				By Cauchy-Schwartz inequality, the equality in \eqref{eq:prob-limits} holds if and only if $p = p^\prime$ almost surely, i.e., $\Pr$ is spherically symmetric. 
			\end{proof}
			
			
			\begin{proof}[\bf Proof of Theorem \ref{thm:sign-rank-dist}]
				Let us define $\mathcal{T}(\bm s,\bm \pi):=\sum_{i=1}^{n-1}\theta(\bm Y_{s_{\pi_i},\pi_i},\bm Y_{s_{\pi_{i+1}},\pi_{i+1}}),$ for $\bm s=(s_1,s_2,\ldots,s_n)\in\{0,1\}^n$ and ${\bm \pi}\in {\mathcal S}_n$. Let $(\bm S,\bm \Pi)$ be as defined in \eqref{eq:optimization}. Then, for any $\bm s_0\in\{0,1\}^n$ and $\bm \pi_0\in S_n$ we have
				\begin{align}\label{eq:sign-rank-prob}
				  \P[\bm S=\bm s_0; \bm \Pi = \bm \pi_0] = \P\left[\mathcal{T}(\bm s_0,\bm \pi_0) \leq \mathcal{T}(\bm s,\bm \pi)~~\forall \bm s\in\{0,1\}^n~and~\bm \pi\in {\mathcal S}_n\right].
				\end{align}
				When $\Pr$ is spherically symmetric, we have $(\bm X_i,\bm X_i^\prime)\stackrel{D}{=}(\bm X_i^\prime,\bm X_i)$ for every $i=1,2,\ldots,n$. So, the random variables $\{\mathcal{T}(\bm s,\bm \pi)\}_{\bm s\in\{0,1\}^n; \bm \pi\in \mathcal{S}_n}$ are exchangeable. Hence, the probability on the right side of \eqref{eq:sign-rank-prob} does not depend on $\bm s_0$ and $\bm \pi_0$. Therefore, using the identity $\sum_{\bm s_0\in \{0,1\}^n}\sum_{\bm \pi_0\in S_n} \P[[\bm S=\bm s_0; \bm \Pi = \bm \pi_0] = 1$, we get
					\begin{align*}
						\P[\bm S=\bm s_0; \bm \Pi = \bm \pi_0] = 2^{-n}(n!)^{-1}.
					\end{align*}
					This implies $\bm S\sim \textnormal{Unif}(\{0,1\}^n)$, $\bm \Pi \sim\textnormal{Unif}({\mathcal S}_n)$ and they are mutually independent. Since $\bm R$ is the inverse permutation of $\bm \Pi$, we also have $\bm R\sim$Unif$({\mathcal S}_n)$ and it is independent of $\bm S$. This proves part (a) of the theorem.
					
					If $\Pr$ is not spherically symmetric, $(\bm X_i,\bm X_i^\prime)$ and $(\bm X_i^\prime, \bm X_i)$ are not distributionally equal for any $i=1,2,\ldots, n$, but $\{(\bm X_i,\bm X_i^\prime)\}_{1\leq i\leq n}$ forms a sequence of i.i.d. random vectors. So, for any $\bm\pi_0\in S_n$ we have
					\begin{align*}
						\P[\bm \Pi = \bm \pi_0] & = \sum_{\bm s_0\in\{0,1\}^n}\P[\bm S=\bm s_0; \bm \Pi = \bm \pi_0] = \sum_{\bm s_0\in\{0,1\}^n}\P[\bm \Pi = \bm \pi_0 \mid \bm S=\bm s_0] ~\P[\bm S=\bm s_0].
					\end{align*}
					For any given $\bm s_0\in\{0,1\}^n$ the random variables $\{\mathcal{T}(\bm s_0,\bm \pi)\}_{\bm \pi\in {\mathcal S}_n}$ are exchangeable. Hence, $\P[\bm \Pi = \bm \pi_0|\bm S = \bm s_0]$ is constant for every $\bm \pi_0\in \mathcal{S}_n$. Using the relation $\sum_{\bm \pi_0\in S_n} \P[\bm \Pi = \bm \pi_0|\bm S = \bm s_0] = 1$ for any $\bm s_0\in \{0,1\}^n$, we get $\P[\bm \Pi = \bm \pi_0|\bm S = \bm s_0] = (n!)^{-1}$ and therefore $\P[\bm \Pi = \bm \pi_0] = (n!)^{-1}$. Here also, by the same argument as above, we have $\bm R\sim$ Unif$({\mathcal S}_n)$. Now, given $\bm \Pi = (\pi_1,\ldots, \pi_n)$ and the augmented data $\{(\bm X_i,\bm X_i^\prime)\}_{1\leq i\leq n}$ we have the following relation.
					\begin{align*}
						\mathrm \mathrm{I}\{S_{\pi_1}=1\}& = \mathrm I\big\{\theta(\bm X_{\pi_1},\bm Y_{S_{\pi_2},\pi_2})\leq \theta(\bm X_{\pi_1}^\prime,\bm Y_{S_{\pi_2},\pi_2})\big\}.
					\end{align*}
				This follows from the fact that the total cost of the path starting with $\bm X_{\pi_1}$ and that starting with  $\bm X^{\prime}_{\pi_1}$ differ only in the first term. Using similar arguments, we also get
					\begin{align*}
						 \mathrm{I}\{S_{\pi_k}=1\} & = \mathrm I\big\{\theta(\bm Y_{S_{\pi_{(k-1)}},\pi_{(k-1)}},\bm X_{\pi_k})+\theta(\bm X_{\pi_k},\bm Y_{S_{\pi_{(k+1)}},\pi_{(k+1)}})\\
						& \hspace{0.5in}\leq \theta(\bm Y_{S_{\pi_{(k-1)}},\pi_{(k-1)}},\bm X_{\pi_{k}}^\prime)+\theta(\bm X_{\pi_{k}}^\prime,\bm Y_{S_{\pi_{(k+1)}},\pi_{(k+1)}})\big\} ~\text{for } k=2,\ldots,n-1~~\text{and}\\
					 \mathrm{I}\{S_{\pi_n}=1\} & = \mathrm I\big\{\theta(\bm Y_{S_{\pi_{n-1}},\pi_{n-1}},\bm X_{\pi_n})\leq \theta(\bm Y_{S_{\pi_{n-1}},\pi_{n-1}},\bm X_{\pi_n}^\prime)\big\}.
					\end{align*}
					Now, taking conditional expectation given $\bm \Pi$, we have
     				\begin{align*}
						\P^*[S_{\pi_1}=1] & = \P^*\big[\theta(\bm X_{\pi_1},\bm Y_{S_{\pi_2},\pi_2})\leq \theta(\bm X_{\pi_1}^\prime,\bm Y_{S_{\pi_2},\pi_2})\big],\\
						\P^*[S_{\pi_k}=1] & = \P^*\big[\theta(\bm Y_{S_{\pi_{(k-1)}},\pi_{(k-1)}},\bm X_{\pi_k})+\theta(\bm X_{\pi_k},\bm Y_{S_{\pi_{(k+1)}},\pi_{(k+1)}})\\
						& \hspace{0.5in}\leq \theta(\bm Y_{S_{\pi_{(k-1)}},\pi_{(k-1)}},\bm X_{\pi_{k}}^\prime)+\theta(\bm X_{\pi_{k}}^\prime,\bm Y_{S_{\pi_{(k+1)}},\pi_{(k+1)}})\big] ~\text{for } k=2,\ldots,n-1~~\text{and}\\
						\P^*[S_{\pi_n}=1] & = \P^*\big[\theta(\bm Y_{S_{\pi_{n-1}},\pi_{n-1}},\bm X_{\pi_n})\leq \theta(\bm Y_{S_{\pi_{n-1}},\pi_{n-1}},\bm X_{\pi_n}^\prime)\big].
						\end{align*}
						where $\P^*$ denotes the conditional distribution given $\bm\Pi$. This establishes the weak dependence structure of the string signs. This completes the proof of part (b) of the theorem.
			\end{proof}

			\subsection{Proofs of the results stated in Section \ref{sec:tests}}
			\begin{proof}[\bf Proof of Theorem \ref{thm:HDLSS-difficulty}]
				For the proof of this theorem, we refer the reader to Theorem 1 in \cite{jung2009pca}. 
			\end{proof}
			
			\begin{proof}[\bf Proof of Theorem \ref{thm:HDLSS-consistency}]
		Let $\bm X_1^\prime = \|\bm X_1\|\bm U_1 $ and $\bm X_2^\prime = \|\bm X_2\|\bm U_2$ be the spherically symmetric variants of $\bm X_1$ and $\bm X_2$, respectively. Here $\bm U_1,\bm U_2$ are i.i.d. Unif$(\mathcal{S}^{d-1})$ independent of $\bm X_1$ and $\bm X_2$. Then, we have 
		\begin{align*}
			\P\left[\theta(\bm X_1,\bm X_2)<\theta(\bm X_1,\bm X_2^\prime)\right] 
			& = \P\left[(\bm X_1^\top\bm X_2)^2>(\bm X_1^\top\bm X_2^\prime)^2\right]
		    = \P\left[(\bm X_1^\top\bm X_2)^2> \|\bm X_2\|^2 (\bm X_1^\top\bm U_2)^2\right].  
		\end{align*}
		By spherical symmetry of $\bm U_2$, for any $\bm a \in {\mathbb R}^d$, we have $(\bm a^\top \bm U_2) \stackrel{D}{=} \|\bm a\| U_{21}$, where $U_{21}$ is the first coordinate of $\bm U_2$. Therefore, conditioned on $\bm X_1$ and $\bm X_2$ we have,
		\begin{align*}
			\P\left[\theta(\bm X_1,\bm X_2)<\theta(\bm X_1,\bm X_2^\prime)\Big|\bm X_1,\bm X_2\right] = \P\left[(\bm X_1^\top\bm X_2)^2> \|\bm X_2\|^2 \|\bm X_1\|^2 U_{21}^2 \Big|\bm X_1\bm X_2\right].  
		\end{align*}
		Now, taking expectations with respect to $\bm X_1$ and $\bm X_2$, we get
		\begin{align*}
			\P\left[\theta(\bm X_1,\bm X_2)<\theta(\bm X_1,\bm X_2^\prime)\right] = \P\left[ \frac{(\bm X_1^\top\bm X_2)^2}{\|\bm X_2\|^2 \|\bm X_1\|^2} >  U^2_{21} \right].  
		\end{align*}
		From the elementary theory of sampling distributions we know that $U_{21}^2$ follows a Beta$(\frac{1}{2}, \frac{d-1}{2})$ distribution over the interval $[0,1]$ \citep[see, e.g.][]{liang2008}. So, we have $\E[U_{21}^2] = \frac{1}{d}$ and $\var[U_{21}^2] = \frac{2}{d(d+2)}$. Therefore, $\{dU_{21}^2\}_{d\geq 1}$ is a tight sequence of random variables. Hence, for all $M>0$, if we have
		\begin{align*}
			\P\left[\frac{d(\bm X_1^\top\bm X_2)^2}{\|\bm X_1\|^2\|\bm X_2\|^2}>M\right]\rightarrow 1,
		\end{align*}
		we can say that $\P\Big[\frac{d(\bm X_1^\top \bm X_2)^2}{\|\bm X-1\|^2\|\bm X_2\|^2}> dU_{21}^2\Big] = \P\left[\theta(\bm X_1,\bm X_2)<\theta(\bm X_1,\bm X_2^\prime)\right]$ converges to one as $d$ diverges to infinity.  Similarly, we can also show that under the given condition, $\P\left[\theta(\bm X_1,\bm X_2)<\theta(\bm X_1^\prime,\bm X_2)\right]$ and $\P\left[\theta(\bm X_1,\bm X_2)<\theta(\bm X_1^\prime,\bm X_2^\prime)\right]$ also converge to one as $d$ diverges to infinity. 
		Now, for any $\bm \pi\in \mathcal{S}_n$, define $$E_i= \Big\{\theta(\bm X_{\pi_i},\bm X_{\pi_{i+1}})=\min\big\{\theta(\bm X_{\pi_i},\bm X_{\pi_{i+1}}),~\theta(\bm X_{\pi_i},\bm X^\prime_{\pi_{i+1}}),~\theta(\bm X^\prime_{\pi_i},\bm X_{\pi_{i+1}}),~\theta(\bm X^\prime_{\pi_i},\bm X^\prime_{\pi_{i+1}})\Big\}$$
		for $i=1,\ldots,n-1$. Clearly, $\P(E_i\mid \bm \Pi = {\bm \pi})\rightarrow 1$ for all $i=1,\ldots,n-1$ as $d \rightarrow \infty$.
		Now, let $\bm S$ be the solution of \eqref{eq:optimization} and define ${\bf 1}_n$ as the identity sequence that has all elements equal to $1$. Note that for any $\bm \pi\in \mathcal{S}_n$, 
		$$\{{\bm S}={\bf 1}_n\} \supseteq
		E_1 \cap E_2\cap \ldots\cap E_{n-1}~\text{and hence } \P({\bm S}={\bf 1}_n \mid \bm \Pi = \bm \pi) \rightarrow 1 ~\text{as } n \rightarrow \infty.$$
		Now, the result follows by a simple application of the Dominated Convergence Theorem.
	\end{proof}
			
			\begin{proof}[\bf Proof of Corollary \ref{cor:spiked-HDLSS}]
		 Here $\bm Z_i$ ($i=1,2$) can be viewed as a standardized version of $\bm X_i$, and we have $\textnormal{Var}(\bm Z_i)={\bf I}_d$, the $d \times d$ identity matrix. Now,
			\begin{align*}
			\vspace{-0.1in}
				(\bm X_1^\top \bm X_2) = \bm Z_1^\top \bm \Lambda_d\bm Z_2 = \sum_{i=1}^d \lambda_{i} Z_{1,i} Z_{2,i}\text{ and }(\bm X_1^\top \bm X_1) = \bm Z_1^\top \bm \Lambda \bm Z_1 = \sum_{i=1}^d \lambda_{i} Z_{1,i}^2.
		\end{align*}
	
	\vspace{-0.25in}
			Note that
			\begin{align*}
				 \var\left[\frac{\sum_{i=2}^d \lambda_{i} Z_{1,i} Z_{2,i}}{\sum_{i=2}^d \lambda_{i}}\right]
				= \frac{1}{\Big(\sum_{i=2}^d \lambda_{i}\Big)^2} \left[\sum_{i=2}^d \lambda_{i}^2 \var\Big[Z_{1,i} Z_{2,i}\Big] + \sum_{2\leq i\not=j\leq d}\lambda_{i}\lambda_{j} \textnormal{Cov}\Big(Z_{1,i} Z_{2,i},~ Z_{1,j} Z_{2,j}\Big)  \right].
			\end{align*}
			Since $\textnormal{Cov}\Big(Z_{1,i} Z_{2,i},~ Z_{1,j} Z_{2,j}\Big) = \E\left[Z_{1,i} Z_{2,i} Z_{1,j} Z_{2,j}\right] = \left[\E\big[Z_{1,i} Z_{1,j}\big]\right]^2  = 0.
			$ for all $i \neq j ~(i,j \ge 2)$, we have
			\begin{align*}
				& \var\left[\frac{\sum_{i=2}^d \lambda_{i} Z_{1,i} Z_{2,i}}{\sum_{i=2}^d \lambda_{i}}\right] = \frac{1}{\Big(\sum_{i=2}^d \lambda_{i}\Big)^2}\sum_{i=2}^d \lambda_{i}^2 \var\Big[Z_{1,i} Z_{2,i}\Big] =   \frac{\sum_{i=2}^d \lambda_{i}^2}{\Big(\sum_{i=2}^d \lambda_{i}\Big)^2},
			\end{align*}
			So, using assumption (b) of the corollary, as $d$ diverges to infinity, we get
			$$\sum_{i=2}^d \lambda_{i} Z_{1,i}Z_{2,i}= o_P\Big(\sum_{i=2}^d \lambda_{i}\Big) \text{ and hence }
			(\bm X_1^\top\bm X_2) = \lambda_{1} Z_{1,1}Z_{2,1} + o_P\Big(\sum_{i=2}^d\lambda_{i}\Big).$$Also, note that under the $\rho$-mixing condition in (A2) and  assumption (b) of the corollary,
			\begin{align*}
				\var\left[\frac{\sum_{i=2}^d \lambda_{i} Z_{1,i}^2}{\sum_{i=2}^d \lambda_{i}}\right] & = \frac{1}{\Big(\sum_{i=2}^d \lambda_{i}\Big)^2} \left[\sum_{i=2}^d \lambda_{i}^2 \var\Big[Z_{1,i}^2\Big] + \sum_{2\leq i\not=j\leq d} \lambda_{i}\lambda_{j} \textnormal{Cov}\Big(Z_{1,i}^2,Z_{1,j}^2\Big)\right] \rightarrow 0 
			\end{align*}
			as $d \rightarrow \infty$ (follows from Theorem 1 from \cite{jung2009pca}). So, under these assumptions, we get
			\begin{align*}
				(\bm X_1^\top\bm X_1)  = \lambda_{1,d} Z_{1,1}^2 + \sum_{i=2}^d \lambda_{i} + o_P\Big(\sum_{i=2}^d\lambda_{i}\Big).~\Big(\text{note that} E\Big[\sum_{i=2}^{d}Z_{1,i}^2\Big]=\sum_{i=2}^{d}\lambda_i\Big)
			\end{align*}
			Now if $\alpha>1$, using assumptions (a) and (b), as $d$ diverges to infinity, we get
			\begin{align*}
				\frac{(\bm X_1^\top\bm X_2)^2}{\|\bm X_1\|^2\|\bm X_2\|^2} =  \frac{(\bm X_1^\top\bm X_2)^2/d^{2\alpha}}{\|\bm X_1\|^2/d^{\alpha}\|\bm X_2\|^2/d^{\alpha}} \rightarrow\frac{c_1^2(Z_{1,1}Z_{2,1})^2}{c_1 Z_{1,1}^2 c_1 Z_{2,1}^2} = 1.
			\end{align*}
			\noindent For $\alpha = 1$, as $d$ diverges to infinity, for any $\epsilon>0$, we have
			\begin{align*}
				\P\Big[\Big|\frac{\bm X_1^\top\bm X_2}{d} - c_1Z_{1,1} Z_{2,1}\Big|>\epsilon\Big]\rightarrow 0 \text{ and } \P\Big[\Big| \frac{\bm X_1^\top \bm X_1}{d} -c_1Z_{1,1}^2 - \frac{\sum_{i=2}^{d} \lambda_i}{d} \Big| >\epsilon\Big] \rightarrow 0.
			\end{align*} 
			Since $\sum_{i=2}^d \lambda_{i}= \mathcal{O}(d)$,  $c_2 = \lim\sup_{d \rightarrow \infty}\sum_{i=2}^d \lambda_{i}/d$.  is non-negative and finite, and for large $d$, we have
				
			$$P\left(\frac{(\bm X_1^\top\bm X_2)^2}{\|\bm X_1\|^2\|\bm X_2\|^2}>t\right) \ge P\left( \frac{c^2_1Z_{1,1}^2Z_{2,1}^2}{(c_1Z_{1,1}^2+c_2)(c_1Z_{2,1}^2+c_2)}>t\right) \text{ for all } t\ge 0.$$ Note that on the right side we have a random variable which takes positive values with probability one.
		This implies both for $\alpha=1$ and  $\alpha > 1$, as $d$ diverges to infinity, we have $$P\Big[\frac{d(\bm X_1^\top\bm X_2)^2}{\|\bm X_1\|^2\|\bm X_2\|^2}>M\big] \rightarrow 1 \text{ for any } M>0.$$ Hence, under the given assumptions, the result follows from Theorem \ref{thm:HDLSS-consistency}.
		\end{proof}

			\begin{proof}[\bf Proof of Theorem \ref{thm:limit-null-distribution}]
				Let $\bm S$ and $\bm R$ be the sign and rank vectors, and $\bm \Pi$ be the vector of anti-ranks (inverse permutation of ${\bm R}$) as defined in Section \ref{sec:definition}. Then, by Theorem \ref{thm:sign-rank-dist}, $\bm S\sim$ Unif$(\{0,1\}^n)$ and ${\bm \Pi}\sim$ Unif$(\mathcal{S}_n)$ irrespective of the dimension $d$. Now, to find the distribution of the linear rank statistic $T_{LR} = \sum_{i=1}^n S_{\pi_i} a(i)$, first note that $\{a(i)(S_{\pi_i}-\frac{1}{2})\}\}_{1\leq i\leq n; n\geq 1}$ forms a triangular array of independent random variables whose distribution does not depend on the dimension of the data. Also, we have  $$\zeta_n^2 := \sum_{i=1}^n \var\big[a(i)(S_{\pi_i}-{1}/{2})\big] = \sum_{i=1}^n a^2(i) \var(S_{\pi_i}) = \frac{1}{4}\Big[\sum_{i=1}^n a^2(i)\Big]~~~\textnormal{ (since $\var[S_{\pi_i}] = {1}/{4}~~\forall i=1(1)n$)}.$$
				Since the sequence of scores $\{a(i)\}$ that satisfies \eqref{eq:consistency-cond}, as $n$ and $d$ both diverge to infinity, we have 
				\begin{align}\label{eq:suff-derivation}
					& \frac{1}{\zeta_n^{3}} \sum_{i=1}^n |a(i)|^3 \E\left[\Big|S_{\pi_i}-{1}/{2}\Big|^3\right] \leq  \frac{\sum_{i=1}^n |a(i)|^3}{\big(\sum_{i=1}^n a^2(i)\big)^{3/2}} \leq\max_{1\leq i\leq n}\frac{|a(i)|}{\big(\sum_{i=1}^n a^2(i)\big)^{1/2}} \rightarrow 0, 
				\end{align}
				
				%
				 So, the triangular array $\{a(i)(S_{\pi_i}-{1}/{2})\}\}_{1\leq i\leq n; n\geq 1}$ satisfies the Lyapunov's condition. Hence, an application of the Lyapunov's Central Limit Theorem (CLT) leads to our desired result.
			\end{proof}

			\begin{proof}[\bf Proof of Theorem \ref{thm:limit-alt-convergence}]
				Let ${\bm S}$, ${\bm R}$ and ${\bm \Pi}$ be 
				as defined in Section \ref{sec:definition}. Then, the lvariance of
				the linear rank statistic $T_{LR} = \sum_{i=1}^n S_{\pi_i} a(i)$ is given by
				\begin{align*}
					\var\left[T_{LR}\right] = \sum_{i=1}^n a^2(i) \var(S_{\pi_i}) + \sum_{1\leq i\not=j \leq n} a(i) a(j) \textnormal{Cov}(S_{\pi_i},S_{\pi_j}).
				\end{align*}
				From Theorem \ref{thm:sign-rank-dist}, we have $\textnormal{Cov}(S_{\pi_i},S_{\pi_j}) = 0$ for all $|i-j|>1$ irrespective of the dimension $d$. Therefore,
				\begin{align*}
					\var[T_{LR}] = \sum_{i=1}^n a^2(i) \var(S_{\pi_i}) + \sum_{|i-j| = 1} a(i) a(j) \textnormal{Cov}(S_{\pi_i},S_{\pi_j}).
				\end{align*}
				
				Now, the uniformity of $\bm \Pi$
				(see Theorem \ref{thm:sign-rank-dist}) suggests that  $\var(S_{\pi_i})$ is constant for all $i=1,2,\ldots,n$ and $\textnormal{Cov}(S_{\pi_i},S_{\pi_{j}})$ is constant for all $i,j$ with $|i-j|=1$. Let us call them $\sigma_{11}^{(d)}$ and $\sigma_{12}^{(d)}$, respectively.  Note that $\sigma_{11}^{(d)}$ and $\sigma_{12}^{(d)}$ are bounded and $\sum_{|i-j|=1} a(i)a(j) = 2\sum_{i=1}^{n-1}a(i)a(i+1)$. Therefore, as $n$ and $d$ diverges to infinity, for a sequence of scores $\{a(i)\}_{1\leq i\leq n}$ satisfying \eqref{eq:consistency-cond}, we have
				$$
					\limsup_{n,d\rightarrow\infty} \var\left[\frac{T-\E[T]}{\sqrt{\sum_{i=1}^n a^2(i)}}\right] = \limsup_{n,d\rightarrow\infty} \left[\frac{\sum_{i=1}^{n} a^2(i) \sigma_{11}^{(d)}}{\sum_{i=1}^n a^2(i)} + 2\frac{\sum_{i=1}^{n-1}a(i)a(i+1)}{\sum_{i=1}^n a^2(i)} \sigma_{12}^{(d)}\right] 
					 \rightarrow \sigma_{11} + 2C\sigma_{12},
				$$
				where $\sigma_{11} = \limsup_{d\rightarrow \infty} \sigma_{11}^{(d)} \leq 1$ and $\sigma_{12} = \limsup_{d\rightarrow\infty} \sigma_{12}^{(d)} \leq 1$ are finite constants. This gives us our desired result.
	\end{proof}

			\begin{proof}[\bf Proof of Theorem \ref{thm:large-dist-runs}]
				Let $\bm S$, $\bm R$ and $\bm \Pi$ be vectors of string signs, string ranks and anti-ranks as defined in Section \ref{sec:definition}. Recall that the runs statistic is given by 
				$$T_R = 1 + \sum_{i=1}^{n-1} \mathrm{I}\{S_{\pi_i}\not = S_{\pi_{i+1}}\}.$$
				When $\Pr$ is spherically symmetric, $\bm S\sim$ Unif$(\{0,1\}^n)$, $\bm \Pi\sim$ Unif$(\mathcal{S}_n)$ and they are independent (see Theorem \ref{thm:sign-rank-dist}). Let us define the filtration $\mathcal{F}^{(d)}_{n,t} = \mathcal{C}(S_{\pi_1},\ldots, S_{\pi_t})$ for all $1\leq t\leq n$ and $\mathcal{F}^{(d)}_{n,0}$ and let be the trivial $\sigma$-field. Note that the sequence $\{V^{(d)}_{n,i}\}_{0\leq i\leq n, n,d\geq 1}$, with $V^{(d)}_{n,0}=0$ and
				$V^{(d)}_{n,t} = \sum_{i=1}^{t} \mathrm{I}\{S_{\pi_i}\not = S_{\pi_{i+1}}\} - \frac{t}{2}$ for all $1\leq t\leq n$, forms a triangular array of martingales adapted to the sequence of filtration $(\mathcal{F}^{(d)}_{n,t})_{1\leq t\leq n; n, d\geq 1}$. So, we can use the martingale central limit theorem (CLT) \citep[see][]{brown1971mclt} to derive the limiting null distribution of $T_R$. 
				
				First, let us look at the triangular array of martingale difference $$Y^{(d)}_{n,i} = V^{(d)}_{n,i} - V^{(d)}_{n,i-1} = \mathrm{I}\{S_{\pi_i}\not = S_{\pi_{i+1}}\} -{1}/{2}, ~~i = 1,2,\ldots,n.$$ One can see that $\big(\sigma^{(d)}_{n,i}\big)^2 = \E[\big(Y^{(d)}_{n,i}\big)^2|\mathcal{F}_{n,i}^{(d)}] = {1}/{4}$ for all $i=1,2,\ldots,n$. Therefore, $\big(\zeta^{(d)}_n\big)^2 = \sum_{i=1}^{n} \big(\sigma_{n,i}^{(d)}\big)^2 = {n}/{4}$ is a deterministic sequence of real numbers. Also, using $|Y^{(d)}_{n,i}|\leq {1}/{2}$, for any $\epsilon>0$ we get
				\begin{align*}
					\frac{1}{\big(\zeta^{(d)}_n\big)^{2}} \sum_{j=1}^n \E\left[\big(Y^{(d)}_{n,j}\big)^2 \mathrm{I}[|Y^{(d)}_{n,j}|\geq \epsilon \zeta^{(d)}_n]\right] \leq \frac{1}{n} \sum_{j=1}^n \E\left[\mathrm{I}\Big[|Y^{(d)}_{n,j}|\geq \epsilon \frac{n}{4}\Big]\right] = \P\left[|Y^{(d)}_{n,1}|\geq \epsilon \frac{n}{4}\right]\rightarrow 0
				\end{align*}
				as $n$ and $d$ diverges to infinity. So, applying the martingale CLT \citep[Theorem 1 from][]{brown1971mclt}, we get
				\begin{align*}
					\frac{V^{(d)}_n}{\zeta^{(d)}_n} \stackrel{D}{\rightarrow} \mathcal{N}(0,1) \textnormal{ or equivalently } \frac{V^{(d)}_n}{\sqrt{n}} \stackrel{D}{\rightarrow} \mathcal{N}\left(0,\frac{1}{4}\right), \text{as $n,d \rightarrow \infty.$}
				\end{align*}
				 Since $T_R = 1+ V_{n-1} + \frac{n-1}{2} = V_{n-1} + \frac{n+1}{2}$, we have $n^{-1/2}({T_R - \frac{n+1}{2}})\stackrel{D}{\rightarrow}  \mathcal{N}\left(0,\frac{1}{4}\right)$ 
				as $n,d \rightarrow \infty$. 
				\end{proof}
				
				\noindent
				\begin{proof}[\bf Proof of Theorem \ref{thm:large-dist-runs-alt}] When $\Pr$ is not spherically symmetric, for $T_R = 1 + \sum_{i=1}^{n-1} \mathrm I[S_{\pi_i}\not = S_{\pi_{i+1}}]$, we have
				\begin{align*}
					\var\left[\frac{T_R-\E[T_R]}{\sqrt{n}}\right] = \frac{1}{n}\left[\sum_{i=1}^{n-1} \var\big[\mathrm{I}\{S_{\pi_i}\not=S_{\pi_{i+1}}\}\big] + \sum_{1\leq i\not = j\leq n-1} \textnormal{Cov}\big(\mathrm{I}\{S_{\pi_i}\not=S_{\pi_{i+1}}\}, \mathrm{I}\{S_{\pi_j}\not=S_{\pi_{j+1}}\}\big) \right].
				\end{align*}
				Now by Theorem \ref{thm:sign-rank-dist},
				$
					\textnormal{Cov}\big(\mathrm{I}\{S_{\pi_i}\not=S_{\pi_{i+1}}\}, \mathrm{I}\{S_{\pi_j}\not=S_{\pi_{j+1}}\}\big) = 0\textnormal{ for all $|i-j|>2$.}
				$.
				So, as $n,d \rightarrow \infty$, \begin{align*}
					& \limsup_{n,d\rightarrow\infty} \var\left[\frac{T_R-\E[T_R]}{\sqrt{n}}\right]\\
					& = \limsup_{n,d\rightarrow\infty}\frac{1}{n}\left[\sum_{i=1}^{n-1} \var\big[\mathrm{I}\{S_{\pi_i}\not=S_{\pi_{i+1}}\}\big] + \sum_{|i-j|\leq 2} \textnormal{Cov}\big(\mathrm{I}\{S_{\pi_i}\not=S_{\pi_{i+1}}\}, \mathrm{I}\{S_{\pi_j}\not=S_{\pi_{j+1}}\}\big) \right]
					\rightarrow \sigma^2,
				\end{align*}
				\vspace{0.05in}
				
				where $\sigma^2 = \limsup\limits_{d\to\infty}\Big[ \var\big[\mathrm{I}\{S_{\pi_1}\not=S_{\pi_2}\}\big] + \textnormal{Cov}\big(\mathrm{I}\{S_{\pi_1}\not=S_{\pi_2}\}, \mathrm{I}\{S_{\pi_2}\not=S_{\pi_3}\}\big)+ \textnormal{Cov}\big(\mathrm{I}\{S_{\pi_1}\not=S_{\pi_2}\}, \mathrm{I}\{S_{\pi_3}\not=S_{\pi_4}\}\big)\Big]$ which is a non-negative finite constant. This gives us our desired result.
			\end{proof}

			\subsection{Proofs of the results stated in Sections \ref{sec:improvements} and \ref{sec:centering-effect}}
			\begin{proof}[\bf Proof of Theorem \ref{thm:modified-rank-tests}]
	Part (a) of the theorem can be proved using arguments identical to the proof of Theorem \ref{thm:sign-rank-dist}. Therefore, to avoid repetition, we omit it and give a detailed proof of part (b) only. Note that

\noindent
\begin{align*}
	\E\left[\frac{1}{d^{\alpha}}\sum_{j=1}^d (\bm X_1)_j^2 (\bm X_2)_j^2 \right] & = \frac{1}{d^\alpha} \textnormal{Tr}({\bf D}^2) \textnormal{ and }\\ 
	\E\left[\frac{1}{d^\alpha}\sum_{j=1}^d (\bm X_1)_j^2 (\bm X_2^\prime)_j^2 \right] & = \E\left[\frac{1}{d^\alpha} \|\bm X_2\|^2\sum_{j=1}^d (\bm X_1)_j^2 (\bm U_2)_j^2 \right]
	= \frac{1}{d^\alpha} \E\left[ \|\bm X_2\|^2\right] \E\left[\sum_{j=1}^d (\bm X_1)_j^2 (\bm U_2)_j^2 \right] \\
	 & = \frac{1}{d^\alpha} \textnormal{Tr}({\bf D}) \left[ \frac{1}{d}\sum_{j=1}^d \E(\bm X_1)_j^2 \right] = \frac{1}{d^{1+\alpha}} \Big(\textnormal{Tr}({\bf D})\Big)^2.
\end{align*}
Similarly one can also show that $\E\left[\frac{1}{d^{\alpha}}\sum_{j=1}^d (\bm X_1^\prime)_j^2(\bm X_2^\prime)_j^2\right] = \frac{1}{d^{1+\alpha}}\Big(\textnormal{Tr}({\bf D})\Big)^2$.
Now, by Assumption (A3), 
$	\left|\frac{1}{d^\alpha}\sum_{j=1}^d (\bm X_1)_j^2 (\bm X_2)_j^2 - \frac{1}{d^\alpha} \textnormal{Tr}({\bf D}^2)\right|, 
$
$	\left|\frac{1}{d^\alpha}\sum_{j=1}^d (\bm X_1)_j^2(\bm X_2^\prime)_j^2 -\frac{1}{d^{1+\alpha}}\big(\textnormal{Tr}(D)\big)^2 \right|$ and  $\left|\frac{1}{d^\alpha}\sum_{j=1}^d (\bm X_1^\prime)_j^2(\bm X_2^\prime)_j^2 \right.$ \linebreak $ - \left. \frac{1}{d^{1+\alpha}}\big(\textnormal{Tr}(D)\big)^2 \right|$ converges to $0$
 in probability as $d$ diverges to infinity. Therefore, if Assumption (A3) holds and $$\liminf_{d\rightarrow\infty}\left\{\frac{1}{d^\alpha}\textnormal{Tr}(D^2)-\frac{1}{d^{1+\alpha}}\Big( \textnormal{Tr}(D)\Big)^2\right\}>0,$$
	we have $\P\left[\tilde{\theta}(\bm X_1,\bm X_2)<\tilde{\theta}(\bm X_1,\bm X_2^\prime)\right]\rightarrow 1$ as $d\rightarrow\infty$. Similarly, $\P\left[\tilde{\theta}(\bm X_1,\bm X_2)<\tilde{\theta}(\bm X_1^\prime,\bm X_2)\right]\rightarrow 1$ and $\P\left[\tilde{\theta}(\bm X_1,\bm X_2)<\tilde{\theta}(\bm X_1^\prime,\bm X_2^\prime)\right]\rightarrow 1$ as $d$ diverges to infinity. 
Now, the result can be proved using the same argument as used in the proof of Theorem \ref{thm:HDLSS-consistency}. 
			\end{proof}
			
			\begin{proof}[\bf Proof of Theorem \ref{thm:HDLSS-boosted-tests}]
				First note that the tests statistics $T_{S}$ and ${\widetilde T}_{S}$ take values in $\{0,1,\ldots, n\}$ and $T_{R},{\widetilde T}_{R}$ take values in $\{1,2,\ldots, n\}$. Now, if condition (a) holds, then by Theorem \ref{thm:HDLSS-consistency} we get $T_{S}\stackrel{P}{\rightarrow} n$ as $d \rightarrow \infty$, and by Theorem \ref{thm:large-dist-runs}, we get  $T_{R}\stackrel{P}{\rightarrow} 1$ as $d\rightarrow \infty$. Whereas, if condition (b) holds, then by Theorem \ref{thm:modified-rank-tests} (b) we get ${\widetilde T}_{S}\rightarrow n$ and ${\widetilde T}_{R}\rightarrow 1$ in probability as $d$ diverges to infinity. So, if either of conditions (a) or (b) holds, then $T_{S}^M = \max\{T_{S},  {\widetilde T}_{S}\}\stackrel{P}{\rightarrow} n$ and $T_{R}^M = \min\{T_{R}, {\widetilde T}_{R}\}\stackrel{P}{\rightarrow} 1$ as $d \rightarrow \infty$.
			\end{proof}
			
			\begin{proof}[\bf Proof of Theorem \ref{thm:HDHSS-modified-tests-null}]
				Let $\bm S$ and $\bm \Pi$ be as defined in \eqref{eq:optimization} and $\tilde{\bm S}$ and $\tilde{\bm \Pi}$ be as defined in \eqref{eq:optimization-1}. Under spherical symmetry of $\Pr$ it is easy to see that the elements of the sequence $\{(S_i,\tilde{S}_i)\}_{1\leq i\leq n}$ are mutually independent. However, for each $i=1,\ldots,n$, $S_i$ and $\tilde{S}_i$ may be dependent.
				
				Now, to find the joint limiting distribution of $(T_S, \tilde{T}_S) = \sum_{i=1}^n (S_i,\tilde{S}_i)$, we first find the joint distribution of $\frac{1}{\sqrt{n}} \sum_{i=1}^n\big(t_1 (S_i-\frac{1}{2}) + t_2(\tilde{S}_i-\frac{1}{2})\big)$ for $t_1,t_2\in \R$. Note, $\{W_{i}^{(d)}:= \big(t_1 (S_i-\frac{1}{2}) + t_2(\tilde{S}_i-\frac{1}{2})\big)\}_{1\leq i\leq n}$ forms a triangular array of row wise i.i.d. bounded random variables where $\E[W_{1}^{(d)}] = 0$ and
				\begin{align*}
					\var[W_{1}^{(d)}] = t_1^2 \var[S_1]+ t_2^2 \var[\tilde{S}_1] + 2~t_1~t_2~ \textnormal{Cov}(S_1, \tilde{S}_1) =  \frac{t_1^2}{4}+   \frac{t_2^2}{4} + 2~t_1~t_2~ \textnormal{Cov}(S_1, \tilde{S}_1).
				\end{align*}
				
				We know that for a row wise i.i.d. triangular array of bounded random variables the Lyapunov's condition holds trivially. Therefore, by Lyapunov's CLT, we get
				\begin{align*}
					\frac{1}{\sqrt{n}} \sum_{i=1}^n W_{i}^{(d)} \stackrel{D}{\rightarrow} N(0,\sigma^2),
				\end{align*}  
				as $n$ and $d$ diverge to infinity, where $\sigma^2 = \lim_{d\rightarrow \infty} \var[W_{1}^{(d)}] =  \frac{t_1^2}{4}+ \frac{t_2^2}{4} + 2~t_1~t_2~ (\sigma_s^2-\frac{1}{4}).$ Since, this distributional convergence holds irrespective of $t_1,t_2\in \R$, using Cramer-Wold device,  we get
				\begin{align*}
					\frac{1}{\sqrt{n}} \left\{(T_S,\tilde{T}_S)-\Big(\frac{n}{2},\frac{n}{2}\Big)\right\}=\frac{1}{\sqrt{n}} \sum_{i=1}^n \left\{\Big(S_i, \tilde{S}_i\Big) - \Big(\frac{1}{2},\frac{1}{2}\Big)\right\} \stackrel{D}{\rightarrow} N(\bm 0, \bm \Sigma_S),
				\end{align*} 
				where the diagonal elements of $\bm \Sigma_S$ are $\frac{1}{4}$ and the off-diagonal element is $\big(\sigma_s^2 - \frac{1}{4}\big)$. Now applying continuous mapping theorem we get,
				\begin{align*}
					\frac{1}{\sqrt{n}} \left\{\max(T_S,\tilde{T}_S)-\frac{n}{2}\right\} \stackrel{D}{\rightarrow} \max\{Z_1,Z_2\},
				\end{align*} 
				where $(Z_1,Z_2)\sim N(\bm 0,\bm \Sigma_S)$. This completes part (a) of the theorem.
				
				For finding the joint limiting distribution of $(T_R, \tilde{T}_R)$ note that when $\Pr$ is a spherically symmetric, $\{M_{n,i}^{(d)}: t_1 \big(\mathrm{I}\{S_{\pi(i)}\not = S_{\pi(i+1)}\}-\frac{1}{2}\big) + t_2 \big(\mathrm{I}\{\tilde{S}_{\tilde{\pi}(i)}\not = \tilde{S}_{\tilde{\pi}(i+1)}\}-\frac{1}{2}\big)\}_{1\leq i\leq n}$ forms a triangular array of martingale differences with respect to the filtration $\{\mathcal{F}_{n,t}^{(d)} = \mathcal{C}(\{S_{\pi(i)},\tilde{S}_{\tilde{\pi}(i)}\}_{1\leq i\leq t})\}_{1\leq t\leq n, n\geq 1, d\geq 1}$. Here, for all $1\leq i\leq n$,
				\begin{align*}
					\sigma_{n,i}^2 & = \E\big[\big(M_{n,i}^{(d)}\big)^2|\mathcal{F}_{n,i}\big]\\
					& = \frac{t_1^2}{4} + \frac{t_2^2}{4} + 2~t_1t_2~\textnormal{Cov}\big(\mathrm{I}\{S_{\pi(i)}\not = S_{\pi(i+1)}\} \mathrm{I}\{\tilde{S}_{\tilde{\pi}(i)}\not = \tilde{S}_{\tilde{\pi}(i+1)}\}\big)\\
					& = \frac{t_1^2}{4} + \frac{t_2^2}{4} + 2~t_1t_2~\textnormal{Cov}\big(\mathrm{I}\{S_{\pi(1)}\not = S_{\pi(2)}\}, \mathrm{I}\{\tilde{S}_{\tilde{\pi}(1)}\not = \tilde{S}_{\tilde{\pi}(2)}\}\big) ~~~~~\textnormal{(by Theorem \ref{thm:sign-rank-dist} and \ref{thm:modified-rank-tests})}
				\end{align*}
				
				Now using the same arguments as in Theorem \ref{thm:large-dist-runs} and applying martingale CLT, we get
				\begin{align*}
					\frac{\sum_{i=1}^n M_{n,i}^{(d)}-\frac{n+1}{2}}{\sqrt{n}} \stackrel{D}{\rightarrow} N(0, \sigma^2) 
				\end{align*}
				as $n$ and $d$ diverges to infinity, where $\sigma^2 = \frac{t_1^2}{4}+\frac{t_2^2}{4} + 2t_1t_2 \big(\sigma_r^2-\frac{1}{4}\big)$. Therefore, applying Cramer-Wold device and continuous mapping theorem, we get
				\begin{align*}
					\frac{1}{\sqrt{n}} \left\{\min(T_R,\tilde{T}_R)-\frac{n+1}{2}\right\} \stackrel{D}{\rightarrow} \min\{Z_1,Z_2\},
				\end{align*} 
				where $(Z_1,Z_2)\sim N(\bm 0,\bm \Sigma_R)$ with $\bm \Sigma_R$ having diagonal elements as $\frac{1}{4}$ and off-diagonal element $\big(\sigma_r^2-\frac{1}{4}\big)$. This completes part (b) of the theorem.
			\end{proof}
			
			\begin{proof}[\bf Proof of Theorem \ref{thm:HDHSS-modified-tests-alt}] If the underlying distribution is spherically symmetric, as $n$ and $d$ diverge to infinity, $T_S^M/n=\max\{T_S,\tilde{T}_S\}/n \stackrel{P}{\rightarrow}1/2$ and $T_R^M/n=\min\{T_R,\tilde{T}_R\}/n \stackrel{P}{\rightarrow}1/2$ (follows from Theorem
				 \ref{thm:HDHSS-modified-tests-null}). So, for these modified tests,
				 all cut-offs also converge to 0.5.
				Now, from Theorems~\ref{thm:limit-alt-convergence} and \ref{thm:large-dist-runs-alt}, we have $|T_S-E(T_S)|/n \stackrel{P}{\rightarrow} 0$ and 
				$|T_R-E(T_R)|/n \stackrel{P}{\rightarrow} 0$ as $n,d$ diverge to infinity. Following the same idea, one can prove this property for ${\widetilde T}_S$ and ${\widetilde T}_R$ as well.
				
				{(a)} Therefore, when $\liminf_{n,d\rightarrow\infty} \E[T_S/n]>0.5$ or $\liminf_{n,d\rightarrow\infty} \E[\tilde{T}_S/n]>0.5$, $T_S^M$ takes value bigger than $0.5$ with probability converging to one.  This implies that the power of the modified sign test based on $T_S^M$ converges to one as $n$ and $d$ diverge to infinity.
				
				{(b)} Similarly, as $d$ and $n$ grows to infinity, we have the consistency of modified runs test based on $T_R^M$ when $\limsup_{n,d\rightarrow\infty} \E[T_R/n]<0.5$ or $\limsup_{n,d\rightarrow\infty} \E[\tilde{T}_R/n]<0.5$.
			\end{proof}

			\begin{proof}[\bf Proof of Lemma \ref{lemma:key-lemma}] Since
				$\bm X_1,\bm X_2$ are symmetric about $\muvec$, the characteristic function of $\bm X_1$ is of the form $\phi(\bm t) = \exp\{i\langle \bm t, \muvec \rangle\} g(\bm t)$, where $g(.)$ is some real-valued function with $g(\bm t) = g(-\bm t)$ for all $\bm t$. Note that if $\bm X_1,\bm X_2$ are spherically symmetric about $\muvec$, then it is trivial to show that $\bm X_1-\bm X_2$ is spherically symmetric about zero. Therefore, we only prove the if part.
				
				If $\bm X_1-\bm X_2$ is spherically symmetric about zero, then its characteristic function is of the form $f(\|\bm t\|)$ where $f(.)$ is some real-valued function. Also, note that
				\begin{align*}
					\phi_{\bm X_1-\bm X_2}(\bm t) = \phi_{\bm X_1}(\bm t) \phi_{-\bm X_2}(\bm t) = \phi_{\bm X_1}(\bm t)\phi_{\bm X_1}(-\bm t) = g^2(\bm t).
				\end{align*}
				Hence, $f(.)$ is non-negative and $g^2(\bm t) = f(\|\bm t\|)~\forall \bm t\in \R^d$. Therefore, 
				$\phi_{\bm X_1}(\bm t) = \exp\{i\langle t,\muvec \rangle\} h(\|\bm t\|),$
				where $|h(\|\bm t\|)|=f^{1/2}(\|\bm t\|)$. 
				This gives us the desired result.
			\end{proof}
			
			\section{Appendix B: Some additional mathematical details}
			
			\subsection{Some auxiliary results are their proofs}
			
			\begin{lemma}
				\label{prop:concentration-1}
				If $\bm X_1$ and $\bm X_2$ are two independent realizations of a $d$-dimensional random vector $\bm X \sim \Pr$. and $\bm X_1^\prime = \|\bm X_1\| \bm U_1,\bm X_2^\prime = \|\bm X_2\|\bm U_2$ (where $\bm U_1,\bm U_2\stackrel{i.i.d.}{\sim}$ Unif$(S^{d-1})$) are their respective spherically symmetric variants, then $\langle \bm X_1, \bm X_2^\prime\rangle \stackrel{D}{=} \langle \bm X_1^\prime, \bm X_2^\prime\rangle$.
			\end{lemma}
			
			\begin{proof}[{\bf Proof}]
				We shall prove this result using the fact that for any $\bm a\in \R^{d}$, $\bm a^\top\bm U_i \stackrel{D}{=}\|\bm a\| U_{i,1} ~(i=1,2)$, where $U_{i,1}$ is the first component of ${\bm U}_i$. Now, note that for any $t\in \R$
				\begin{align*}
					& \E[\exp\{it \langle \bm X_1, \bm X_2^\prime\rangle\}] = \E[\exp\{i t \|\bm X_2\| \langle \bm X_1, \bm U_2\rangle\}] = \E[\exp\{i t \|\bm X_2\| \|\bm X_1\| U_{2,1}\}] ~~\text{ and}\\
					&\E[\exp\{it \langle \bm X_1^\prime, \bm X_2^\prime\rangle\}] = \E[\exp\{i t \|\bm X_1\|\|\bm X_2\| \langle \bm U_1, \bm U_2\rangle\}] = \E[\exp\{i t \|\bm X_1\|\|\bm X_2\| U_{2,1}\}].
				\end{align*}
				
				The equality of these two characteristic functions proves the result.
			\end{proof}

			\begin{lemma}\label{lem:expectations}
				Let $\bm X_1,\bm X_2$ be independent copies of $\bm X \sim\Pr$, which has mean zero and variance covariance matrix $\bm \Sigma$. Let $\bm X_1^\prime$ and $\bm X_2^\prime$ be the spherically symmetric variants of $\bm X_1$ and $\bm X_2$, respectively. Then,
				\begin{align*}
					\E\left\{\frac{1}{d} (\bm X_1^\top\bm X_2)^2\right\} = \frac{1}{d} \textnormal{Tr}(\bm\Sigma^2)\text{ and } \E\left\{\frac{1}{d} (\bm X_1^\top\bm X_2^\prime)^2\right\} = \E\left\{\frac{1}{d} (\bm X_1^{\prime\top}\bm X_2^\prime)^2\right\} = \left(\frac{1}{d} \textnormal{Tr}(\bm \Sigma)\right)^2. 
				\end{align*}
			\end{lemma}
			
			\begin{proof}[{\bf Proof}]
				Note that $(\bm X_1^\top \bm X_2)^2 = \textnormal{Tr}(\bm X_2\bm X_2^\top \bm X_1\bm X_1^\top)$, and hence we have
				\begin{align*}
					\E\left\{\frac{1}{d} (\bm X_1^\top \bm X_2^\top)^2\right\} = \E\left\{\frac{1}{d} \textnormal{Tr}(\bm X_2\bm X_2^\top \bm X_1\bm X_1^\top)\right\} = \frac{1}{d} \textnormal{Tr}\big(\E\{\bm X_2\bm X_2^\top\} \E\{\bm X_1\bm X_1^\top\}\big) = \frac{1}{d} \textnormal{Tr}(\bm \Sigma^2).
				\end{align*}
				
				Now, note that $(\bm X_1^\top \bm X_2^\prime)^2 = \|\bm X_2\|^2 (\bm X_1^\top \bm U_2)^2 =  \|\bm X_2\|^2 \textnormal{Tr}(\bm U_2\bm U_2^\top \bm X_1\bm X_1^\top)$ where $\bm U_2\sim $ Unif$(S^{d-1})$ independent of $\bm X_1$ and $\bm X_2$. It is also important to note that $\E\{\bm U_2 \bm U_2^\top\} = \frac{1}{d} \mathrm {\bf I}_d$, where $\mathrm {\bf I}_d$ is the $d\times d$ identity matrix, and $\E\{\|\bm X_2\|^2\} = \sum_{i=1}^d \E\{X_{2i}^2\} = \textnormal{Tr}(\bm \Sigma)$. Then, taking expectations, we get
				\begin{align*}
					\vspace{-0.1in}
					\E\left\{\frac{1}{d} (\bm X_1^\top \bm X_2^\prime)^2\right\} = \E\|\bm X_2\|^2 \E\left\{\frac{1}{d} \textnormal{Tr}(\bm U_2\bm U_2^\top \bm X_1\bm X_1^\top)\right\}
					& = \textnormal{Tr}(\bm\Sigma)~ \frac{1}{d} \textnormal{Tr}\big(\E\{\bm U_2\bm U_2^\top\} \E\{\bm X_1\bm X_1^\top\}\big)\\
					& = \textnormal{Tr}(\bm\Sigma)~ \frac{1}{d} \textnormal{Tr}( \frac{1}{d} \bm {\mathrm I}_d \bm\Sigma) = \left(\frac{1}{d} \textnormal{Tr}(\bm\Sigma)\right)^2.
				\end{align*}
				Now, the proof follows from the fact that $\bm X_1^\top \bm X_2^\prime \stackrel{D}{=} \bm {X_1^\prime}^\top \bm X_2^\prime$.
				(see Lemma \ref{prop:concentration-1}).
			\end{proof}

				
				\begin{lemma}
					Let $\bm  X_1,\ldots, \bm X_n$ be independent copies of  $\bm X \sim \Pr$ and $\bm X_1^\prime,\ldots, \bm X_n^\prime$ be their respective spherically symmetric variants. For $s\in\{0,1\}^n$ define $\bm Y_{s,i} = s_i \bm X_i + (1-s_i)\bm X_i^\prime$ for each $i=1,\ldots,n$. Then, for the sign statistic $T_S$ (defined in Section \ref{sec:definition}) based on $\bm X_1,\ldots, \bm X_n$, we have 
					\begin{align*}\Big|\frac{1}{n}E(T_S) - \P\left[\theta(\bm Y_{S_1,1},\bm X_2)+\theta(\bm X_2,\bm Y_{S_3,3})\leq \theta(\bm Y_{S_1,1},\bm X_2^\prime)+\theta(\bm X_2^\prime,\bm Y_{S_3,3})\right]\Big| \rightarrow 0 \mbox { as } n,d \rightarrow \infty,
					\end{align*}
					where $S_1,S_3$ are i.i.d. Unif$(\{0,1\})$.
					\label{lemma:sign-stat-convergence}
				\end{lemma}
				
				\begin{proof}[{\bf Proof}]
					Recall the vectors of string signs  $\bm S = (S_1,\ldots, S_n)$, string ranks  $\bm R = (R_1,\ldots, R_n)$ and anti-ranks  $\bm \Pi = (\pi_1,\ldots,\pi_n)$ as defined in Section  \ref{sec:definition}. Now,
					\begin{align}
						\label{eq:signs-expectations}
						&\E\left[\frac{T_S}{n}\right] = \E\left[\frac{1}{n}\sum_{i=1}^n S_i\right] = \E\left[\frac{1}{n}\sum_{i=1}^n S_{\pi_i}\right]= \frac{1}{n}\Big(\P[S_{\pi_1} = 1] + \sum_{i=2}^{n-1} \P[S_{\pi_i}=1] + \P[\pi_n = 1]\Big)
					\end{align}
				    From the proof of Theorem \ref{thm:sign-rank-dist} we have,
					\begin{align*}
						\P[S_{\pi_1}=1\mid {\bm S}_{-\pi_1}] & = \P\big[\theta(\bm X_{\pi_1},\bm Y_{S_{\pi_2},\pi_2})\leq \theta(\bm X_{\pi_1}^\prime,\bm Y_{S_{\pi_2},\pi_2})\big| S_{\pi_2}\big],\\
						\P[S_{\pi_k}=1\mid {\bm S}_{-\pi_k}] & = \P\big[\theta(\bm Y_{S_{\pi_{k-1}},\pi_{k-1}},\bm X_{\pi_k})+\theta(\bm X_{\pi_k},\bm Y_{S_{\pi_{k+1}},\pi_{k+1}})\\
						& \hspace{0.5in}\leq \theta(\bm Y_{S_{\pi_{k-1}},\pi_{k-1}},\bm X_{\pi_k}^\prime)+\theta(\bm X_{\pi_k}^\prime,\bm Y_{S_{\pi_{k+1}},\pi_{k+1}})\big| S_{\pi_{k-1}},S_{\pi_{k+1}}\big],\\
						\P[S_{\pi_n}=1\mid{\bm S}_{-\pi_n}] & = \P\big[\theta(\bm Y_{S_{\pi_{n-1}},\pi_{n-1}},\bm X_{\pi_n})\leq \theta(\bm Y_{S_{\pi_{n-1}},\pi_{n-1}},\bm X_{\pi_n}^\prime)\big| S_{\pi_{n-1}}].
					\end{align*}
					Note that the distribution of $\bm \Pi$ does not depend on the distribution $\Pr$ and it follows Unif$(\mathcal{S}_n)$. Therefore, 
					\begin{align*}
						 \P[S_{\pi_1}=1]
						& = \E\left[\P[S_{\pi_1}=1\mid S_{-\pi_1}]\right]\\
						& = \E\left[\P[\theta(\bm X_{\pi_1}, \bm Y_{S_{\pi_2}, \pi_2}) \leq \theta(\bm X_{\pi_1}^\prime, \bm Y_{S_{\pi_2}, \pi_2})\mid S_{\pi_2}]\right]\\
						& = \left[\P[S_{\pi_2}=1] ~\P[\theta(\bm X_{\pi_1}, \bm Y_{1, \pi_2})\leq \theta(\bm X_{\pi_1}^\prime, \bm Y_{1, \pi_2})] + \P[S_{\pi_2}=0] ~\P[\theta(\bm X_{\pi_1}, \bm Y_{0, \pi_2})\leq \theta(\bm X_{\pi_1}^\prime, \bm Y_{0, \pi_2})]\right] \\
						& = \P[S_{2}=1] ~\P[\theta(\bm X_{1}, \bm Y_{1, 2})\leq \theta(\bm X_{1}^\prime, \bm Y_{1, 2})] + \P[S_{2}=0] ~\P[\theta(\bm X_{1}, \bm Y_{0,2})\leq \theta(\bm X_{1}^\prime, \bm Y_{0, 2})] \\
						& = \P[\theta(\bm X_{1}, \bm Y_{S_2, 2})\leq \theta(\bm X_{1}^\prime, \bm Y_{S_2, 2})]
					\end{align*}
					
                    Similarly, we can also show that					\begin{align}\label{eq:sign-probs}
					    \P[S_{\pi_k}=1] & = \P\big[\theta(\bm Y_{S_{1},1)},\bm X_{2})+\theta(\bm X_{2},\bm Y_{S_{3},3})\leq \theta(\bm Y_{S_{1},1},\bm X_{2}^\prime)+\theta(\bm X_{2}^\prime,\bm Y_{S_{3},3})\big], \textnormal{ for $k = 2,\ldots, n-1$}\nonumber\\
						\P[S_{n}=1] & = \P\big[\theta(\bm Y_{S_{n-1},n-1},\bm X_{n})\leq \theta(\bm Y_{S_{n-1},n-1},\bm X_{n}^\prime)].
					\end{align}
					Therefore, combining \eqref{eq:signs-expectations} and \eqref{eq:sign-probs} we get,
					\begin{align*}
						& \Big|\E\left[\frac{T_S}{n}\right]- \P\big[\theta(\bm Y_{S_{1},1},\bm X_{2})+\theta(\bm X_{2},\bm Y_{S_{3},3})\leq \theta(\bm Y_{S_{1},1},\bm X_{2}^\prime)+\theta(\bm X_{2}^\prime,\bm Y_{S_{3},3})\big] \Big|\\
						& = \Big|\frac{1}{n}\big( \P[S_{\pi_1}=1] + \P[S_{\pi_n}=1] \big)\Big| \rightarrow 0,
					\end{align*}
					as $n$ and $d$ diverge to infinity simultaneously. This completes the proof.
				\end{proof}

				\subsection{Pitman efficiency of the linear rank statistic}
				
				We know that the linear rank tests for univariate data are Pitman efficient \citep[see][]{hajekranktest}. So, one may wonder whether the linear rank tests defined in Section \ref{sec:tests} have the same property. In the following theorem, we address this issue for finite dimensional data. 
				
				{\begin{thm}\label{thm:pitman-efficiency}
						Let $\Pr$ be a spherically symmetric distribution  and $\mathrm Q$ be a non-spherical distribution, which are mutually absolutely continuous and have densities $p(.)$ and $q(.)$ such that $\int \left|q(\bm x)/p(\bm x)-1\right|^3 p(\bm x)\mathrm d\bm x$ is finite. For any  $\delta>0$, consider the contamination model
						$$F_n = \left(1-\frac{\delta}{\sqrt{n}}\right) \Pr + \frac{\delta}{\sqrt{n}}\mathrm Q,$$
						as a local asymptotically normal contiguous alternative \citep[see Proposition 3.1][]{banerjee2024consistent}. 
						Let $\bm X_1,\bm X_2,\ldots,\bm X_n$ be $n$ independent realizations of $\bm X \sim F_n$ and $\bm X_i^\prime = \|\bm X_i\|\bm U_i$ for $i=1,2,\ldots,n$, where $\bm U_1,\bm U_2,\ldots, \bm U_n$ are independent Unif$(S^{d-1})$ random variables. Consider a sequence of uniformly bounded scores $\{a(i)\}_{1\leq i\leq n}$ with the following properties
						\begin{align}\label{eq:score-conditions-pitman}
							\frac{1}{n}\sum_{i=1}^n a^2(i)\rightarrow \sigma^2~ (\sigma^2>0), ~~~~\frac{1}{n}\sum_{i=1}^n a(i)\rightarrow \tau~~~~  ~~\textnormal{and}~~~~ \max_{1\leq i\leq n}\frac{a^2(i)}{\sum_{i=1}^n a^2(i)}\rightarrow 0 \mbox{ ~~~~~~as } n\rightarrow\infty.
						\end{align}
						Then, as $n\rightarrow\infty$ we have
						$$\frac{T_{LR} - \frac{1}{2}\sum_{i=1}^n a(i)}{\sqrt{\sum_{i=1}^n a^2(i)}} \stackrel{D}{\rightarrow}\mathcal{N}\left(\delta\frac{\tau}{\sigma}\Big(p-\frac{1}{2}\Big),\frac{1}{4}\right),$$
						where
						$	p= \P\left[\theta(\bm X_{1},\bm V)+\theta(\bm V,\bm X_{2}) < \theta(\bm X_{1},\bm V^\prime)+\theta(\bm V^\prime,\bm X_{2})\right]
						$ 
						for $\bm X_{1},\bm X_{2}\stackrel{i.i.d.}{\sim} \mathrm P$, $\bm V\sim \mathrm Q$ and $\bm V^\prime$ is a spherically symmetric variant of $\bm V$.    
				\end{thm}}
				
				\begin{proof}[\bf Proof of Theorem \ref{thm:pitman-efficiency}]
					By Proposition 3.1 in \cite{banerjee2024consistent} we have
					\begin{align*}
						\log\bigg\{\prod_{i=1}^n\Big(1+\frac{\delta}{\sqrt{n}}\Big\{\frac{q({\bf X}_i)}{p({\bf X}_i)}-1\Big\}\Big)\bigg\}=\frac{\delta}{\sqrt{n}}\sum_{i=1}^n \bigg\{\frac{q({\bf X}_i)}{p({\bf X}_i)}-1\bigg\}-\frac{\delta^2}{2}\E\bigg\{\frac{q({\bf X}_1)}{p({\bf X}_1)}-1\bigg\}^2 + o_P(1).
					\end{align*} 
					Note that by Jensen's inequality, the finiteness of $\int |q(\bm u)/p(\bm u)-1|^2 p(\bm u)\mathrm d\bm u$ follows from the finiteness of $\int |q(\bm u)/p(\bm u)-1|^3 p(\bm u)\mathrm d\bm u$. Now, let us look at the triangular array 
					$$\left\{W_{ni} = t_1 \frac{a(i)(S_{\pi_i}-\frac{1}{2})}{\sqrt{\sum_{i=1}^n a^2(i)}} + t_2 \frac{\delta}{\sqrt{n}} \bigg\{\frac{q({\bf X}_{\pi_i})}{p({\bf X}_{\pi_i})}-1\bigg\}\right\}_{1\leq i\leq n, n\geq 1}.$$
					Under $H_0$, this is an array of i.i.d. random variables with finite third moments. Now under the assumptions $\frac{1}{n}\sum_{i=1}^n a^2(i)\rightarrow \sigma^2$ and $\frac{1}{n}\sum_{i=1}^n a(i) \rightarrow \tau$, we have
					\begin{align}\label{eq:sn-convergence}
						s_n^2  = \sum_{i=1}^n \var(W_{ni}) 
						& = \frac{t_1^2}{4} + t_2^2\delta^2 ~\E\left\{\frac{q(\bm X_1)}{p(\bm X_1)}-1\right\}^2 + 2 t_1t_2~ \frac{\delta}{\sqrt{n}}\sum_{i=1}^n \E\left\{\frac{a(i)(S_{\pi_i}-\frac{1}{2})}{\sqrt{\sum_{i=1}^n a^2(i)}}\Big\{\frac{q({\bf X}_{\pi_i})}{p({\bf X}_{\pi_i})}-1\Big\}\right\}\nonumber\\
						& = \frac{t_1^2}{4} + t_2^2 \delta^2~ \E\left\{\frac{q(\bm X_1)}{p(\bm X_1)}-1\right\}^2 + 2 t_1t_2~ \frac{\delta}{n}\sum_{i=1}^n \E\left\{\frac{a(i)(S_{\pi_i}-\frac{1}{2})}{\sqrt{\frac{1}{n}\sum_{i=1}^n a^2(i)}}\Big\{\frac{q({\bf X}_{\pi_i})}{p({\bf X}_{\pi_i})}-1\Big\}\right\}\nonumber\\
						& = \frac{t_1^2}{4} + t_2^2 \delta^2~ \E\left\{\frac{q(\bm X_1)}{p(\bm X_1)}-1\right\}^2 + 2 t_1t_2 ~\frac{\delta}{n}\sum_{i=2}^{n-1} \E\left\{\frac{a(i)(S_{\pi_i})}{\sigma}\Big\{\frac{q({\bf X}_{\pi_i})}{p({\bf X}_{\pi_i})}-1\Big\}\right\} + o(1)\nonumber\\
						& = \frac{t_1^2}{4} + t_2^2 \delta^2 ~\E\left\{\frac{q(\bm X_1)}{p(\bm X_1)}-1\right\}^2 + 2 t_1t_2 ~\frac{\delta}{n}\sum_{i=2}^{n-1} \frac{a(i)}{\sigma} \E\left\{S_{\pi_2}\Big\{\frac{q({\bf X}_{\pi_2})}{p({\bf X}_{\pi_2})}-1\Big\}\right\} + o(1)\nonumber\\
						& = \frac{t_1^2}{4} + t_2^2 \delta^2~ \E\left\{\frac{q(\bm X_1)}{p(\bm X_1)}-1\right\}^2 + 2 t_1t_2~ \delta\frac{\tau}{\sigma} \E\left\{S_{\pi_2}\Big\{\frac{q({\bf X}_{\pi_2})}{p({\bf X}_{\pi_2})}-1\Big\}\right\} + o(1).
					\end{align}
					Now note that
					$	\mathrm I\Big\{S_{\pi_2}=1\Big\}  = \mathrm I\Big\{ \theta(\bm Y_{S_{\pi_1},\pi_1},\bm X_{\pi_2})+\theta(\bm X_{\pi_2},\bm Y_{S_{\pi_3},\pi_3})
						\leq \theta(\bm Y_{S_{\pi_1},\pi_1},\bm X^\prime_{\pi_2})+\theta(\bm X^\prime_{\pi_2},\bm Y_{S_{\pi_3},\pi_3}) \Big\}
					$
					 where $\bm Y_{S,i} = S\bm X_i + (1-S)\bm X_i^\prime$ ($i = \pi_1,\pi_2,\pi_3$). Clearly, $\E[S_{\pi_2}]={1}/{2}$ under spherical symmetry and 
					\begin{align*}
						& \E\left\{S_{\pi_2}\Big\{\frac{q({\bf X}_{\pi_2})}{p({\bf X}_{\pi_2})}\Big\}\Big| S_{\pi_1},S_{\pi_3},\pi_1,\pi_3\right\}\\
						& = \P\Big[\theta(\bm Y_{S_{\pi_1},\pi_1},\bm V)+\theta(\bm V,\bm Y_{S_{\pi_3},\pi_3})\leq \theta(\bm Y_{S_{\pi_1},\pi_1},\bm V^\prime)+\theta(\bm V^\prime,\bm Y_{S_{\pi_3},\pi_3})\Big| S_{\pi_1},S_{\pi_3},\pi_1,\pi_3\Big],
					\end{align*}
					where $\bm V\sim \mathrm Q$ and $\bm V^\prime = \|\bm X\|\bm U$ where $\bm U\sim$ Unif$(S^{d-1})$ independent of $\bm V$. Here $\bm Y_{S,i}$ are as defined above. Now, under $H_0$, $S_{\pi_1},S_{\pi_3}$ are i.i.d.  Unif$(\{0,1\})$ and $\pi_1,\pi_3$ are simple random samples without replacement from $\{1,2,\ldots,n\}$ independent of $\bm S$. So, taking expectations with respect to $S_{\pi_1},S_{\pi_3}$ and $\pi_1,\pi_3$ we get
					\begin{align}\label{eq:exact-expr-off-diagonal}
						& \E\left\{S_{\pi_2}\Big\{\frac{q({\bf X}_{\pi_2})}{p({\bf X}_{\pi_2})}\Big\}\right\}\nonumber\\
						& = \sum_{S_1,S_2\in \{0,1\}} \sum_{\pi_1,\pi_3 \in \mathcal{S}_n} \frac{1}{2^2} \frac{1}{n(n-1)} \P\Big[\theta(\bm Y_{S_{\pi_1},\pi_1},\bm V)+\theta(\bm V,\bm Y_{S_{\pi_3},\pi_3})\leq \theta(\bm Y_{S_{\pi_1},\pi_1},\bm V^\prime)+\theta(\bm V^\prime,\bm Y_{S_{\pi_3},\pi_3})\Big]\nonumber\\
						& = \frac{1}{4} \sum_{S_1,S_2\in \{0,1\}} \P\Big[\theta(\bm Y_{S_1,1},\bm V)+\theta(\bm V,\bm Y_{S_2,2})\leq \theta(\bm Y_{S_{1},1},\bm V^\prime)+\theta(\bm V^\prime,\bm Y_{S_2,2})\Big] = p, ~~~~~~\mbox{(as $n\rightarrow\infty$)}.
					\end{align}
					 However, since $\bm X_1$ and $\bm X_2$ are spherically symmetric under $H_0$, the probabilities in the right hand side of \eqref{eq:exact-expr-off-diagonal} are all equal. Hence, the last equality follows. Therefore, under $H_0$, 
					$$s_n^2 \rightarrow\frac{t_1^2}{4} + t_2^2 \delta^2~ \E\left\{\frac{q(\bm X_1)}{p(\bm X_1)}-1\right\}^2 + 2 t_1t_2\delta \frac{\tau}{\sigma} \big(p-\frac{1}{2}\big) \text{ as } n \rightarrow \infty.$$
					 Also, note that
					\begin{align}\label{eq:Lyapunov-pitman}
						\frac{1}{s_n^3} \sum_{i=1}^n \E\left[\Big|W_{ni}- \E[W_{ni}]\Big|^3\right] & = 	\frac{1}{s_n^3} \sum_{i=1}^n \E\left[\Big|t_1\Big(\frac{a(i)(S_{\pi_i}-\frac{1}{2})}{\sqrt{\sum_{i=1}^n a^2(i)}}\Big) + t_2 \Big(\frac{\delta}{\sqrt{n}} \big\{\frac{q({\bf X}_{\pi_i})}{p({\bf X}_{\pi_i})}-1\big\}\Big) \Big|^3 \right]\nonumber\\
						&\leq 2^2 |t_1|^3 \frac{1}{s_n^3} \sum_{i=1}^n \E\left[\Big|\Big(\frac{a(i)(S_{\pi_i}-\frac{1}{2})}{\sqrt{\sum_{i=1}^n a^2(i)}}\Big)\Big|^3\right] + 2^2 |t_2|^3 \frac{1}{s_n^3} \sum_{i=1}^n \E\left[\Big|\Big(\frac{\delta}{\sqrt{n}} \big\{\frac{q({\bf X}_i)}{p({\bf X}_i)}-1\big\}\Big) \Big|^3 \right],
					\end{align}
					where the last inequality follows using $|a+b|^p\leq 2^{p-1}(|a|^p + |b|^p)$ with $p = 3$. The first term in \eqref{eq:Lyapunov-pitman} goes to zero using \eqref{eq:suff-derivation}. The second term in \eqref{eq:Lyapunov-pitman} is of order $\mathcal{O}(n^{-1/2})$ by assumption $\int \big|q(\bm x)/p(\bm x)-1\big|^3\mathrm d \bm x<\infty$. Therefore, using Lyapunov's CLT we also have 
						$
						\sum_{i=1}^n W_{ni} \stackrel{D}{\rightarrow} N(0,\sigma^2),
					$
					where $\sigma^2 = \frac{t_1^2}{4} + t_2^2 \delta^2~ \E\left\{\frac{q(\bm X_1)}{p(\bm X_1)}-1\right\}^2 + 2 t_1t_2\delta \frac{\tau}{\sigma} \big(p-\frac{1}{2}\big)$. Now applying Cramer-Wold device we can conclude
					\begin{align*}
						\Big(\frac{\sum_{i=1}^n a(i)(S_{\pi_i}-\frac{1}{2})}{\sqrt{\sum_{i=1}^n a^2(i)}}, \frac{\delta}{\sqrt{n}}\sum_{i=1}^n \Big\{\frac{q(\bm X_i)}{p(\bm X_i)}-1\Big\}\Big)\stackrel{D}{\rightarrow} \mathcal{N}(0,\sigmat),
					\end{align*}
					where $\sigmat$ is a $2\times 2$ matrix with diagonal entries $1/4$ and $ \delta^2~ \E\left\{\frac{q(\bm X_1)}{p(\bm X_1)}-1\right\}^2$, and off-diagonal entry $\delta \frac{\tau}{\sigma} \big(p-\frac{1}{2}\big)$, respectively. Now using Le Cam's third lemma \citep[see][]{van2000asymptotic} we get 
					\begin{align*}
						\frac{\sum_{i=1}^n a(i)(S_{\pi_i}-\frac{1}{2})}{\sqrt{\sum_{i=1}^n a^2(i)}} \stackrel{D}{\rightarrow} \mathcal{N}\Big(\delta \frac{\tau}{\sigma} \Big(p-\frac{1}{2}\Big), \frac{1}{4}\Big)
					\end{align*}
					as $n$ diverges to infinity. This gives us our desired result.
				\end{proof}
				
				Theorem \ref{thm:pitman-efficiency} establishes that for a sequence of contiguous alternatives of the form $\{F_n:~n\ge 1\}$ for which with $p\not= 0.5$, the limiting distribution of the linear rank test introduced in Section \ref{sec:definition} is a non-centered normal distribution. Therefore, if the score functions satisfy assumption \ref{eq:score-conditions-pitman} and $p\not=0.5$, the corresponding tests are Pitman efficient. Several distributions satisfy this assumption (see Section \ref{sec:empirical-analysis}). 

			\end{document}